\documentclass[a4paper, 10pt, twoside, notitlepage]{amsart}

\usepackage[utf8]{inputenc}
\usepackage{color}
\usepackage{amsmath}
\usepackage{amssymb}
\usepackage{amsthm}
\usepackage{geometry}
\usepackage{graphicx}
\usepackage{esint}
\usepackage[colorlinks=true,linkcolor=blue,pdfencoding=auto,psdextra,unicode]{hyperref}
\usepackage{tikz, tikz-3dplot}
\usetikzlibrary{patterns,arrows}
\usepackage[capitalize]{cleveref}
\usepackage{todonotes}
\usepackage{comment}
\usepackage[english]{babel}
\usepackage{lmodern}
\usepackage{microtype}
\usepackage{subcaption}
\usepackage[version=4,arrows=pgf-filled,
textfontname=sffamily,
mathfontname=mathsf]{mhchem}

\usepackage[overload]{textcase}

\theoremstyle{plain}
\newtheorem{thm}{Theorem}
\newtheorem{prop}{Proposition}[section]
\newtheorem{lem}[prop]{Lemma}
\newtheorem{cor}[prop]{Corollary}

\newtheorem{defi}[prop]{Definition}
\newtheorem{rmk}[prop]{Remark}

\crefname{thm}{Theorem}{Theorems}
\crefname{lem}{Lemma}{Lemmata}
\crefname{cor}{Corollary}{Corollaries}
\crefname{equation}{}{}
\crefname{figure}{Figure}{Figures}
\crefname{enumi}{}{}
\crefname{case}{case}{cases}

\newcommand {\R} {\mathbb{R}} \newcommand {\Z} {\mathbb{Z}}
\newcommand {\T} {\mathbb{T}} \newcommand {\N} {\mathbb{N}}
\newcommand {\C} {\mathbb{C}} 
\newcommand {\p} {\partial}

\newcommand {\supp} {\textup{supp}}
\newcommand {\diam} {\textup{diam}}
\newcommand {\rank} {\textup{rank}}

\newcommand {\diag} {\textup{diag}}

\newcommand{\Per}{\textup{Per}}
\newcommand{\K}{\mathcal{K}}
\newenvironment{psmallmatrix}
    {\left(\begin{smallmatrix}}
    {\end{smallmatrix}\right)}

\renewcommand{\S}{\mathbb{S}}

\DeclareMathOperator{\di}{div}

\DeclareMathOperator {\dist} {dist}

\DeclareMathOperator{\Id} {Id}

\definecolor{Orange}{RGB}{230, 159, 0}
\definecolor{LightBlue}{RGB}{86, 180, 233}
\definecolor{Green}{RGB}{0, 158, 115}
\definecolor{Yellow}{RGB}{240, 228, 66}
\definecolor{Blue}{RGB}{0, 144, 178}
\definecolor{Red}{RGB}{213, 94, 0}
\definecolor{Purple}{RGB}{204, 121, 167}

\title[On surface energies in scaling laws for singular perturbation problems]{On surface energies in scaling laws for singular perturbation problems for martensitic phase transitions}

\author{Angkana Rüland}
\address{Institute for Applied Mathematics and Hausdorff Center for Mathematics, Endenicher Allee 60, University of Bonn, 53115 Bonn, Germany}
\email{rueland@uni-bonn.de}
\author{Camillo Tissot}
\address{Institute for Applied Mathematics, Endenicher Allee 60, University of Bonn, 53115 Bonn, Germany}
\email{camillo.tissot@uni-bonn.de}
\author{Antonio Tribuzio}
\address{Institute for Applied Mathematics, Endenicher Allee 60, University of Bonn, 53115 Bonn, Germany}
\email{tribuzio@iam.uni-bonn.de}
\author{Christian Zillinger}
\address{Department of Mathematics, Karlsruhe Institute of Technology, Englerstrasse. 2, 76131 Karlsruhe,
Germany}
\email{christian.zillinger@kit.edu}

\begin{document}
\begin{abstract}
The objective of this article is to compare different surface energies for multi-well singular perturbation problems associated with martensitic phase transformations involving higher order laminates. We deduce scaling laws in the singular perturbation parameter which are \emph{robust} in the choice of the surface energy (e.g., diffuse, sharp, an interpolation thereof or discrete). Furthermore, we show that these scaling laws do not require the presence of isotropic surface energies but that generically also \emph{highly anisotropic surface energies} yield the same scaling results. More precisely, the presence of essentially generic partial directional derivatives in the regularization terms suffices to produce the same scaling behaviour as in the isotropic setting. The only sensitive directional dependences are directly linked to the \emph{lamination directions} of the well structure -- and even for these only the ``inner-most'' lamination direction is of significance in determining the scaling law. In view of experimental applications, this shows that also for higher-order laminates, the precise structure of the surface energies -- which is often very difficult to determine experimentally -- does not have a crucial impact on the scaling behaviour of the investigated structures but only enters when considering finer properties.
\end{abstract}


\maketitle

\tableofcontents

\section{Introduction}

Surface energies play an important role in singular perturbation models for solid-solid phase transformations. Combined with elastic energies, they introduce a natural length scale into the models. Thus, the combination of elastic and surface energies provides important information on the different length scales present in the experimentally observed microstructures. However, from an experimental point of view, surface energies are notoriously difficult to measure and are often highly anisotropic. It is thus of particular significance to investigate the robustness of singular perturbation models with respect to different choices of surface energy regularizations. The purpose of this article is to prove that the scaling behaviour is robust with respect to a rather large class of modifications of the surface energies and that the resulting scaling laws do not depend on the fine structure of the singular perturbation term. In fact, only minimal requirements are necessary, which only depend on very basic information on the model, even if higher order laminates are involved.

We are particularly interested in martensitic phase transformations in shape-memory alloys. These materials are typically metal alloys such as CuAlNi or NiTi which undergo a first order, diffusionless, solid-solid phase transformation. In this transition symmetry is reduced from the high to the low temperature phase which gives rise to multiple energy wells and complex material behaviour.

In this article, we will consider simplified models without gauges for the formation of microstructures in these materials which are governed by energies of the following form \cite{B,M1}
\begin{align}
\label{eq:intro_energy1}
E_{\epsilon}(u):= E_{el}(u) + \epsilon E_{surf}(u).
\end{align}
Here $u: \Omega \subset \R^{d} \rightarrow \R^{d}$ denotes the deformation with respect to the reference configuration $\Omega$ and $E_{el}(u):= \int\limits_{\Omega} \dist^2(\nabla u, \mathcal{K}) dx$ models the elastic energy. Typically, the energy density is of multi-well nature with $\mathcal{K} \subset \R^{d \times d}$ a prescribed set in matrix space corresponding to the energy wells of the respective model. The main focus of this article is on the structure of the second energy contribution in \cref{eq:intro_energy1}, $E_{surf}(u)$, which models a surface energy. The parameter $\epsilon>0$ is material specific and typically small. Mathematically, the surface energy provides a higher order regularizing contribution which penalizes fine-scale oscillations in $\nabla u$. In what follows below, we will consider different models for $E_{surf}(u)$. Our objective is to prove that for our model class of martensitic phase transformations, the investigated scaling laws typically do not depend on the precise structure of the surface energies. On the contrary, relatively ``rough'' information suffices to produce equivalent scaling laws.
In order to illustrate this, we discuss different prototypical model classes for microstructures in shape-memory alloys.

\subsection{Sharp surface energies}
\label{sec:intro_L1}
We begin by considering sharp interface models. In particular, we focus on settings involving anisotropy. Here, we specify the set-up from \cref{eq:intro_energy1} as follows
\begin{align}
\label{eq:L1_energy}
E_{\epsilon}(u,\chi) := E_{el}(u,\chi) + \epsilon E_{surf}(\chi):=\int\limits_{\Omega} |\nabla u - \chi|^2 dx + \epsilon \sum\limits_{\nu \in \mathcal{N}} \|D_{\nu} \chi\|_{TV(\Omega)}.
\end{align}
For given $F \in \R^{d \times d}$, $u \in \mathcal{A}_F$ models the deformation, where
\begin{align} \label{eq:AdmissibleFunctionsL2}
\mathcal{A}_F := \{ v \in H^1(\Omega;\R^d): v(x) = Fx \text{ on } \p \Omega\},
\end{align}
while $\chi:\Omega \rightarrow \mathcal{K} \subset \R^{d \times d}$ represents the phase indicator.

As all the considered quantities are translation invariant, we could also consider a boundary condition of the form $u(x) = Fx +b$ on $\p \Omega$ for some $b \in \R^d$. For the sake of simplicity, we assume $b = 0$.

In our study below, the set $\mathcal{K}$ will represent the wells of phase
transformation models with a discrete set of minima for the energy density
(i.e., we do not consider typical gauges such as $SO(d)$ or $Skew(d)$
invariances in our model). The surface energy under consideration $ \|D_{\nu}
\chi\|_{TV(\Omega)}$ is of sharp-interface nature and highly anisotropic,
depending only on specified linearly independent directions $\nu \in
\mathcal{N}\subset \mathbb{S}^{d-1}$ with $\# \mathcal{N}\leq d$ (see
\cref{sec:prelim_directBV} for further discussion and definitions).
In what follows below, we will discuss minimal conditions on the choice of the directions $\nu \in  \mathcal{N}$
in order to ensure ``generic'' behaviour in the scaling laws -- thus proving their robustness in this choice of surface energy. In particular, in many instances it suffices that $\mathcal{N}$ consists of a single, non-degenerate direction $\mathcal{N}=\{\nu\}$.
We will investigate scaling laws for energies of the type \cref{eq:L1_energy} for various possible choices of the set $\mathcal{K}$ and discuss microstructures of different complexities.

\subsubsection{The two-well problem}
\label{sec:introL12}

We begin with an essentially scalar setting by considering $\mathcal{K} = \{A,B\}$ with $\mbox{rank}(A-B) = 1$. In this setting, the expected microstructure consists of a branched version of twinning \cite{KM1, KM2}. Such structures play an important role in austenite-martensite interfaces \cite{KM1, KM2, CO, CO1, TS21, TS21a}.
As already highlighted in the seminal works \cite{KM1, KM2}, in order to observe this phenomenon, it is not necessary to include all directional derivatives in the singular perturbation term. As expected from the experimental microstructure and the almost one-dimensional character of the problem, it suffices to regularize in the direction of oscillation. Hence, the scaling of the fully surface regularized and of the only in direction of oscillation regularized model behave analogously.
We recall this in the following proposition (cf. \cite{KM1, KM2}).

\begin{prop}
\label{prop:L1_2wells}
  Let $A,B \in \R^{d \times d}$ be such that $A-B = a \otimes e_1$ for some $a \in \R^d \setminus \{0\}$.
  Let $\Omega = (0,1)^d$ and $F_\alpha = \alpha A + (1-\alpha) B$ for some $\alpha \in (0,1)$.
  Let $\nu \in \S^{d-1}$ be such that $\nu \cdot e_1 \neq 0$.
  Consider for $\mathcal{N} = \{\nu\}$, $u \in \mathcal{A}_{F_\alpha}$, cf. \cref{eq:AdmissibleFunctionsL2}, and $\chi \in BV_\nu(\Omega;\{A,B\})$ the energy $E_{\epsilon}(u,\chi)$ as defined in \cref{eq:L1_energy}.
  There are constants $C =
  C(d,\alpha,|A-B|) > 0$ and $\epsilon_0 = \epsilon_0(d, \alpha, |A-B|, |\nu_1|) > 0$ such that for any $\epsilon \in (0,\epsilon_0)$
  \begin{align*}
     \inf_{\chi \in BV_\nu(\Omega;\{A,B\})}\inf_{u \in \mathcal{A}_{F_\alpha}}  E_\epsilon(u,\chi) \geq C^{-1} |\nu_1|^{\frac{2}{3}} \epsilon^{\frac{2}{3}}.
  \end{align*}
  Here $\nu_1$ denotes the $e_1$ component of the vector $\nu$.

    In two dimensions the matching upper bounds hold, i.e. for $d = 2$ and every $\epsilon \in (0,\epsilon_0)$ there are $u \in \mathcal{A}_{F_\alpha}$ and $\chi \in BV_\nu(\Omega;\{A,B\})$ such that
    \begin{align*}
        E_\epsilon(u,\chi) \leq C |\nu_1|^{\frac{2}{3}} \epsilon^{\frac{2}{3}}.
    \end{align*}

  Moreover (in $d$ dimensions), if $\nu \cdot e_1 =0$, then under the above assumptions for all $\epsilon >0$
  \begin{align*}
      \inf_{\chi \in BV_\nu(\Omega;\{A,B\})}\inf_{u \in \mathcal{A}_{F_{\alpha}}}  E_\epsilon(u,\chi) = 0.
  \end{align*}
\end{prop}

Here and below the space $BV_{\nu}(\Omega; \{A,B\})$ denotes an anisotropic version of the space BV in which BV regularity is only required in the direction $\nu \in \mathbb{S}^{d-1}$. We refer to \cref{sec:prelim_directBV} for the precise definition of it.

\begin{rmk}
    The matching upper bound can also be generalized to hold in higher dimensions.
    For the sake of simplicity we do not discuss this in this article; the construction for an isotropic surface penalization can, for instance, be found in \cite[Section 6.2]{RT23}.
\end{rmk}

\subsubsection{The three-well problem of Lorent}
\label{sec:introL13}

With the almost one-dimensional two-well problem in mind, we turn to models involving higher order laminates. Here a model problem is given by the three-well configuration of Lorent \cite{L01}. In this setting, we have (up to normalization)
\begin{align} \label{eq:Lorent_wells}
\mathcal{K}_3:= \{ A_1, A_2, A_3\} \mbox{ with }
A_1 = \begin{pmatrix}
    0 & 0 \\ 0 & 0
\end{pmatrix}, \
A_2 = \begin{pmatrix}
    1 & 0 \\ 0 & 0
\end{pmatrix}, \
A_3 = \begin{pmatrix}
    \frac{1}{2} & 0 \\ 0 & 1
\end{pmatrix}.
\end{align}
These three matrices are chosen such that $\mbox{rank}(A_1-A_2)=1$, but neither $A_1$ nor $A_2$ are rank-one connected with the well $A_3$. The lamination convex hull $\K_{3}^{lc}$ of $\mathcal{K}_3$ and, hence, the observable microstructure, consists of laminates up to second order (we recall notions such as lamination convexity in \cref{sec:prelim_lam}). The first order laminates $\mathcal{K}^1_3$ consist of convex combinations of the wells $A_1,A_2$ while the second order laminates $\mathcal{K}^2_3$ are obtained by a convex combination of the well $A_3$ with the auxiliary matrix $\begin{pmatrix} \frac{1}{2} & 0 \\ 0 & 0\end{pmatrix}$ (see \cref{fig:Lorent}). More precisely,
\begin{align}
\label{eq:laminates_order}
    \K_3^1 &:= \K_3^{(1)} \setminus \K_3 = \left\{ \begin{pmatrix} \alpha & 0 \\ 0 & 0 \end{pmatrix} : \alpha \in (0,1) \right\}, &
    \K_3^2 & := \K_3^{(2)} \setminus \K_3^{(1)}= \left\{ \begin{pmatrix} \frac{1}{2} & 0 \\ 0 & \alpha \end{pmatrix}: \alpha \in (0,1) \right\}.
\end{align}

\begin{figure}[t]
  \centering
  \begin{tikzpicture}[thick, scale = 1.5]
    \draw[Blue] (0,0) -- (1,0);
    \draw[Red] (0.5,0) -- (0.5,1);

    \fill (0,0) circle (0.05) node[below] {$A_1$};
    \fill (1,0) circle (0.05) node[below] {$A_2$};
    \fill (0.5,1) circle (0.05) node[right] {$A_3$};
  \end{tikzpicture}
  \caption{The Lorent three-well setting. The first order laminates are shown in blue, the second order laminates in orange.}
  \label{fig:Lorent}
\end{figure}
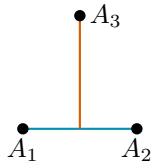

In \cite{RT23} in the setting with isotropic surface energies, the different scaling behaviour of microstructures with affine boundary conditions was deduced with the behaviour depending only on the complexity of the boundary data encoded in their lamination order.
Here we prove that this remains valid, if in \cref{eq:L1_energy} we consider generic regularization directions. Indeed, due to the one-dimensional structure of the lamination convex hull, it suffices to choose $\mathcal{N} = \{\nu\}$ such that $\nu \cdot e_1 \neq 0$ in order to recover the scaling from \cite{RT23} in which we considered the full gradient in the surface energy.
If, however, $\nu = \pm e_2$, then, the setting indeed changes and the scaling behaviour becomes that of a laminate of one order less.

\begin{thm}
\label{thm:L1_3wells}
Let $\K_3$ be given by \cref{eq:Lorent_wells}, $\Omega = (0,1)^2$, and $\nu \in \S^{1}$.
Consider the energy $E_\epsilon(u,\chi)$ as in \cref{eq:L1_energy} with $\mathcal{N} = \{\nu\}$, $u \in \mathcal{A}_F$, cf. \cref{eq:AdmissibleFunctionsL2}, for $F \in \K_3^{lc}\setminus \K_3$, and $\chi \in BV_\nu(\Omega;\K_3)$. Let $\mathcal{K}^1_3$ and $\mathcal{K}^2_3$ be as in \cref{eq:laminates_order}.
We then have the following scaling laws:
\begin{enumerate}
    \item\label[case]{itm:L1_3wells_i} \emph{First order laminates:} For $F \in \K_3^1$ there are constants $C = C(F) > 0$ and $\epsilon_0 = \epsilon_0(F,\nu)>0$ such that for any $\epsilon \in (0,\epsilon_0)$
    \begin{align*}
       C^{-1} |\nu_1|^{\frac{2}{3}} \epsilon^{\frac{2}{3}} \leq \inf_{\chi \in BV_\nu(\Omega;\K_3)} \inf_{u \in \mathcal{A}_F} E_\epsilon(u,\chi) \leq C |\nu_1|^{\frac{2}{3}} \epsilon^{\frac{2}{3}}.
    \end{align*}
    \item\label[case]{itm:L1_3wells_ii} \emph{Second order laminates:} For $F \in \K_3^2$ there are constants $C = C(F) > 0$ and $\epsilon_0 = \epsilon_0(F,\nu) > 0$ such that for any $\epsilon \in (0,\epsilon_0)$ we have
    \begin{align*}
        C^{-1} \left( |\nu_1|^{\frac{1}{2}} \epsilon^{\frac{1}{2}} + |\nu_2|^{\frac{2}{3}} \epsilon^{\frac{2}{3}} \right) \leq \inf_{\chi \in BV_\nu(\Omega;\K_3)} \inf_{u \in \mathcal{A}_F} E_\epsilon(u,\chi) \leq C \left( |\nu_1|^{\frac{1}{2}} \epsilon^{\frac{1}{2}} + |\nu_2|^{\frac{2}{3}} \epsilon^{\frac{2}{3}} \right).
    \end{align*}
\end{enumerate}
\end{thm}
\begin{rmk}
We highlight that the estimate from \cref{itm:L1_3wells_ii} for the second order laminates in fact includes matching bounds for both the non-degenerate case in which $\nu \cdot e_1 \neq 0$ and the degenerate case in which $\nu \cdot e_1 = 0$. Indeed, by a case distinction, for $F \in \K_3^2$, on the one hand, if $\nu \cdot e_1 \neq 0$ and if $\epsilon_0>0$ is sufficiently small we have that
    \begin{align*}
       C^{-1} |\nu_1|^{\frac{1}{2}} \epsilon^{\frac{1}{2}} \leq \inf_{\chi \in BV_\nu(\Omega;\K_3)} \inf_{u \in \mathcal{A}_F} E_\epsilon(u,\chi) \leq C |\nu_1|^{\frac{1}{2}} \epsilon^{\frac{1}{2}}.
    \end{align*}
    On the other hand, if $\nu \cdot e_1 = 0$, then
    \begin{align*}
        C^{-1} |\nu_2|^{\frac{2}{3}} \epsilon^{\frac{2}{3}} \leq \inf_{\chi \in BV_\nu(\Omega;\K_3)} \inf_{u \in \mathcal{A}_F} E_\epsilon(u,\chi) \leq C |\nu_2|^{\frac{2}{3}} \epsilon^{\frac{2}{3}}.
    \end{align*}
\end{rmk}

\subsubsection{Settings involving a higher number of wells}
\label{sec:introL14}

In order to illustrate that the above phenomenon is no coincidence, we show that it persists for a certain class of diagonal wells $\mathcal{K}_N$ having one-dimensional lamination convex hulls,
which had also been discussed with a full surface energy regularization in \cite{RT23}. Indeed, for $N \leq d+1$ we consider
\begin{align*}
\mathcal{K}_N:= \{A_1,A_2,\dots,A_N\} \subset \R^{d \times d}_{\diag},
\end{align*}
with
    \begin{gather} \label{eq:HigherLaminates_wells}
    \begin{aligned}
        A_1 &= 0, & A_2 & = \diag(1,0,\dots,0), \\
        A_3 &= \diag(\frac{1}{2},1,0,\dots,0), & A_4 &= \diag(\frac{1}{2},\frac{1}{2},1,\dots,0), \\
        A_j & = \diag(\frac{1}{2},\frac{1}{2},\dots,\frac{1}{2},1,0,\dots,0), &&
        \end{aligned}
    \end{gather}
for $j=5,6,\dots,N$, where for $A_j$ the entry $1$ is at the $(j-1)$-th diagonal entry.
For $d=2$ and $N = 3$ the set $\K_3$ is exactly the one defined in \cref{eq:Lorent_wells} above.
As in \cref{sec:introL13}, we again obtain a structure such that the lamination convex hull of the set $\mathcal{K}_N$ consists of one-dimensional segments (see \cref{fig:4_wells}): For $2 \leq \ell \leq N-1$
\begin{align*}
\K_N^{1} &:= \K_N^{(1)} \setminus \K_N = \{ \diag(\alpha,0,0,\dots,0) \in \R^{d \times d}: \alpha \in (0,1)\}, \\
\K_N^2 &:= \K_N^{(2)} \setminus \K_N^{(1)} = \{ \diag(\frac{1}{2},\alpha,0,\dots,0) \in \R^{d \times d}: \alpha \in (0,1) \}, \\
\K_N^3 &:= \K_N^{(3)} \setminus \K_N^{(2)} = \{ \diag(\frac{1}{2},\frac{1}{2},\alpha,0,\dots,0) \in \R^{d \times d}: \alpha \in(0,1)\}, \\
\mathcal{K}_N^{\ell} &:= \K_N^{(\ell)} \setminus \K_N^{(\ell-1)} = \{ \diag(\frac{1}{2},\frac{1}{2},\dots,\frac{1}{2},\alpha,0,\dots,0) \in \R^{d \times d} : \alpha \in (0,1)\},
\end{align*}
where $\K_N^\ell$ then has $(\ell-1)$ entries $1/2$, cf. \cref{sec:prelim_lam} for the definition of $K_N^{(\ell)}$.
Also in this setting, we show that in terms of lower bounds only the direction of the ``inner-most'' lamination is relevant for the scaling of the singularly perturbed model \cref{eq:L1_energy}, i.e. to obtain the ``classical'' isotropic scaling only $\nu \cdot e_1 \neq 0$ is necessary.

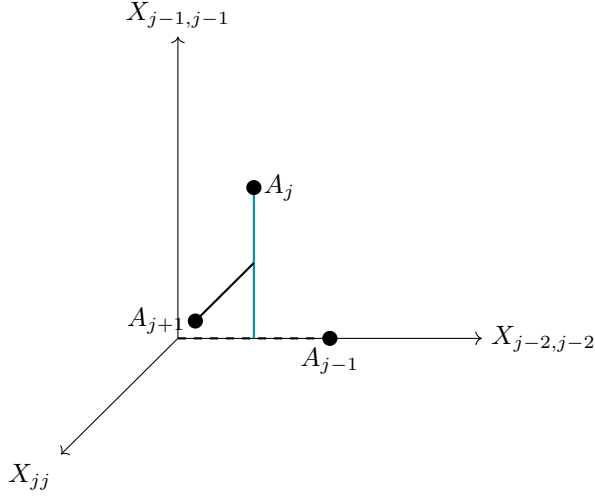
\begin{figure}
    \centering
    \begin{tikzpicture}[thick, scale = 2]
    \draw[thin,->] (0,0,0) --++ (2,0,0) node[right]{$X_{j-2,j-2}$};
    \draw[thin,->] (0,0,0) --++ (0,2,0) node[above] {$X_{j-1,j-1}$};
    \draw[thin,->] (0,0,0) --++ (0,0,2) node[below left] {$X_{jj}$};

        \draw[dashed] (0,0,0) -- (1,0,0);
        \draw[Blue] (0.5,0,0) -- (0.5,1,0);
        \draw (0.5,0.5,0) -- (0.5,0.5,1);

        \fill (1,0,0) circle (0.05) node[below] {$A_{j-1}$};
        \fill (0.5,1,0) circle (0.05) node[right] {$A_j$};
        \fill (0.5,0.5,1) circle (0.05) node[left] {$A_{j+1}$};
    \end{tikzpicture}
    \caption{Illustration of the relation between $A_{j-1}, A_j, A_{j+1} \in \K_N$ for $3 \leq j \leq N-1$. The lines show rank one connections, where the dashed line is connected to the structure spanned by the previous wells. The set $\K_N^{j-1}$ is highlighted in blue.}
    \label{fig:4_wells}
\end{figure}

\begin{thm}
\label{thm:L1_mult_wells}
Let $d \geq 2$, $N\leq d+1$, $\Omega = (0,1)^d$, $\nu \in \S^{d-1}$, $\K_N$ be given by \cref{eq:HigherLaminates_wells}, and $\ell \in \{1,2,\dots,N-1\}$.
Consider the energy $E_\epsilon(u,\chi)$ given in \cref{eq:L1_energy} with $\mathcal{N} = \{\nu\}$, $u \in \mathcal{A}_F$ for some $F \in \K^{\ell}_N$, cf. \cref{eq:AdmissibleFunctionsL2}, and $\chi \in BV_\nu(\Omega;\K_N)$.
Then, there are $C = C(d,F,\ell) > 0$ and $\epsilon_0 = \epsilon_0(d,F,\ell,\nu) > 0$ such that for any $\epsilon \in (0,\epsilon_0)$
\begin{align*}
    \inf_{\chi \in BV_\nu(\Omega;\K_N)} \inf_{u \in \mathcal{A}_F} E_\epsilon(u,\chi) \geq C \sum_{j=0}^{\ell-1} |\nu_{j+1}|^{\frac{2}{\ell-j+2}} \epsilon^{\frac{2}{\ell-j+2}} .
\end{align*}
\end{thm}

\begin{rmk}
    As above, the bound in \cref{thm:L1_mult_wells} includes various individual estimates which arise depending on the dominating degree of degeneracy of the data with respect to the direction $\nu \in \mathbb{S}^{d-1}$. Indeed, by considering different cases for $\nu$, we can also state the following lower scaling estimates:
\begin{enumerate}
\item For $F \in \mathcal{K}^{\ell}_{N}$ and $\nu \cdot e_1 \neq 0$ there are $C = C(d,F,\ell) >0$ and $\epsilon_0 = \epsilon_0(d,F,\ell,\nu) > 0$ such that for any $\epsilon \in (0,\epsilon_0)$
\begin{align*}
    \inf_{\chi \in BV_\nu(\Omega;\K_{N})} \inf_{u \in \mathcal{A}_F} E_{\epsilon}(u,\chi) \geq C|\nu_1|^{\frac{2}{\ell+2}}\epsilon^{\frac{2}{\ell+2}}.
\end{align*}
\item For $F \in \K_N^{\ell}$, $0<k < \ell$ and $\nu \cdot e_1 = \nu \cdot
  e_2 = \dots = \nu \cdot e_{k} = 0$ and $\nu \cdot e_{k+1} \neq 0$ there are $C
  = C(d,F,\ell) > 0$ and $\epsilon_0 = \epsilon_0(d,F,\ell,\nu) > 0$
   such that for any $\epsilon \in (0,\epsilon_0)$
\begin{align*}\inf_{u \in \mathcal{A}_F}
   \inf_{\chi \in BV_\nu(\Omega;\K_{N})} E_\epsilon(u,\chi) \geq C |\nu_{k+1}|^{\frac{2}{\ell-k+2}} \epsilon^{\frac{2}{\ell - k+2}}.
\end{align*}
\item For $F \in \K_N^{\ell}$ and $\nu \cdot e_1 = \nu \cdot e_2 = \dots = \nu \cdot e_{\ell} = 0$ we have $\inf_{\chi \in BV_\nu(\Omega;\K_{N})} \inf_{u \in \mathcal{A}_F} E_\epsilon (u,\chi) = 0$.
\end{enumerate}
\end{rmk}

We expect that this behaviour is sharp in the sense that there are matching upper bound constructions. As, however, already the upper bounds in \cite{RT23} were rather involved, we do not discuss these here.

\subsubsection{Main ideas for the sharp interface anisotropic surface energies}
Let us briefly comment on the ideas for the derivation of the above scaling laws with anisotropic surface energy contributions. As in \cite{RT23} these results rely on a Fourier space perspective on the elastic and surface energies \cite{K91, CO1, CO, RT23a, RT22}. In all these results, the role of the surface energy is to control high frequencies. While in the previous works this was done in an isotropic way with a frequency cut-off in all directions, we here leverage on the fact that the (diagonal) components of the phase indicator are already controlled by the elastic energy in most directions. In fact, in all the above models, the elastic energy provides strong control outside of certain cones along one-dimensional axes. Hence, only for these directions high frequency control becomes necessary. As a second ingredient, we exploit that the components of the phase indicator in the respective cones are not independent but are all functions of the inner-most one. This allows us to reduce the high frequency control even further and to require a singular perturbation regularization only in a single direction which is given by the direction of the inner-most component.

\subsection{Diffuse $L^{q}$-based surface energies}
\label{sec:intro_Lp}
Next, we turn to $L^{q}$-based diffuse surface energies with microstructures governed by energy functionals of the following form:
\begin{align}
\label{eq:Lp_energy}
E_{\epsilon,q}(u,\chi) := \int\limits_{\Omega} |\nabla u - \chi|^2 dx + \epsilon^{q} \sum\limits_{\nu \in \mathcal{N}} \int\limits_{\Omega}|\p_{\nu} (\nabla u)|^{q} dx, \ q \geq1,
\end{align}
with $\mathcal{N} \subset \mathbb{S}^{d-1}$.
As above, we will be particularly interested in anisotropic variants of this energy with $\# \mathcal{N} \leq d$ a finite set of linearly independent directions.
Here, due to the regularizing property of the surface energy, a novel phenomenon arises compared to the sharp structures which renders their analysis more challenging: due to the diffuse regularization, a second length scale emerges, which regularizes the ``zig-zag'' structures. This is already observed in one-dimensional models \cite{M93}, see also \cite{AM01}. In contrast to the sharp setting from above, we hence do not discuss these energies in Fourier space (although we believe that such a strategy should be possible) but deduce direct lower bounds in terms of the sharp energies.

Moreover, as the same arguments apply if the gradient is replaced by the symmetrized gradient $\nabla^{\textup{sym}}u= 1/2(\nabla u + \nabla u^T)$ we prove the results in a more general framework in order to be applicable also for this case.
For this let $p,q \in [1,\infty)$, $F \in \R^n$, $\nu \in \S^{d-1}$, $r \in \S^{n-1}$, and $L(D)$ a differential operator. We consider functions
 \begin{align*}
    U \in \{ U \in L^p_{\textup{loc}}(\R^d;\R^n): \ U = F \text{ outside } \Omega, \ \p_\nu (U \cdot r) \in L^q(\R^d;\R),\ L(D) U = 0 \text{ in } \R^d\},
\end{align*}
where the equation $L(D) U = 0$ holds in the distributional sense in $\R^d$.
For $L(D) = \operatorname{curl}$, this translates directly to $U = \nabla u$ for some $u \in W^{1,p}(\Omega;\R^d)$ such that $u(x) = Fx$ on $\p \Omega$.

Since this is not the main focus point of our article, in the following we will consider $L(D) \in \{\operatorname{curl}, \operatorname{curl} \operatorname{curl}\}$ but the same result holds for a more general class of operators, as for instance those considered in \cite{RRTT24}.

\begin{thm}
\label{thm:comparison}
Let $d \geq 2$, $n \geq 1$, let $p,q \in [1,\infty)$, $\nu \in \S^{d-1}$, $r \in \S^{n-1}$, and $\Omega \subset \R^d$ be a bounded Lipschitz
  domain and let $\K = \{A_1,\dots,A_N\} \subset \R^n$.
  For any $U \in L^p(\R^d;\R^{n})$ with $U = 0$ outside $\Omega$ and $\p_\nu (U \cdot r) \in L^q(\R^d;\R)$,
  $\chi \in L^\infty(\Omega;\K)$, there exists $\tilde{\chi} \in L^\infty(\Omega;\K)$ with $\tilde{\chi} \cdot r \in BV_{\nu}(\Omega;\K)$ and a constant $C =
  C(\K,p,q)>0$ such that for any $\epsilon >0$
  \begin{align*}
    \int_\Omega |U-\chi|^p + \epsilon^q |\p_\nu (U \cdot r)|^q dx \geq C \epsilon  \Vert D_{\nu} (\tilde{\chi} \cdot r) \Vert_{TV(\Omega)}
        \quad\text{and}\quad
    \int_\Omega|U-\chi|^p dx \geq C \int_\Omega|U-\tilde{\chi}|^p dx.
  \end{align*}
  In particular,
  \begin{align*}
    \int_\Omega |U-\chi|^p + \epsilon^q |\p_\nu (U \cdot r)|^q dx \geq \frac{C}{2}\left( \int_\Omega
    |U-\tilde{\chi}|^p dx + \epsilon \Vert D_\nu (\tilde{\chi} \cdot r)\Vert_{TV(\Omega)} \right).
  \end{align*}
\end{thm}

This relies on Modica-Mortola type arguments which have been discussed in various contexts in the literature \cite{MM77, KK11, Z14, CC14}, in particular, we follow the argument from \cite{KK11}.
Here the surface energy is even more degenerate than in the previous section in the sense that we only consider a single component of $U$ or $\tilde{\chi}$. We further elaborate on this additional degeneracy in \cref{rmk:lower_bound_component}.
We emphasize that the same estimates remain valid for the full data in which $U \cdot r$ and $\tilde{\chi} \cdot r$ are replaced by $U$ and $\tilde{\chi}$, respectively.

Together with the associated upper bound constructions, as a corollary, in our setting involving higher order laminates, we directly obtain that the scaling behaviour of the microstructures from the previous subsection is unchanged. As an example, we formulate this for the Lorent three-well problem from above where we consider $U = \nabla u$:

\begin{cor}\label{cor:Lp_3wells}
Let $q \in [1,\infty)$, $\K_3$ be given by \cref{eq:Lorent_wells}, $\Omega = (0,1)^2$, $\nu \in \S^1$, and consider for $\epsilon>0$, $F \in \K_3^{lc} \setminus \K_3$, $u \in \mathcal{A}_F$, cf. \cref{eq:AdmissibleFunctionsL2}, such that $\p_{\nu} \nabla u \in L^q(\Omega;\R^{2 \times 2})$, and $\chi \in L^{\infty}(\Omega;\K_3)$ the energy $E_{\epsilon,q}$ given by \cref{eq:Lp_energy} with $\mathcal{N}=\{\nu\}$.
The following scaling laws hold:

\begin{enumerate}
\item \emph{First order laminates:} For $F \in \K_3^1$ there is a constant $C = C(F, q)  > 0$ and $\epsilon_0 = \epsilon_0(F,q, \nu) > 0$ such that for any $\epsilon \in (0,\epsilon_0)$
\begin{align*}
    C^{-1} |\nu_1|^{\frac{2}{3}} \epsilon^{\frac{2}{3}} \leq \inf_{\chi \in L^2(\Omega;\K_3)} \inf_{\substack{u \in \mathcal{A}_F \\ \p_\nu \nabla u \in L^q(\Omega;\R^{2 \times 2})}}  E_{\epsilon,q}(u,\chi) \leq C |\nu_1|^{\frac{2}{3}} \epsilon^{\frac{2}{3}}.
\end{align*}
\item \emph{Second order laminates:} For $F \in \K_3^2$ there is a constant $C = C(F,q) > 0$ and $\epsilon_0 = \epsilon_0(F,q,\nu) > 0$ such that for any $\epsilon \in (0,\epsilon_0)$
\begin{align*}
    C^{-1}\left( |\nu_1|^{\frac{1}{2}} \epsilon^{\frac{1}{2}} + |\nu_2|^{\frac{2}{3}} \epsilon^{\frac{2}{3}} \right) \leq \inf_{\chi \in L^2(\Omega;\K_3)} \inf_{\substack{u \in \mathcal{A}_F \\ \p_\nu \nabla u \in L^q(\Omega;\R^{2 \times 2})}}  E_{\epsilon,q}(u,\chi) \leq C \left( |\nu_1|^{\frac{1}{2}} \epsilon^{\frac{1}{2}} + |\nu_2|^{\frac{2}{3}} \epsilon^{\frac{2}{3}}\right).
\end{align*}
\end{enumerate}
\end{cor}

\subsection{Fractional $L^2$-based surface energies}
\label{sec:frac_surface}

Building on the Fourier analysis from the sharp interface setting, we also consider fractional models for surface energies.
To this end, for $\mathcal{K} \subset \R^{d\times d}$ and $s \in (0,\frac{1}{2})$, we define the fractional (directional) Sobolev space on the torus using the Fourier transform, cf. \cref{eq:FourierTrafo},
\begin{align}
\label{eq:anisotropic_Hs}
    H^s_\nu(\T^d;\R^{d \times d}) = \{u \in L^2(\T^d;\R^{d \times d}): \sum_{k \in \Z^d} (1+4\pi^2 |k \cdot \nu|^2)^s |\hat{u}(k)|^2 < \infty\}
\end{align}
and set for $\chi \in H^s_\nu(\T^d;\K)$
\begin{align*}
    E_{surf}^s(\chi) := \left( \sum_{k \in \Z^d \setminus \{0\}} |k \cdot \nu|^{2s} |\hat{\chi}(k)|^2 \right)^{\frac{1}{2s}}.
\end{align*}
For $s < 1/2$ it then holds $BV_\nu(\T^d;\K) \subset H^s_\nu(\T^d;\K)$, see the proof of \cref{thm:frac_energy}. Moreover, in the following by extending $\chi \in L^2((0,1)^d;\K)$ one-periodically, we view it as a function on the torus $\chi \in L^2(\T^d;\K)$ and define the space $H^s_\nu((0,1)^d;\R^{d \times d})$ analogously.
The power of $1/2s$ is chosen such that the surface energy admits the correct behaviour in the length of interfaces.
In addition, by the choice of $\nu \in \mathbb{S}^{d-1}$ we also focus in this nonlocal context on anisotropic settings.
The full energy is then defined to be
\begin{align}
\label{eq:frac_energy}
E_{\epsilon,s}(u,\chi) := \int\limits_{\Omega} |\nabla u - \chi|^2 dx + \epsilon   E_{surf}^s(\chi).
\end{align}

We prove that also in this setting, the scaling laws from \cref{sec:intro_L1} remain valid.

\begin{thm}
\label{thm:frac_energy}
Let $d \geq 2$, $N \leq d+1$, $s \in (0,\frac{1}{2})$, $\Omega = (0,1)^d$, $\nu \in \S^{d-1}$, and let $\K_N$ be defined in \cref{eq:HigherLaminates_wells}, and $\ell\in\{1,\dots,N-1\}$. For $F \in \K_N^\ell$, cf. \cref{sec:prelim_lam}, consider the energy $E_{\epsilon,s}(u,\chi)$ defined in \cref{eq:frac_energy} for $u \in \mathcal{A}_F$, cf. \cref{eq:AdmissibleFunctionsL2}, and $\chi \in H^s_\nu(\Omega;\K_N)$.
Then there are $C = C(d,F,s,\ell) > 0$ and $\epsilon_0 = \epsilon_0(d,F,s,\ell,\nu) > 0$ such that for any $\epsilon \in (0,\epsilon_0)$
\begin{align*}
    \inf_{\chi \in H^s_\nu(\Omega;\K_N)} \inf_{u \in \mathcal{A}_F} E_{\epsilon,s}(u,\chi) \geq C \left( \sum_{j=0}^{\ell-1} |\nu_{j+1}|^{\frac{2}{\ell-j+2}} \epsilon^{\frac{2}{\ell-j+2}} \right).
\end{align*}
Moreover, for $d=2$ and $N = 3$ a matching upper bound holds
\begin{align*}
    \inf_{\chi \in H^s_\nu(\Omega;\K_3)} \inf_{u \in \mathcal{A}_F} E_{\epsilon,s}(u,\chi) \leq C^{-1} \left( \sum_{j=0}^{\ell-1} |\nu_{j+1}|^{\frac{2}{\ell-j+2}} \epsilon^{\frac{2}{\ell-j+2}} \right).
\end{align*}
\end{thm}
As above, this shows the robustness of the scaling bounds with respect to possible choices of the (anisotropic) surface energies. While we only discuss the upper bound in the case of the Lorent three-well problem, our interpolation strategy of \cref{sec:frac_proof} shows that the fractional energy bounds are sharp whenever they are sharp in the sharp interface settings.

\subsection{Discrete models}
\label{sec:intro_Tartar}
As a last prototypical setting we turn to discrete models as regularized versions of continuum models. It is well-known \cite{D03,L01,L09,CM99,ALP2017} in the context of shape-memory alloys and, more generally, in phase transition problems \cite{ABC07} that under suitable conditions, discretizations lead to surface-type energies in a first order expansion. In fact, \cite{L09} proved the equivalence of scaling laws for continuous singular perturbation models with isotropic surface energy regularization and a class of associated discrete models. The article \cite{L09} even considered settings in the presence of gauges, i.e., with the full $SO(2)$ symmetry. These results, hence, prove that also for shape-memory type models, one expects that -- provided that the lattice structures and the rank-one geometry of the wells match -- the scaling behaviour in the discretization parameter corresponds to that of the associated singular perturbation model.

Similarly as above, in this section, we seek to extend these observations to \emph{anisotropic} scenarios. As above, we search for minimal conditions on the lattice with respect to the
geometry of the wells to guarantee this behaviour. For a fixed class of lattice structures, we here show that discrete energies with an ``anisotropic lattice structure'' can be compared to singular perturbation models with anisotropic surface regularization contributions.

\subsubsection{The discrete set-up}
In order to outline our result, let us introduce some notation. We consider a specific choice of lattice structure. Here
the triangulation is given by using the following ``upper" and ``lower" triangles (see \cref{fig:triangulation_referencetriags})
\begin{align}
\label{eq:lattice}
    T_h := \{x \in [0,h)^2: x_2 < h - x_1\},\quad T_h' := \{x \in [0,h)^2: x_2 \geq h-x_1\},
\end{align}
and, hence,
\begin{align*}
    \mathcal{T}_h := \left\{  T_h+z: z \in h \Z^2 \right\} \cup \left\{T_h'+z : z \in h \Z^2 \right\}.
\end{align*}
We also consider rotated variants of this triangulation (see \cref{fig:triangulation_unit_square}), that is for $R \in SO(2)$, consider
\begin{align} \label{eq:Triangulation}
    \mathcal{T}_h^R := R \mathcal{T}_h =  \left\{ R(T_h+z): z \in h \Z^2 \right\} \cup \left\{ R(T_h'+z): z \in h \Z^2\right\}.
\end{align}

Given the lattice structure, we associate an energy to it:
For $\Omega \subset \R^2$ bounded and polygonal, $F \notin \K$, and $p \in [1,\infty)$, we define the sets of admissible functions
\begin{align}
\label{eq:admissible_discrete}
\begin{split}
\mathcal{A}_{h, F}^{p,R} &:= \{u\in W^{1,p}_{\textup{loc}}(\R^2;\R^2) : u \text{ is affine on each triangle } \tau \in \mathcal{T}_h^R, u(x) = Fx \text{ outside } \Omega\}, \\
\mathcal{C}_{h}^{R} &:= \{\chi \in L^\infty(\R^2;\K) : \chi \text{ is constant on  each triangle } \tau  \in \mathcal{T}_h^R\}.
\end{split}
\end{align}
We consider the energy given by
\begin{align}
\label{eq:energy_discrete}
    E_{el,h}^{p}(u,\chi) := \int_\Omega |\nabla u - \chi|^p dx,
    \quad
    \text{for }u\in\mathcal{A}_{h, F}^{p,R}, \chi\in\mathcal{C}_{h}^{R}.
\end{align}
We highlight that by the choice of the admissible deformations this corresponds to a discrete energy. Upper bounds for multi-well energies of this type have been systematically investigated in \cite{C99} for a rich class of well structures. The two-well problem had been studied with matching upper and lower bounds in \cite{CM99}.
We also refer to \cite{CCK95, BP04} for further upper bound constructions and matching lower bounds on the two-well problem.

\begin{figure}
    \centering
    \begin{subfigure}[t]{0.3\textwidth}
        \centering
        \begin{tikzpicture}[scale=2]
            \draw (0,0) -- (1,0) -- (0,1) -- cycle;
            \draw (1,0) -- (1,1) -- (0,1) -- cycle;

            \node at (0,0) [above right] {$T_h$};
            \node at (1,1) [below left] {$T_h'$};

            \draw[|-|] (-0.2,0) -- (-0.2,1) node[midway,left] {$h$};
        \end{tikzpicture}
        \caption{Illustration of the upper and lower triangles as defined in \cref{eq:lattice}.}
        \label{fig:triangulation_referencetriags}
    \end{subfigure}%
    ~
    \begin{subfigure}[t]{0.6\textwidth}
    \centering
    \begin{tikzpicture}[scale = 0.7]
    \draw[very thick, Blue, fill = Blue, fill opacity = 0.3] (0,0) rectangle (6,6);
    \begin{scope}[rotate = 15]
        \draw (-1,-2) rectangle (8,6);
        \foreach \x in {0,1,2,3,4,5,6,7}{
            \draw (\x,-2) -- (\x,6);
        }
        \foreach \y in {-1,0,1,2,3,4,5}{
            \draw (-1,\y) -- (8,\y);
        }
        \foreach \x in {0,1,2,3,4,5,6,7}{
            \draw (\x,-2) -- (-1,\x-1);
            \draw (8,\x-2) -- (\x,6);
        }
        \draw[|-|] (-1.2,-2) -- (-1.2,-1);
        \node at (-1.2,-1.5)[left] {$h$};
        \end{scope}
    \end{tikzpicture}
    \caption{Triangulation $\mathcal{T}_h^R$ of $(0,1)^2$ (highlighted in blue) with $h = 1/6$ and $R$ a rotation by $15^\circ$.}
    \label{fig:triangulation_unit_square}
    \end{subfigure}
    \caption{Illustration of the definition of triangulations used in this section.}
    \label{fig:triangulation}
\end{figure}
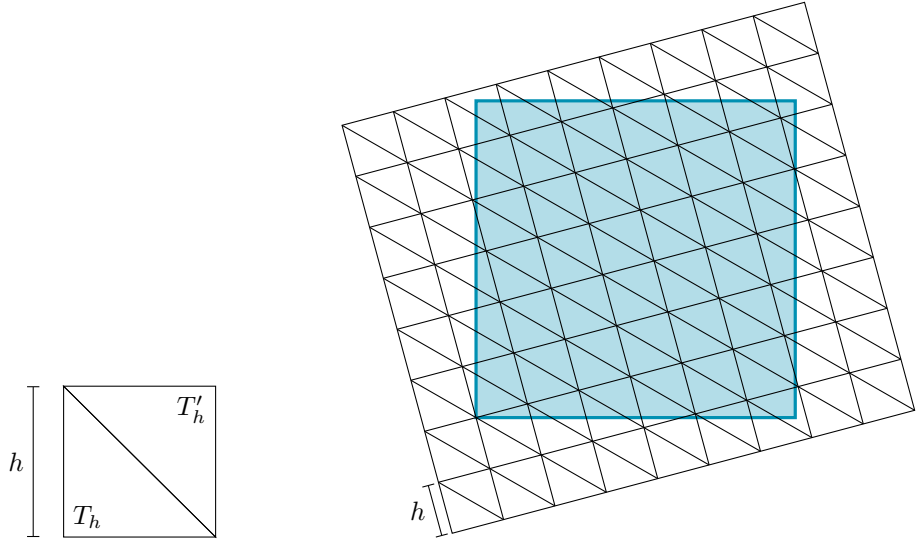

For ease of notation, in this section we consider the row-wise cross product of a vector and a matrix defined by
\begin{align*}
    (v \times M)_j := v_1 M_{j2} - v_2 M_{j1}, \ v \in \R^2, M \in \R^{2 \times 2}, j=1,2.
\end{align*}
In particular, using this notation, there is a rank-one connection between two matrices $A, B \in \R^{2 \times 2}$, i.e., $A-B = a \otimes n$ for $a, n \in \R^2 \setminus \{0\}$, if and only if $A\neq B$ and
\begin{align*}
    n \times (A-B) = 0.
\end{align*}

Given this set-up, our main result is given by the following comparison estimate.

\begin{thm} \label{cor:DiscreteToSharp}
Let $\Omega \subset \R^2$ be an open, bounded polygonal domain. Let $R \in SO(2)$.
   Let $\K = \{A_1,A_2,\dots,A_N\} \subset \R^{2 \times 2}$ and $F \notin \K$.
    Let $p \in [1,\infty)$ and consider $\mathcal{T}_h^R$ defined in \cref{eq:Triangulation}, and the discrete energy $E_{el,h}^{p}(u,\chi)$ for $u \in \mathcal{A}^{p,R}_{h,F}$ and $\chi \in \mathcal{C}^R_h$, cf. \cref{eq:admissible_discrete}, defined in \cref{eq:energy_discrete}.
    Then the following results hold:
    \begin{enumerate}
    \item \emph{Isotropic setting:}
     Assume that for any $A_j, A_k \in \K$ with $j \neq k$ it holds
\begin{align}\label{eq:assumption-wells}
(Re_1) \times (A_j - A_k) \neq 0, \ (Re_2) \times (A_j - A_k) \neq 0, \ \Big(R\begin{psmallmatrix} 1/\sqrt{2} \\ 1/\sqrt{2} \end{psmallmatrix}\Big) \times (A_j - A_k) \neq 0.
\end{align}
   Then, there are constants $C = C(\K,R,p,F,\Omega) > 0$ and $h_0 = h_0(\Omega,R,p)>0$ such that for all $h \in (0,h_0)$
    \begin{align*}
        E_{el,h}^p(u,\chi) \geq C \left( \int_{\Omega} |\nabla u - \chi|^p dx + h \Vert D \chi \Vert_{TV(\Omega)} +h \right).
    \end{align*}
    In particular, we have
    \begin{align*}
        \inf_{u\in\mathcal{A}^{p,R}_{h,F}} \inf_{\chi\in\mathcal{C}^R_h}  E_{el,h}^p(u,\chi) \geq C \Big( \inf_{\chi \in BV(\Omega;\K)} \inf_{\substack{u \in W^{1,p}(\Omega;\R^2) \\ u(x) = Fx  \text{ on } \p \Omega}}\int_{\Omega} |\nabla u - \chi|^p dx + h \Vert D \chi \Vert_{TV(\Omega)} +h\Big).
    \end{align*}

    \item \emph{Anisotropic setting:}
    Assume that there exists $v \in \Big\{e_1, e_2, \begin{psmallmatrix} 1/\sqrt{2} \\ 1/\sqrt{2} \end{psmallmatrix}\Big\}$ such that for any $j\neq k$
    \begin{align}
    \label{eq:assumption-wells-deg}
        (Rw) \times (A_j - A_k) \neq 0, \quad \text{for every } w\in\Big\{e_1,e_2,\begin{psmallmatrix} 1/\sqrt{2} \\ 1/\sqrt{2} \end{psmallmatrix}\Big\}\setminus\{v\}.
    \end{align}
    Then, there are constants $C = C(\K,R,p,F,\Omega,v) > 0$ and $h_0 = h_0(\Omega,R,p)>0$ such that for all $h \in (0,h_0)$
    \begin{align*}
        E_{el,h}^p(u,\chi)  \geq C \left( \int_\Omega|\nabla u - \chi|^p + h \Vert D_{R \nu} \chi \Vert(\Omega)+h\right),
    \end{align*}
    with $\nu \in \S^1$ such that $\nu \cdot v = 0$.
    \end{enumerate}
\end{thm}

We emphasize that the result on the \emph{isotropic} case is not new. It only recovers the setting already analysed in \cite{L09} for the general $N$-well problem in two-dimensions which in \cite{L09} is even considered with $SO(2)$ symmetry. In the two-well case, it recovers the result from \cite{CM99}. Our main contribution is the bound in the \emph{anisotropic} setting. Its proof relies on similar ideas as in \cite{L09}.

\subsubsection{Applications}
\label{sec:intro_appl}
We will consider two prototypical examples -- the Lorent three-well setting and the Tartar square. In both settings, it is interesting to trace the effects of anisotropy.

We begin by discussing the discrete version of the Lorent three-well problem, introduced in \cref{sec:introL13}. It admits laminates up to order two.

\begin{cor}
\label{cor:Lorent_discrte}
    Let $\Omega = (0,1)^2$, $R \in SO(2)$, $\mathcal{T}_h^R$ be the triangulation defined in \cref{eq:Triangulation}, $\K_3$ be given in \cref{eq:Lorent_wells}, and $F \in \K_3^{lc} \setminus \K = \K_3^1 \cup \K_3^2$, cf. \cref{eq:laminates_order}.
    Let $E_{el,h}(u,\chi) := E_{el,h}^2(u,\chi)$ be defined in \cref{eq:energy_discrete} for $p=2$ and $u \in \mathcal{A}^{2,R}_{h,F}$, $\chi \in \mathcal{C}^R_h$, cf. \cref{eq:admissible_discrete}.
    Then there are constants $C = C(R,F) > 0$ and $h_0 = h_0(R,F)> 0$ such that one of the following applies.
    \begin{enumerate}
    \item\label[case]{itm:Lorent_discrete_i} \emph{Isotropic setting.} If $R \in SO(2)$ is such that $R^T e_1 \notin \{\pm e_1, \pm e_2, \pm 2^{-1/2}(e_1+e_2)\}$ we have for any $h \in (0,h_0)$
    \begin{align*}
        \inf_{u \in \mathcal{A}^{2,R}_{h,F}} \inf_{\chi \in \mathcal{C}^R_h} E_{el,h}(u,\chi) \geq \begin{cases}
            C h^{\frac{2}{3}}, & F \in \K_3^1, \\
            C h^{\frac{1}{2}}, & F \in \K_3^2.
        \end{cases}
    \end{align*}
    \item\label[case]{itm:Lorent_discrete_ii} \emph{Anisotropic setting.}
    If $R^T e_1 \in \{\pm e_1, \pm e_2, \pm 2^{-1/2}(e_1+e_2)\}$ it holds for any $h \in (0,h_0)$
    \begin{align*}
        \inf_{u \in \mathcal{A}^{2,R}_{h,F}} \inf_{\chi \in \mathcal{C}^R_h} E_{el,h}(u,\chi) \geq \begin{cases}
            C h, & F \in \K_3^1, \\
            C h^{\frac{2}{3}}, & F \in \K_3^2.
        \end{cases}
    \end{align*}
    \end{enumerate}
    Moreover, the matching upper bounds hold, that is for every $h \in (0,h_0)$
    \begin{align*}
        \inf_{u \in \mathcal{A}^{2,R}_{h,F}} \inf_{\chi \in \mathcal{C}^R_h} E_{el,h}(u,\chi) \leq \begin{cases}
            C^{-1} h^{\frac{2}{3}}, & F \in \K_3^1, \\ C^{-1} h^{\frac{1}{2}}, & F \in \K_3^2,
        \end{cases}
    \end{align*}
    in the \emph{isotropic setting}, and
    \begin{align*}
        \inf_{u \in \mathcal{A}^{2,R}_{h,F}} \inf_{\chi \in \mathcal{C}^R_h} E_{el,h}(u,\chi) \leq \begin{cases}
            C^{-1} h, & F \in \K_3^1, \\ C^{-1} h^{\frac{2}{3}}, & F \in \K_3^2,
        \end{cases}
    \end{align*}
    in the \emph{anisotropic setting}.
\end{cor}

Let us comment on these results. In the isotropic setting, we recover the lower scaling bounds from the continuous setting. A discretized version of the upper bounds from the continuous setting shows that these estimates are sharp. The more interesting setting arises in the anisotropic situation. Here, for second order laminate boundary conditions, the lower bounds also match the anisotropic ones in the Lorent three-well problem in \cref{thm:L1_3wells}. For boundary data consisting of first order laminates, we observe a difference to the continuum -- due to the limitation of the size of the lengths scales, a scaling bound of the order $h$ emerges in the discrete setup, while it vanishes completely in the continuum. Again, a discretization of the upper bounds from the continuous setting yields the sharpness of these estimates. We expect that similar behaviour is observed for other well configurations.

As a second example and as an extreme case, we consider the setting of the Tartar square. This is a setting of four diagonal matrices which play a prominent role both in inner-mathematical settings \cite{T93, MS03, FS08} and materials \cite{CS13,BFKK94,IKRTZ24}. Indeed, while these four diagonal matrices are pairwise incompatible, they still admit microstructure and display a first loss of rigidity in that approximate solutions become flexible.
  They are given by $T_4 := \{A_1,A_2,A_3,A_4\}$ with
  \begin{align} \label{eq:Tartar}
    A_1 &= \begin{pmatrix} -1 & 0 \\ 0 & -3 \end{pmatrix}, & A_2 &= \begin{pmatrix} -3 & 0 \\ 0 & 1 \end{pmatrix}, &  A_3 &= \begin{pmatrix} 1 & 0 \\ 0 & 3 \end{pmatrix}, & A_4 &= \begin{pmatrix} 3 & 0 \\ 0 & -1 \end{pmatrix}.
  \end{align}
As is well-known in the case of the Tartar square, we have that $T_4^{lc} = T_4$, but its quasiconvex hull is given by $T_4^{qc}:= \{J_1, J_2, J_3, J_4 \}^{\textup{conv}} \cup \bigcup\limits_{j=1}^{4}[A_j,J_j]$ (see \cite[Theorem 9.4]{Ri18}), where
for a set $M\subset \R^{d\times d}$, the notation $M^{\textup{conv}}$ denotes the convex hull of $M$. Here the auxiliary matrices $J_1, \dots, J_4$ are given as
  \begin{align} \label{eq:Tartar_aux}
    J_1 &= \begin{pmatrix} -1 & 0 \\ 0 & -1 \end{pmatrix}, & J_2 &= \begin{pmatrix} -1 & 0 \\ 0 & 1 \end{pmatrix}, &  J_3 &= \begin{pmatrix} 1 & 0 \\ 0 & 1 \end{pmatrix}, & J_4 &= \begin{pmatrix} 1 & 0 \\ 0 & -1 \end{pmatrix}.
  \end{align}

With this notation in hand, we prove the following scaling law for the discrete version of the Tartar energy.

\begin{cor}
\label{cor:Tartar_discrete}
    Let $\Omega=(0,1)^2$, $R \in SO(2)$, and let $T_4 = \{A_1,A_2,A_3,A_4\}$ be the Tartar square given in \cref{eq:Tartar} and let $F \in T_4^{qc}\setminus T_4$. Consider the discrete energy $E_{el,h}(u,\chi) := E_{el,h}^2(u,\chi)$ for $u \in \mathcal{A}^{2,R}_{h,F}$, $\chi \in \mathcal{C}^R_h$, cf. \cref{eq:admissible_discrete}, defined in \cref{eq:energy_discrete}.
    Then, for any $R \in SO(2)$ and any $\eta>0$ there exist constants $C_{\eta} >0,C  = C(F,R) >0$ and $h_0 = h_0(F,R) > 0$ such that for all $h \in (0,h_0)$
    \begin{align*}
        \inf_{u \in \mathcal{A}^{2,R}_{h,F}} \inf_{\chi \in \mathcal{C}^R_h} E_{el,h}(u,\chi) \geq C \exp(-C_\eta |\log h|^{\frac{1}{2}+\eta}).
    \end{align*}
\end{cor}

Let us comment on this result. Contrary to the setting from \cref{cor:Lorent_discrte}, the lower scaling bound holds for all lattice structures in our class of lattices, independently of the choice of the rotation $R$.
This is due to the fact that the Tartar square is extremely rigid, in the sense that no rank-one connections are present. Hence, independently of the precise lattice structure, the lower bound directly follows from an application of the isotropic setting from \cref{cor:DiscreteToSharp} for any rotation $R \in SO(2)$. We remark that the lower bound is essentially sharp. An upper bound of the form $C^{-1} \exp(-C' |\log h|^{\frac{1}{2}})$ (with a constant $C'>0$)  had been deduced in \cite{C99}. The loss of $\eta>0$ in the exponent in \cref{cor:Tartar_discrete} is expected to be a technical artifact, which had already been present in the discussion of its continuous counterpart in \cite{RT22}.

\subsection{Relation to the literature}
Let us connect our results to the literature on martensitic phase transformations. In the literature, martensitic phase transformations are considered with various types of surface energies, including sharp \cite{KM1,KM2,CO,CO1}, diffuse \cite{Z14},
mixed ones \cite{BMC09} and discrete regularizations \cite{D03, L01, L09, ABC07}. In fact, Lorent \cite{L09} proved that under additional technical assumptions on the finite element discretization, for a model with frame indifference, the discrete and continuous, diffuse energies are scaling equivalent.
 We refer to \cite{KK11, CC14, CC15, CDZ17, KO19, RTZ19, RRT23, RT24, GRTZ24, TZ24, AKKR24, IKRTZ24} for further results on scaling laws for shape-memory alloys with various types of surface energy regularizations.
Also experimentally, one observes different transitions: In \cite{BvTA86} one observes atomistically sharp interfaces while in \cite{MvTA86} rather diffuse boundaries are observed. Measurements of surface energies are experimentally notoriously difficult and, hence, the precise structure is often not known. Thus, it is particularly important to deduce mathematical results \emph{independently} of the precise form of the surface energies. By discussing the scaling laws from \cite{RT23, RT22} in the context of various surface regularizations, it is our objective to illustrate and prove that these results are rather robust in the choice of the surface energies and to identify for an interesting class of martensitic phase transformations minimal anisotropic conditions recovering the known scaling laws for isotropic singular perturbation contributions.

\subsection{Outline of the article}
The remainder of the article is structured as follows. We begin by discussing some preliminary results in \cref{sec:prelim}. In particular, we recall the relevant $BV$ set-up and the associated high-frequency bounds and the Fourier analysis from \cite{RT22, RT23}. In \cref{sec:L1AnisotropicSurf} we turn to the setting of anisotropic, sharp surface energies. We present both upper and lower bound results in the outlined highly anisotropic settings and thus provide the proofs of \cref{thm:L1_3wells,thm:L1_mult_wells}. In \cref{sec:diffuse} we turn to the discussion of diffuse energies. Here we provide comparison results such as \cref{thm:comparison} which relate diffuse surface energies with the sharp energies and also present the lower bounds in the setting of anisotropic fractional energies. Finally, we consider discrete (anisotropic) situations in \cref{sec:discrete} and present the proof of \cref{cor:DiscreteToSharp} and the applications from \cref{cor:Lorent_discrte,cor:Tartar_discrete}.

\section{Preliminaries}
\label{sec:prelim}

In this section, we collect various auxiliary results which we will use in the following sections.

\subsection{On the lamination convex hull}
\label{sec:prelim_lam}
For the convenience of the reader we recall the lamination convex hull of a set $\K \subset \R^{d\times d}$.

\begin{defi}
\label{def:lam_hull}
Let $\K \subset \R^{d \times d}$. The \emph{lamination convex hull} $\K^{lc}$ of $\K$ is given by
\begin{align*}
\K^{lc}:= \bigcup\limits_{j=0}^{\infty} \K^{(j)},
\end{align*}
where
\begin{align*}
\K^{(0)}:= \K, \
\K^{(j)}:= \{\lambda A + (1-\lambda) B: \ A, B \in \K^{(j-1)}, \ \lambda \in [0,1], \ \rank(A-B) = 1\}.
\end{align*}
For $j\geq 1$, we refer to the elements of $\K^j := \K^{(j)}\setminus \K^{(j-1)}$ as \emph{laminates of order $j$}.
\end{defi}

In what follows below, we will prove that in our geometric settings with anisotropic surface energies, the interaction between the directional dependences in the anisotropy (in the surface energies) and the lamination orders of the boundary data will determine the scaling behaviour of the investigated microstructures.

\subsection{Directional derivative in a $BV$ sense} \label{sec:prelim_directBV}
Given $\Omega\subseteq\R^d$ open, for a function $f \in BV(\Omega;\R^n)$ we denote the total variation norm of the measure $Df$ by
$\Vert Df \Vert_{TV(\Omega)}$.
Building on the definition of $BV$ functions, we consider functions for which only a single directional
derivative exists as a measure.
Let $f \in L^1(\Omega;\R^n)$ and $\nu \in \S^{d-1}$. We write $f \in BV_\nu(\Omega;\R^n)$ if the
distributional derivative of $f$ in direction $\nu$ is a $\R^n$-valued Radon-measure, denoted by $D_\nu f$,
i.e. for every $\phi \in C^\infty_c(\Omega;\R^n)$ it holds that
\begin{align*}
  \int_\Omega f \cdot \p_\nu \phi dx = - \int_\Omega \phi \cdot d(D_\nu f).
\end{align*}
The
total variation norm of this measure, again denoted by $\Vert D_\nu f \Vert_{TV(\Omega)}$, is given as
\begin{align*}
  \Vert D_\nu f \Vert_{TV(\Omega)} = \sup \left\{\int_\Omega f \cdot \p_\nu \phi dx: \phi \in
  C^1_c(\Omega; \R^n), \Vert \phi \Vert_{\infty} \leq 1 \right\}.
\end{align*}

As shown in \cite[Thm 3.103]{AFP00}, for $f \in BV_\nu(\Omega;\R^n)$ it holds
\begin{align} \label{eq:AnisoBV_slicing}
  \Vert D_\nu f \Vert_{TV(\Omega)} = \int_{\Omega_\nu} \Vert Df_y^{\nu} \Vert_{TV(\Omega^{\nu}_{y})} dy,
\end{align}
where we introduce the notation
\begin{align*}
  \Omega_\nu := \Pi_{\nu^\perp} \Omega \subset \R^d
\end{align*}
as the orthogonal projection of $\Omega$ onto $\nu^\perp$, and for $y \in \Omega_\nu$, we set
\begin{align*}
  \Omega^\nu_y := \{t \in \R: y + t \nu \in \Omega\} \subset \R.
\end{align*}
Hence, the total variation norm of $D_\nu f$ is given by integration of the one-dimensional total
variation norms of the distributional derivatives of
\begin{align*}
  f^\nu_y: \Omega^\nu_y \to \R^n, \quad f^\nu_y(t) := f(y+t\nu).
\end{align*}
In particular, for $f \in BV_\nu(\Omega;\R^n)$ for almost every $y \in \Omega_\nu$ we have
$f^\nu_y \in BV(\Omega^\nu_y;\R^n)$.

As a remark,  let $E \subset \Omega$ be a set of finite perimeter, then it holds
\begin{align*}
    \Vert D_\nu \chi_{E} \Vert_{TV(\Omega)} = \int_{\p E \cap \Omega} |n_{\p E} \cdot \nu| d \mathcal{H}^{d-1},
\end{align*}
where $n_{\p E}$ is the outer unit normal of $E$.
We will often exploit this in the sequel.

\subsection{Directional high frequency control}
As a preparation for our discussion of sharp energies, we deduce lower bounds for the surface energies. Here we follow similar ideas as in \cite{CO,CO1,RT22, RRT23} but with only directional control for the surface energies. For the convenience of the reader, we hence recall the arguments.
These arguments use Fourier methods, thus let us recall the definition of the Fourier transform of one-periodic functions $u \in L^2(\T^d;\R^n)$ as $\mathcal{F}[u] \in \ell^2(\Z^d;\C^n)$
\begin{align} \label{eq:FourierTrafo}
    \mathcal{F}[u](k): = \hat{u}(k) := \int_{\T^d} u(x) e^{-2\pi i k \cdot x} dx, \quad k \in \Z^d.
\end{align}
In the following, we often consider the one-periodic extension of a function $u \in L^2((0,1)^d;\R^n)$ without changing the notation and hence also define the Fourier transform of such functions as above.

In order to motivate our arguments in the following sections, we recall part of the strategy from the isotropic setting from \cite{CO,CO1,RT22, RRT23} in two dimensions.
As a central ingredient for the lower bound, we use the surface energy to control high frequencies. More precisely, viewing $\chi \in BV((0,1)^2;\K)$ as a periodic function, we have
\begin{align*}
  \sum_{|k| \geq \lambda} |\hat{\chi}(k)|^2 \leq C \lambda^{-1} (\Vert D \chi \Vert_{TV((0,1)^2)} + \Per((0,1)^2)).
\end{align*}
As a complementary step, the elastic energy is used to control the frequencies of the associated ``multiplier'' in the form of a coercivity bound in the complement of certain non-elliptic regions. The analysis of these multipliers shows that the only regions without coercivity are given by the complement of cones (cf.\ the discussion below).
 As a consequence, it seems natural to conjecture that it suffices to have only control of high frequencies in direction of the axes of these cones instead of requiring control in all directions.

Our first step towards the anisotropic setting, thus, is showing an analogous high frequency control as above, but using only a single direction.
\begin{lem} \label{lem:HighFreqDir}
Let $d\geq 2$, $n \geq 1$, let
  $\Omega = (0,1)^d$, $\nu \in \S^{d-1}$, and $f \in BV_\nu(\Omega;\R^n) \cap L^\infty(\Omega;\R^n)$, then there is a constant $C = C(d) > 0$ independent of $\nu$ and $f$, such that for any $\lambda >0$ it
  holds
  \begin{align*}
    \sum_{k \in \Z^d: |k \cdot \nu| \geq \lambda} |\hat{f}(k)|^2 \leq C \lambda^{-1} \Vert f\Vert_{L^\infty} \left( \Vert D_\nu f \Vert_{TV(\Omega)} + \Vert f \Vert_{L^\infty} \Per(\Omega) \right),
  \end{align*}
  where we extend $f$ one-periodically.
\end{lem}

\begin{proof}
We preliminarily notice that, in the case $\lambda\le1$, the statement is straightforward.
Indeed,
$$
\sum_{k \in \Z^d: |k \cdot \nu| \geq \lambda} |\hat{f}(k)|^2\le\|f\|_{L^2(\Omega)}^2\le\|f\|_{L^\infty(\Omega)}^2\le\lambda^{-1}\|f\|_{L^\infty(\Omega)}^2,
$$
which yields the statement by taking $C\ge\Per(\Omega)^{-1}.$
We are then left to prove the result in the case $\lambda>1$.

We use the notation introduced in \cref{sec:prelim_directBV}, i.e.\ we set $\Omega_\nu = \Pi_{\nu^\perp} \Omega$ the orthogonal projection of $\Omega$ onto $\nu^\perp$, and for $y \in \Omega_\nu$ we set $\Omega_y^\nu = \{t \in \R: y + t \nu \in\Omega\}$.
For given $f \in BV_\nu(\Omega;\R^n)$, we consider the one-periodic extension of $f$. Thus, let us first consider $f \in BV_\nu(\T^d;\R^n)$.

    \emph{Step 1: Estimate on a difference quotient.} We claim that, for every $|h|<1$
    \begin{align} \label{eq:DiffQuotBV}
        \int_{\T^d} \frac{|f(x+h\nu)-f(x)|}{|h|} dx \leq \Vert D_\nu f \Vert_{TV(\T^d)}.
    \end{align}

    To show \cref{eq:DiffQuotBV}, we begin by noting that for a given connected domain of integration $Q \subset \R^d$ and for  $f^\nu_y \in BV(Q^\nu_y;\R^n)$, by the Fundamental Theorem (and approximation), we obtain that
    \begin{align*}
        \int_{Q_y^\nu\cap(Q_y^\nu-h)}|f^\nu_y(t+h) - f^\nu_y(t)| dt \leq |h| \Vert {D}f^\nu_y \Vert_{TV(Q^\nu_y)}.
    \end{align*}

    We choose $Q \subset \R^d$ as a rotated cube with one face in direction $\nu$, such that $\Omega \subset Q\cap(Q-h\nu)$,
    see \cref{fig:HighFreq_Cube}.
    Thus, by \cref{eq:AnisoBV_slicing} and the periodicity of $f$
    \begin{align*}
        \int_{\T^d} |f(x+h\nu) - f(x)| dx &\leq \int_{Q\cap(Q-h\nu)} |f(x+h \nu) - f(x)| dx \\
        &= \int_{Q_{\nu}} \int_{Q_y^\nu\cap(Q_y^\nu-h)} |f^\nu_y(t+h) - f^\nu_y(t)| dt dy \\
        & \leq |h| \int_{Q_{\nu}} \Vert D f^\nu_y \Vert_{TV(Q^\nu_y)} dy = |h| \Vert D_\nu f \Vert_{TV(Q)} \leq |h| C(Q,d) \Vert D_\nu f  \Vert_{TV(\T^d)},
    \end{align*}
    where the constant $C(Q,d) \geq \#\{z \in \Z^d: (\overline{\Omega} + z) \cap Q \neq \emptyset\}> 0$ is chosen large enough and independent of $\nu$ such that it is larger than the number of copies of $\overline{\Omega}$ required to cover $Q$.
    \begin{figure}
        \centering
        \begin{tikzpicture}
            \foreach \x in {-2,-1,...,2}{
                \foreach \y in {-2,-1,...,2}{
                    \draw[dotted,thin] (\x,\y) rectangle (\x+1,\y+1);
                }
            }

            \draw (0,0) rectangle (1,1);
            \draw[gray] (0.3,0.4) rectangle (1.3,1.4);
            \draw[->] (0,0) -- (0.3,0.4);

            \draw (0.97,-1.15) --++ (1.5,2) --++ (-2,1.5) --++ (-1.5,-2) --++ (2,-1.5);
            \draw[fill = Blue, opacity = 0.3] (0.97,-1.15) --++ (1.2,1.6) --++ (-2,1.5) --++ (-1.2,-1.6) --++ (2,-1.5);

            \node[above] at (0.5,0) {\scriptsize $\Omega$};
            \node[above] at (0.8,0.4) {\scriptsize $\Omega + h \nu$};
            \node[above] at (0.97,-1.15) {$Q$};
        \end{tikzpicture}
        \caption{Illustration of the choice $Q$, with $Q\cap(Q-h\nu)$ highlighted in blue.}
        \label{fig:HighFreq_Cube}
    \end{figure}
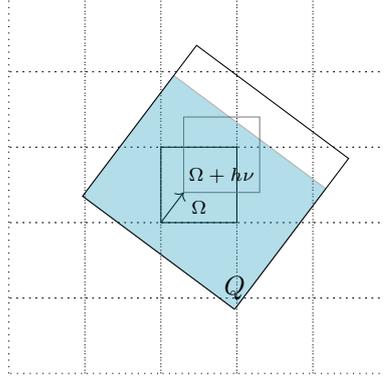

  \emph{Step 2: Fourier estimate.} Let $f \in BV_\nu(\Omega;\R^n) \cap L^\infty(\Omega;\R^n)$ be as in the statement. Then by step 1, the one-periodic extension, still denoted by $f$, satisfies
  \begin{align*}
    \int_{\T^d} \frac{|f(x+h \nu)-f(x)|^2}{|h|} dx \leq 4 \Vert f \Vert_{L^\infty} \int_{\T^d}
    \frac{|f(x+h \nu)-f(x)|}{|h|} dx \leq 4 C \Vert f \Vert_{L^\infty} \Vert D_\nu f \Vert_{TV(\T^d)}.
  \end{align*}

 With this estimate in hand, we can reformulate the difference quotient in Fourier space which turns the estimate into
  \begin{align*}
    4 C \Vert f \Vert_{L^\infty} \Vert D_\nu f \Vert_{TV(\T^d)} \geq |h|^{-1}\sum_{k \in \Z^d} |(e^{2 \pi i h k
    \cdot \nu} -1) \hat{f}(k)|^2.
  \end{align*}
  Integrating this inequality for $L < 1$ over $h \in (-L,L)$, we get
  \begin{align*}
    8 L C \Vert f \Vert_{L^\infty} \Vert D_\nu f \Vert_{TV(\T^d)}& \geq L^{-1} \int_{-L}^L \sum_{k \in \Z^d}
                                                        |e^{2\pi i h k \cdot \nu} -1|^2 |\hat{f}(k)|^2 dh \\
    &= L^{-1} \sum_{k \in \Z^d} |\hat{f}(k)|^2
    \int_{-L}^{L} |e^{2 \pi i h k \cdot \nu}-1|^2 dh.
  \end{align*}
  Calculating the integral on the right-hand side yields for $k \cdot \nu \neq 0$
  \begin{align*}
    \int_{-L}^L |e^{2 \pi i h k \cdot \nu} -1|^2 dh & = \int_{-L}^L (2-2\cos(2\pi h k \cdot \nu)) dh
                                                      = \Big[2h - \frac{2 \sin(2 \pi h k \cdot \nu)}{2
                                                      \pi k \cdot \nu}\Big]_{h = -L}^L \\
    & = 4L - \frac{4 \sin(2 \pi L k \cdot \nu)}{2 \pi k \cdot \nu} \geq 4L - \frac{2}{\pi| k \cdot \nu|}.
  \end{align*}
  Hence, after restricting the series in $k$ to $|k \cdot \nu| \geq L^{-1}$
  \begin{align*}
    8 L C \Vert f \Vert_{L^\infty} \Vert D_\nu f \Vert_{TV(\T^d)} &\geq L^{-1} \sum_{|k \cdot \nu|\geq L^{-1}} |\hat{f}(k)|^2 (4L - \frac{2}{\pi |k\cdot \nu|}) \\
    &\geq L^{-1} \sum_{|k \cdot \nu|\geq L^{-1}}|\hat{f}(k)|^2 (4L - \frac{2}{\pi L^{-1}}) \\
    &=\frac{4\pi-2}{\pi} \sum_{|k \cdot \nu|\geq L^{-1}} |\hat{f}(k)|^2.
  \end{align*}
  Setting $L = \lambda^{-1} <1$, and observing that $\Vert D_\nu f \Vert_{TV(\T^d)} \leq \Vert D_\nu f \Vert_{TV(\Omega)} + \Vert f \Vert_{L^\infty}  \Per(\Omega)$, shows the statement.
\end{proof}

\subsection{Fourier localization}
In this section we recall some of the relevant tools to derive our Fourier-based lower bounds in \cref{sec:intro_L1}. They build on the strategy from \cite{RT23}, see also \cite{CO1,KW16}.
For this we introduce the following truncated cone for $j \in \{1,2,\dots,d\}$ and $\mu, \lambda > 0$
\begin{align} \label{eq:TruncCone_1}
    C_{j,\mu,\lambda} &:= \{k \in \Z^d: |k|^2 - k_{j}^2 \leq \mu^2 |k|^2, \ |k_{j}| \leq \frac{2}{|\nu_1|} \lambda \},
\end{align}
and the following (infinitely extended) cones 
\begin{align} \label{eq:Cone_j}
    C_{j,\mu} &:= \{k \in \Z^d: |k|^2 - k_j^2 \leq \mu^2 |k|^2\}.
\end{align}
For these cones let $m_{j,\mu,\lambda}(D), m_{j,\mu}(D)$ be the corresponding Fourier multipliers acting on $f \in L^2(\T^d;\R)$ as
\begin{align*}
    m_{j,\mu,\lambda}(D) f(x) & = \sum_{k \in \Z^d} m_{j,\mu,\lambda}(k) \hat{f}(k) e^{2 \pi i k \cdot x}.
\end{align*}
They are given (for instance) by
\begin{align} \label{eq:multiplier}
    \begin{split}
    m_{j,\mu,\lambda}(k) &= \big(1-\varphi(2|k|)\big)\varphi\Big(\frac{\sqrt{|k|^2-k_{j}^2}}{\mu |k|} \Big) \varphi\Big(\frac{|\nu_1||k_{j}|}{2 \lambda}\Big) + \varphi(2|k|) \in C^\infty(\R^d), \\
    m_{j,\mu}(k) &= \big(1-\varphi(2|k|)\big)\varphi\Big(\frac{\sqrt{|k|^2-k_j^2}}{\mu |k|} \Big)  + \varphi(2|k|) \in C^\infty(\R^d),
    \end{split}
\end{align}
where
\begin{align*}
        \varphi \in C^\infty([0,+\infty);[0,1]), \quad \varphi(x) = 1,\varphi'(x)\le0, x \in [0,1], \quad \varphi(x) = 0, x \notin (0,2).
    \end{align*}
    Thus the multipliers satisfy for $k \in \Z^d$
    \begin{align*}
        m_{j,\mu,\lambda}(k) & = 1, \ k \in C_{j,\mu,\lambda}, & m_{j,\mu,\lambda}(k) & = 0, \ k \notin C_{j,2\mu,2\lambda}, \\
        m_{j,\mu}(k) & = 1, \ k \in C_{j,\mu} & m_{j,\mu}(k) & =0, \ k \notin C_{j,2\mu}.
    \end{align*}
    With this choice of multipliers we are able to apply Marcinkiewicz's multiplier theorem \cite[Cor. 6.2.5]{Grafakos} (combined with the transference principle \cite[Thm. 4.3.7]{Grafakos}).
    These multipliers have been considered in the context of shape-memory alloys in \cite{RT22,RT23}.

We start by giving a Fourier interpretation of the elastic energy. For notational convenience, here and in what follows below, we will use the notation $\R^{d \times d}_{\diag}$ to denote the diagonal matrices.

\begin{lem} \label{lem:ElasticFourier}
Let $d \geq 2$ and $\Omega = (0,1)^d$.
For $u \in H^1_0(\Omega;\R^d)$ and $\chi \in L^2(\Omega;\R^{d \times d}_{\diag})$ let the elastic energy $E_{el}(u,\chi)$ be given by
\begin{align*}
    E_{el}(u,\chi) = \int_\Omega |\nabla u - \chi|^2 dx.
\end{align*}
Then,
\begin{align*}
    E_{el}(u,\chi) \geq |\hat{\chi}(0)|^2 + \sum_{k \in \Z^d \setminus \{0\}} \sum_{j =1}^d \frac{|k|^2-k_j^2}{|k|^2} |\hat{\chi}_{jj}(k)|^2,
\end{align*}
    where we consider the one-periodic extensions of $u$ and $\chi$ without changing the notation.
\end{lem}

\begin{proof}
    We consider the one-periodic extensions of $u$ and $\chi$ and switch to Fourier space
    \begin{align*}
        E_{el}(u,\chi) =|\hat{\chi}(0)|^2 + \sum_{k \in \Z^d \setminus \{0\}} |2 \pi i \hat{u} \otimes k - \hat{\chi}|^2.
    \end{align*}
    For any given $\chi$, we solve the related Euler-Lagrange equation.
    That is, we choose $\hat{u}(k)$ such that it fulfils
     \begin{align*}
     2 \pi i (\hat{u} \otimes k ) k = \hat{\chi} k, \quad k \neq 0.
     \end{align*}
     Plugging this choice of $\hat{u}(k)$ into the above Fourier representation and using that $\chi$ is diagonal then gives the desired result.
\end{proof}

Combining \cref{lem:HighFreqDir,lem:ElasticFourier}, we can deduce a first Fourier localization argument.
\begin{lem} \label{lem:Localization}
Let $d \geq 2$, $\Omega = (0,1)^d$, and $\nu \in \S^{d-1}$ be such that $\nu \cdot e_1 \neq 0$.
Consider the energy to be given for $\epsilon > 0$ and $\chi \in BV_\nu(\Omega;\R^{d \times d}_{\diag})\cap L^\infty(\Omega;\R^{d\times d}_{\diag})$ by
\begin{align*}
    E_\epsilon(\chi) = \inf_{u \in H^1_0(\Omega;\R^d)} \int_\Omega |\nabla u - \chi|^2 dx + \epsilon \Vert D_\nu \chi \Vert_{TV(\Omega)}.
\end{align*}
Extending $\chi$ one-periodically and
considering the cones $C_{1,\mu,\lambda}$ and $C_{j,\mu}$ for $\mu \in (0,\frac{|\nu_1|}{2})$, $\lambda >0$, and $j=2,3,\dots,d$ defined in \cref{eq:TruncCone_1,eq:Cone_j}
it holds
\begin{align*}
    \sum_{k \notin C_{1,\mu,\lambda}} |\hat{\chi}_{11}(k)|^2 + \sum_{j=2}^d \sum_{k \notin C_{j,\mu}} |\hat{\chi}_{jj}(k)|^2 \leq C \left( \mu^{-2} + \Vert \chi \Vert_{L^\infty} (\lambda \epsilon)^{-1} \right) E_\epsilon(\chi) + C \Vert \chi \Vert_{L^\infty}^2 \lambda^{-1} \Per(\Omega),
\end{align*}
where $C = C(d) > 0$ is a constant independent of $\nu,\mu,\lambda$ and $\chi$.
\end{lem}

\begin{proof}
    By \cref{lem:ElasticFourier}, we get
    \begin{align*}
        \inf_{u \in H^{1}_0(\Omega;\R^d)} \int_\Omega |\nabla u - \chi|^2 dx \geq \sum_{j=1}^d \sum_{k \in \Z^d} \frac{|k|^2-k_j^2}{|k|^2} |\hat{\chi}_{jj}(k)|^2,
    \end{align*}
    where we set the multiplier to be equal to one for $k =0$.
    For $j \in \{1,\dots,d\}$ and $k \notin C_{j,\mu}$ it therefore holds
    \begin{align*}
        \frac{|k|^2-k_j^2}{|k|^2} > \mu^2,
    \end{align*}
    and thus
    \begin{align} \label{eq:EllipticityEst}
        \sum_{j=1}^d \sum_{k \notin C_{j,\mu}} |\hat{\chi}_{jj}(k)|^2 \leq \mu^{-2} \sum_{j=1}^d \sum_{k \in \Z^d} \frac{|k|^2-k_j^2}{|k|^2} |\hat{\chi}_{jj}(k)|^2 \leq \mu^{-2} \inf_{u \in H^1_0(\Omega;\R^d)} \int_\Omega |\nabla u - \chi|^2 dx.
    \end{align}

    To improve the estimate on $\chi_{11}$ to a bound outside of the truncated cone $C_{1,\mu,\lambda}$ instead of in the complement of the infinitely extended cone $C_{1,\mu}$, we note that due to the assumptions that $\nu \cdot e_1 \neq 0$ and $\mu < \frac{|\nu_1|}{2}$
    \begin{align*}
        \{k \in \Z^d: |k|^2 - k_1^2 \leq \mu^2 |k|^2, |k \cdot \nu| \leq \lambda \} \subset C_{1,\mu,\lambda}.
    \end{align*}
    Indeed, for $k=(k_1,k')$ and $\nu = (\nu_1,\nu') \neq \pm e_1$, using that $\mu < \frac{|\nu_1|}{2}$, we have $\mu^2/(1-\mu^2) \leq |\nu_1|^2/(4 |\nu'|^2)$.
    Hence, for $k$ such that $|k \cdot \nu| \leq \lambda$ and $|k'|^2=|k|^2 -k_1^2 \leq \mu^2 |k|^2$, by exploiting the reverse triangle inequality we observe that
    \begin{align} \label{eq:ineq_inclusion_cones}
    |k\cdot\nu|=|k_1\nu_1+k'\cdot\nu'|\ge|k_1||\nu_1|-\mu|k||\nu'|\ge|k_1|\Big(|\nu_1|-\frac{\mu|\nu'|}{\sqrt{1-\mu^2}}\Big)\ge|k_1|\frac{|\nu_1|}{2}.
    \end{align}
    Hence, it holds that
    \begin{align*}
        |k_1| \leq \frac{2}{|\nu_1|} \lambda.
    \end{align*}
    Note, that here it is of importance that $\nu_1 \neq 0$. 

    By \cref{eq:EllipticityEst,lem:HighFreqDir} it follows that
    \begin{align*}
        \sum_{k \notin C_{1,\mu,\lambda}} |\hat{\chi}_{11}(k)|^2 & \leq \sum_{k\notin C_{1,\mu}} |\hat{\chi}_{11}(k)|^2 + \sum_{|k \cdot \nu| > \lambda} |\hat{\chi}_{11}(k)|^2 \\
        & \leq \mu^{-2} \sum_{k \in \Z^d} \frac{|k|^2-k_1^2}{|k|^2} |\hat{\chi}_{11}(k)|^2 \\
        & \quad + C(d) \lambda^{-1} \Vert \chi_{11} \Vert_{L^\infty} \big( \Vert D_\nu \chi_{11} \Vert_{TV(\Omega)} + \Vert \chi_{11} \Vert_{L^\infty} \Per(\Omega) \big).
    \end{align*}

    Summing this estimate and \cref{eq:EllipticityEst} for $j=2,3,\dots,d$, we get for $\chi \in L^\infty(\Omega;\R^{d\times d}_{\text{diag}})$
    \begin{align*}
        \sum_{k \notin C_{1,\mu,\lambda}} |\hat{\chi}_{11}(k)|^2 + \sum_{j=2}^d \sum_{k \notin C_{j,\mu}} |\hat{\chi}_{jj}(k)|^2
        &\leq C(d)(\mu^{-2} + \Vert \chi \Vert_{L^\infty} (\lambda \epsilon)^{-1}) E_\epsilon(\chi) \\
        & \quad + C(d) \Vert \chi \Vert_{L^\infty}^2 \lambda^{-1} \Per(\Omega).
    \end{align*}
\end{proof}

With this result we have a combination of an ellipticity estimate of the form \cref{eq:EllipticityEst} and a high frequency control in one direction using \cref{lem:HighFreqDir}. The next result provides a low frequency control, with similar methods as in \cite[Lem. 4.2]{RT23}.

\begin{lem} \label{lem:LowFreqEst}
    Assume that the same conditions hold as in \cref{lem:ElasticFourier} and further assume that $\inf_{u \in H^1_0(\Omega;\R^d)} E_{el}(u,\chi) > 0$. Viewing $\chi \in L^2(\Omega;\R^{d \times d}_{\diag})$ as a function on $\T^d$, there is a constant $C > 0$ such that for any $\bar{\lambda}> 1$
    \begin{align*}
        \sum_{|k_1| \leq \bar{\lambda}} |\hat{\chi}_{11}(k)|^2 \leq C \bar{\lambda}^2 \inf_{u \in H^1_0(\Omega;\R^d)} E_{el}(u,\chi).
    \end{align*}
\end{lem}
\begin{proof}
    The proof relies on the diagonal structure of $\chi$ and the zero boundary data which we impose on the functions $u$.
    For $u \in H^1_0(\Omega;\R^d)$, seen as a function on $\T^d$, it holds by Poincar\'e's inequality in $x_2$ and as $\chi_{12} = 0$
    \begin{align*}
        \sum_{|k_1| \leq \bar{\lambda}} |\hat{u}_1(k)|^2 & \leq \int_{\Omega} |u_1(x)|^2 dx \le \int_\Omega |\p_2 u_1(x)|^2 dx \\
        & = \int_\Omega |\p_2 u_1 - \chi_{12}|^2 dx \leq E_{el}(u,\chi).
    \end{align*}
    Choosing $v \in H^1_0(\Omega;\R^d)$ such that $E_{el}(v,\chi) \leq 2 \inf_{u \in H^1_0(\Omega;\R^d)} E_{el}(u,\chi)$, where we use that $\inf_{u \in H^1_0(\Omega;\R^d)} E_{el}(u,\chi) > 0$, implies for $\bar{\lambda} > 1$
    \begin{align*}
        \sum_{|k_1| \leq \bar{\lambda}} |\hat{\chi}_{11}(k)|^2 & \leq 2 \sum_{|k_1| \leq \bar{\lambda}} \Big( |\widehat{\p_1 v_1}(k) - \hat{\chi}_{11}(k)|^2 + |2 \pi i k_1 \hat{v}_1(k)|^2\Big) \\
        & \leq 2 E_{el}(v,\chi) + 8 \pi^2 \bar{\lambda}^2 \sum_{|k_1| \leq \bar{\lambda}} |\hat{v}_{1}(k)|^2 \\
        & \leq C \bar{\lambda}^2 E_{el}(v,\chi) \leq 2 C \bar{\lambda}^2 \inf_{u \in H^1_0(\Omega;\R^d)} E_{el}(u,\chi)
    \end{align*}
    which yields the result.
\end{proof}

As a final ingredient, for more than two wells, we rely on a commutator estimate. With this we can use the truncation of one cone, cf. \cref{lem:Localization}, and carry it over to another one and reduce its size in the process.
This gives rise to an iterative procedure, where the number of iterations determines the scaling in \cref{thm:L1_3wells,thm:L1_mult_wells}.

\begin{lem}[{\cite[Lem. 3]{RT22}, \cite[Prop. 4.6]{RT23}}] \label{lem:CommutatorEst}
    Let $d \geq 2$, $\Omega = (0,1)^d$, $\nu \in \S^{d-1}$ with $\nu \cdot e_1 \neq 0$.
    For $\mu \in (0,\frac{|\nu_1|}{2})$ and $\lambda > 0$,
    let $m_{1,\mu,\lambda}(D)$ be given by \cref{eq:multiplier}.
    Moreover, for $t > 0$ let $\psi_t(x) = \max\{|x|,|x|^t\}$.
    Let $f_1, f_2 \in BV_\nu(\Omega;\R) \cap L^\infty(\Omega;\R)$ and let $g: \R \to \R$ be a polynomial of degree two with $f_2 = g(f_1)$. Then for any $\gamma \in (0,1)$ there is a constant $C = C(g,\gamma,\Vert f_1 \Vert_{L^\infty}) >0$ such that
    \begin{align*}
        \Vert f_2 - g(m_{1,\mu,\lambda}(D) f_1) \Vert_{L^2} \leq C \psi_{1-\gamma}\left( \Vert f_1 - m_{1,\mu,\lambda}(D) f_1 \Vert_{L^2} \right).
    \end{align*}
\end{lem}

\begin{proof}
    By the triangle inequality we can assume without loss of generality that $g(x) = x^2$.
    Invoking Hölder's inequality, we get
    \begin{align*}
        \Vert f_2 - g(m_{1,\mu,\lambda}(D) f_1) \Vert_{L^2} & = \left\Vert f_1^2 - \left( m_{1,\mu,\lambda}(D) f_1 \right)^2 \right\Vert_{L^2} \\
        & \leq \Vert f_1 - m_{1,\mu,\lambda}(D) f_1 \Vert_{L^{2+2\gamma}} \ \Vert f_1 + m_{1,\mu,\lambda}(D) f_1 \Vert_{L^{\frac{2+2\gamma}{\gamma}}}.
    \end{align*}
    By virtue of the interpolation inequality of $L^p$ spaces and the $L^p$-$L^p$ multiplier bounds from Marcinkiewicz's theorem and the transference principle \cite[Cor. 6.2.5, Thm. 4.3.7]{Grafakos}, we get the desired estimate. For details we refer to \cite[Prop. 4.6]{RT23}.
\end{proof}

\begin{cor} \label{cor:CommutatorEst_Applied}
Let $d \geq 2$, $\Omega = (0,1)^d$, $\nu \in \S^{d-1}$ with $\nu \cdot e_1 \neq 0$, $\lambda > 0$, and $\mu \in (0,\frac{|\nu_1|}{16})$.
Let $\chi\in BV_\nu(\Omega;\R^{d\times d}_\diag)\cap L^\infty(\Omega;\R^{d\times d}_\diag)$, extended one-periodically, and suppose that
    \begin{align*}
        \sum_{j=2}^d \alpha_j \chi_{jj} = g(\chi_{11}),
    \end{align*}
    for a polynomial $g: \R \to \R$ of degree two and coefficients $\alpha_j \in \R$.
    Let $m_{j,\mu,\lambda}$ and $m_{j,\mu}$ be as in \cref{eq:multiplier}.
    Then for any $\gamma \in (0,1)$ there is $C = C(d,g,\gamma, \Vert \chi_{11} \Vert_{L^\infty})>0$ such that
    \begin{align} \label{eq:CommutatorEst_Applied_1}
    \begin{split}
        \Big\Vert \sum_{j=2}^d \alpha_j m_{j,\mu}(D) \chi_{jj} - g(m_{1,\mu,\lambda}(D) \chi_{11}) \Big\Vert_{L^2}
        &\leq C \psi_{1-\gamma}(\Vert \chi_{11} - m_{1,\mu,\lambda}(D) \chi_{11} \Vert_{L^2})\\
        &\quad + \sum_{j=2}^d |\alpha_j| \Vert \chi_{jj} - m_{j,\mu}(D) \chi_{jj} \Vert_{L^2}.
        \end{split}
    \end{align}
    Moreover, there is $M >0$ such that for $\lambda_2 = M \mu \lambda < \lambda$ it holds that
    \begin{align*}
        \sum_{j=2}^d |\alpha_j| \Vert \chi_{jj} - m_{j,\mu,\lambda_2}(D) \chi_{jj} \Vert_{L^2}
      & \leq C \psi_{1-\gamma}(\Vert \chi_{11} - m_{1,\mu,\lambda}(D) \chi_{11} \Vert_{L^2}) \\
      & \quad + 2 \sum_{j=2}^d |\alpha_j| \Vert \chi_{jj} - m_{j,\mu}(D) \chi_{jj} \Vert_{L^2}.
    \end{align*}
\end{cor}

\begin{proof}
    The first statement is a direct consequence of \cref{lem:CommutatorEst} and the triangle inequality:
    \begin{align*}
    &\Big\Vert \sum_{j=2}^d \alpha_j m_{j,\mu}(D) \chi_{jj} - g(m_{1,\mu,\lambda}(D) \chi_{11}) \Big\Vert_{L^2} \\
    &\qquad \leq \Big\Vert \sum_{j=2}^d \alpha_j m_{j,\mu}(D) \chi_{jj} - \sum_{j=2}^d \alpha_j  \chi_{jj} \Big\Vert_{L^2}
         + \Big\Vert \sum_{j=2}^d \alpha_j  \chi_{jj} - g(m_{1,\mu,\lambda}(D) \chi_{11}) \Big\Vert_{L^2} \\
    & \qquad \leq \sum_{j=2}^d |\alpha_j| \Vert \chi_{jj} - m_{j,\mu}(D) \chi_{jj} \Vert_{L^2} + \Vert f_2 - g(m_{1,\mu,\lambda}(D) f_1) \Vert_{L^2},
    \end{align*}
    with $f_1 = \chi_{11}$ and $f_2 = \sum_{j=2}^d \alpha_j \chi_{jj}$.

    To see the second claim, we note that,
    since the support of $\mathcal{F}[g(m_{1,\mu,\lambda}(D) \chi_{11})]$ is contained in $C_{1,2\mu,2\lambda}+C_{1,2\mu,2\lambda}$ (in the sense of the Minkowski sum),
    there is an $M > 0$, independent of $\mu$, $\lambda$, $\nu$, such that with $\lambda_2 = M \mu \lambda$
    \begin{align} \label{eq:convolution_supp}
        \mathcal{F}[g(m_{1,\mu,\lambda}(D) \chi_{11})](k) = 0, \quad \text{for } |k_j| > \frac{2}{|\nu_1|} \lambda_2, \text{ for any } j \in \{2,\dots,d\}.
    \end{align}
    Furthermore, we use that $|m_{j,\mu}(k) - m_{j,\mu,\lambda_2}(k)| \leq \chi_{\{k: |k_j| \geq \frac{2}{|\nu_1|} \lambda_2\}}(k) m_{j,\mu}(k)$, thus
    \begin{align*}
        \Vert m_{j,\mu}(D) \chi_{jj} - m_{j,\mu,\lambda_2}(D) \chi_{jj} \Vert_{L^2} & \leq \Vert \chi_{\{k: |k_j| \geq \frac{2}{|\nu_1|} \lambda_2\}}(D) m_{j,\mu}(D) \chi_{jj} \Vert_{L^2},
    \end{align*}
    and after an application of the triangle inequality
    \begin{align} \label{eq:improving_truncation}
    \begin{split}
        \Vert \chi_{jj} - m_{j,\mu,\lambda_2}(D) \chi_{jj} \Vert_{L^2} & \leq \Vert \chi_{jj} - m_{j,\mu}(D) \chi_{jj} \Vert_{L^2} \\
        & \quad + \Vert \chi_{\{k: |k_j| \geq \frac{2}{|\nu_1|} \lambda_2\}}(D) m_{j,\mu}(D) \chi_{jj} \Vert_{L^2}.
        \end{split}
    \end{align}
    Here, with slight abuse of notation, we define $\chi_{\{k: |k_j| \geq \frac{2}{|\nu_1|} \lambda_2\}}(D)$ as the Fourier multiplier associated with the function $\chi_{\{k: |k_j| \geq \frac{2}{|\nu_1|} \lambda_2\}}(k) $ in Fourier space. 

    Using the fact that as $\mu < \frac{1}{2\sqrt{2}}$ we have $m_{j,\mu}(k) m_{\ell,\mu}(k) = \delta_0(k)$ for $j \neq \ell$, we see that $m_{j,\mu}(D) \chi_{jj}$ and $m_{\ell,\mu}(D) \chi_{\ell \ell}$ have disjoint Fourier support away from zero.
    In particular the functions $\alpha_j \chi_{\{|k_j| \geq 2 |\nu_1|^{-1} \lambda_2\}}(D)m_{j,\mu}(D) \chi_{jj}$ have pairwise disjoint Fourier support, hence, after summing \cref{eq:improving_truncation} over $j = 2, \dots, d$ with the weights $|\alpha_j|$, we get
    \begin{align}
    \label{eq:cor_comm1}
    \begin{split}
        \sum_{j=2}^d |\alpha_j| \Vert \chi_{jj} - m_{j,\mu,\lambda_2}(D) \chi_{jj} \Vert_{L^2} & \leq \sum_{j=2}^d |\alpha_j| \Vert \chi_{jj} - m_{j,\mu}(D) \chi_{jj} \Vert_{L^2}\\
        & \quad + C(d) \Big\Vert \sum_{j=2}^d \alpha_j \chi_{\{|k_j| \geq \frac{2}{|\nu_1|} \lambda_2\}}(D) m_{j,\mu}(D) \chi_{jj} \Big\Vert_{L^2}.
        \end{split}
    \end{align}

    In the following, we write $k = (k_1,k')$ with $k' = (k_2,\dots,k_d) \in \Z^{d-1}$ and also use $|k'|_\infty = \max\{|k_2|,\dots,|k_d|\}$.
    Since $\mu < \frac{1}{4}$, we remark that $\chi_{\{|k_j| \geq \frac{2}{|\nu_1|} \lambda_2\}}(k) m_{j,\mu}(k) = \chi_{\{|k'|_\infty \geq \frac{2}{|\nu_1|} \lambda_2\}}(k) m_{j,\mu}(k)$ for $j=2,\dots,d$ and moreover by \cref{eq:convolution_supp}, we can further control the second term in \cref{eq:cor_comm1} as follows
    \begin{align}
    \label{eq:cor_comm2}
    \begin{split}
        & \Big\Vert \sum_{j=2}^d \alpha_j \chi_{\{|k_j| \geq \frac{2}{|\nu_1|} \lambda_2 \}}(D) m_{j,\mu}(D) \chi_{jj} \Big\Vert_{L^2} \\
        &\qquad = \Big\Vert \chi_{\{|k'|_\infty \geq \frac{2}{|\nu_1|} \lambda_2 \}}(D) \Big( \sum_{j=2}^d \alpha_j m_{j,\mu}(D) \chi_{jj} - g(m_{1,\mu,\lambda}(D) \chi_{11}) \Big) \Big\Vert_{L^2} \\
        &\qquad \leq \Big\Vert \sum_{j=2}^d \alpha_j m_{j,\mu}(D) \chi_{jj} - g(m_{1,\mu,\lambda}(D) \chi_{11}) \Big\Vert_{L^2}.
    \end{split}
    \end{align}

    Thus, gathering \cref{eq:cor_comm1}, \cref{eq:cor_comm2} and the first statement \cref{eq:CommutatorEst_Applied_1}, we can conclude
    \begin{align*}
        \sum_{j=2}^d |\alpha_j| \Vert \chi_{jj} - m_{j,\mu,\lambda_2}(D) \chi_{jj} \Vert_{L^2} & \leq (C(d)+1) \sum_{j=2}^d |\alpha_j| \Vert \chi_{jj} - m_{j,\mu}(D) \chi_{jj} \Vert_{L^2} \\
        & \quad + C(d) \psi_{1-\gamma}(\Vert \chi_{11}- m_{1,\mu,\lambda}(D) \chi_{11} \Vert_{L^2}).
    \end{align*}
\end{proof}

Let us compare the use of the nonlinear relation with the one from \cite{RT23}. In \cite{RT23} the nonlinear relation is exploited in the form $\chi_{kk} = g(\sum_{j \neq k} \alpha_j \chi_{jj})$. This strategy hence requires working with truncated cones in \emph{all but one} direction. In contrast, in the formulation of \cref{cor:CommutatorEst_Applied}, since the polynomial relation is inverted, it is clear why there is only a single truncation necessary in order to ``propagate'' the truncation to the other diagonal entries.

\begin{rmk}\label{rmk:corollary-variant}
In the sequel, when dealing with higher order laminates, we need to iterate the cone localization from above.
For this reason, we will exploit the following variant of \cref{cor:CommutatorEst_Applied}, whose proof is identical.

In the context of \cref{cor:CommutatorEst_Applied}, let $1\le k\le d$ be the component we want to (further) localize and let $\lambda_k>0$ denote the localization length of $\chi_{kk}$.
If $\chi_{kk}$ has the nonlinear relation
$$
\sum_{j=k+1}^d\alpha_j\chi_{jj}=g(\chi_{kk})
$$
then, for any $\gamma\in(0,1)$ and given $\lambda_{k+1}=M\mu\lambda_k$ it holds that
\begin{align*}
\sum_{j=k+1}^d|\alpha_j| \|\chi_{jj}-m_{j,\mu,\lambda_{k+1}}(D)\chi_{jj}\|_{L^2} &\le C\psi_{1-\gamma}\big(\|\chi_{kk}-m_{k,\mu,\lambda_k}(D)\chi_{kk}\|_{L^2}\big) \\
&\qquad +2\sum_{j=k+1}^d|\alpha_j|\|\chi_{jj}-m_{j,\mu}(D)\chi_{jj}\|_{L^2}.
\end{align*}

We will also apply \cref{cor:CommutatorEst_Applied} and the above variant by replacing $\nu\cdot e_1$ with $\nu\cdot e_n$ for some $1\le n\le d$, with the definitions of $C_{j,\mu,\lambda}, C_{j,\mu}, m_{j,\mu,\lambda}$ and $m_{j,\mu}$ modified accordingly.
\end{rmk}

\subsection{Fractional surface energies in $L^2$-based settings}
\label{sec:prelim_frac}

For the proofs of the lower bound of \cref{thm:L1_3wells,thm:L1_mult_wells} we invoke Fourier methods, relying on \cref{lem:HighFreqDir} for the surface energy.
Therefore we can also replace the sharp surface energy by a fractional $L^2$-based surface energy, giving rise to a similar high frequency control as in \cref{lem:HighFreqDir}.

By extending $\chi \in L^\infty((0,1)^d;\R^{d\times d}_\diag)$ one-periodically, we consider for $s \in (0,\frac{1}{2})$
\begin{align*}
    E_{surf}^s(\chi) := \left( \sum_{k \in \Z^d\setminus\{0\}} |k \cdot \nu|^{2s} |\hat{\chi}(k)|^2 \right)^{\frac{1}{2s}}.
\end{align*}
It then follows directly that it holds
\begin{align*}
   \left( \sum_{k \in \Z^d: |k \cdot \nu| \geq \lambda} |\hat{\chi}(k)|^2 \right)^{\frac{1}{2s}} \leq \lambda^{-1} E_{surf}^s(\chi).
\end{align*}

Based on this observation, we have an analogous result as in \cref{lem:Localization}.
\begin{lem} \label{lem:LocalizationFrac}
    Let $d \geq 2$, $\Omega = (0,1)^d$, and $\nu \in \S^{d-1}$ such that $\nu \cdot e_1 \neq 0$.
    Consider for $\epsilon > 0$, $s \in (0,\frac{1}{2})$ and $\chi \in H^s_\nu(\Omega;\R^{d\times d}_{\diag})$, cf. \cref{eq:anisotropic_Hs}, the following energy
    \begin{align*}
        E_{\epsilon,s}(\chi) = \inf_{u \in H^{1}_0(\Omega;\R^d)} \int_{\Omega} |\nabla u - \chi|^2 dx + \epsilon E_{surf}^s(\chi).
    \end{align*}
    Then, it holds
    \begin{align*}
        \sum_{k \not\in C_{1,\mu,\lambda}} |\hat{\chi}_{11}(k)|^2 + \sum_{j=2}^d \sum_{k \notin C_{j,\mu}} |\hat{\chi}_{jj}(k)|^2 \leq \psi_{2s}\big( (\mu^{-2} + (\lambda \epsilon)^{-1}) E_{\epsilon,s}(\chi) \big),
    \end{align*}
    where, for $\mu \in (0,\frac{|\nu_1|}{2})$, $\lambda > 0$, the cones $C_{1,\mu,\lambda},C_{j,\mu}$ are given as in \cref{eq:TruncCone_1,eq:Cone_j}, and $\psi_t(x) = \max\{|x|, |x|^t\}$ for $t > 0$.
\end{lem}
\begin{proof}
    We argue as in the proof of \cref{lem:Localization}.
    By \cref{lem:ElasticFourier} we have
    \begin{align*}
        E_{\epsilon,s}(\chi) \geq \sum_{j=1}^d \sum_{k \in \Z^d} \frac{|k|^2-k_j^2}{|k|^2} |\hat{\chi}_{jj}(k)|^2 + \epsilon \left( \sum_{k \in \Z^d \setminus \{0\}} |k \cdot \nu|^{2s} |\hat{\chi}(k)|^2 \right)^{\frac{1}{2s}}.
    \end{align*}
    Consequently, for $j \in \{1,\dots,d\}$ and $k \notin C_{j,\mu}$ we have $|k|^2 - k_j^2 \geq \mu^2 |k|^2$ and, hence,
    \begin{align*}
        \sum_{j=1}^d\sum_{k \notin C_{j,\mu}} |\hat{\chi}_{jj}(k)|^2 \leq \mu^{-2} \sum_{j=1}^d \sum_{k \in \Z^d} \frac{|k|^2-k_j^2}{|k|^2} |\hat{\chi}_{jj}(k)|^2 \leq \mu^{-2} E_{\epsilon,s}(\chi).
    \end{align*}

    For the surface energy it is immediate that it holds
    \begin{align*}
        \sum_{|k \cdot \nu| \geq \lambda} |\hat{\chi}(k)|^2 \leq \lambda^{-2s} \sum_{k \in \Z^d \setminus \{0\}} |k \cdot \nu|^{2s} |\hat{\chi}(k)|^2 \leq (\lambda \epsilon)^{-2s} E_{\epsilon,s}(\chi)^{2s}.
    \end{align*}
    Moreover, if we combine both the above estimates 
    we arrive at
    \begin{align*}
        \sum_{k \notin C_{1,\mu,\lambda}} |\hat{\chi}_{11}(k)|^2 & + \sum_{j=2}^{d}\sum_{k \notin C_{j,\mu}} |\hat{\chi}_{jj}(k)|^2 \leq \sum_{j=1}^d \sum_{k \notin C_{j,\mu}} |\hat{\chi}_{jj}(k)|^2 + \sum_{|k \cdot \nu| \geq \lambda} |\hat{\chi}_{11}(k)|^2 \\
        & \leq \mu^{-2} E_{\epsilon,s}(\chi) + (\lambda \epsilon)^{-2s} E_{\epsilon,s}(\chi)^{2s} \\
        & \leq 2\max\{ (\mu^{-2} + (\lambda \epsilon)^{-1}) E_{\epsilon,s}(\chi), \big((\mu^{-2} + (\lambda \epsilon)^{-1}) E_{\epsilon,s}(\chi)\big)^{2s}\} \\
        & = 2\psi_{2s}((\mu^{-2} + (\lambda \epsilon)^{-1})E_{\epsilon,s}(\chi))
    \end{align*}
    and the result is proven.
\end{proof}

As we will see below in \cref{sec:L1AnisotropicSurf} an estimate of this form is sufficient to deduce the lower scaling estimates, consequently, we will be able to generalize \cref{thm:L1_mult_wells} to fractional surface energies as stated in \cref{thm:frac_energy}.

\section{Sharp surface energies -- Proofs of \cref{thm:L1_3wells,thm:L1_mult_wells}}
\label{sec:L1AnisotropicSurf}

In this section we focus on the sharp anisotropic surface energies introduced in \cref{sec:intro_L1} and seek to identify minimal assumptions on the surface energy in order to ensure the same scaling as with isotropic surface energy penalizations.

\subsection{Two-well problem}

As motivation, we begin by considering a simple two-well gradient inclusion $\nabla u \in \{0,e_1 \otimes e_1\}$.
In the standard models in the literature a quantification of this problem is most often considered with an isotropic surface penalization.
There are also instances where only the oscillation in $e_1$ direction is penalized, as for instance in the seminal works \cite{KM1,KM2,C1}. As we view the two-well problem as a prototypical model set-up which we will then, in the following sections, generalize to more complex microstructures, we briefly present the proof of \cref{prop:L1_2wells}.

We use similar Fourier methods as in \cite{RT22, RT23} and argue in three steps.
We first give the arguments for the lower and upper bounds in the case of $\nu \cdot e_1 \neq 0$, afterwards we present the upper bound construction for $\nu \cdot e_1 = 0$.
With this we have a full characterisation of which directions are required in the surface energy to have the same scaling as the isotropic surface penalization.

\begin{proof}[Proof of \cref{prop:L1_2wells} for $\nu \cdot e_1 \neq 0$]

  \emph{Step 1: Lower bound.}
  We first assume that $F_\alpha = 0$ and $A-B = e_1 \otimes e_1$,  namely $A=(1-\alpha)e_1\otimes e_1$, $B=-\alpha e_1\otimes e_1$.
  The general case will be recovered at the end of the proof.
  We define for $\chi \in BV_\nu(\Omega;\{A,B\})$
  \begin{align*}
      E_\epsilon(\chi):= \inf_{u \in \mathcal{A}_{F_\alpha}} E_\epsilon(u,\chi) = \inf_{u \in \mathcal{A}_{F_\alpha}} \int_{\Omega} |\nabla u - \chi|^2 dx + \epsilon \Vert D_\nu \chi\Vert_{TV(\Omega)}.
  \end{align*}
  Thus, after
  extending $\chi$ and $\nabla u$ one-periodically, we can apply \cref{lem:ElasticFourier} to obtain
  \begin{align*}
    \int_{\Omega} |\nabla u - \chi|^2 dx \geq \sum_{k \in \Z^d \setminus \{0\}}
    \frac{|k|^2-k_1^2}{|k|^2} |\hat{f}|^2 + |\hat{f}(0)|^2,
  \end{align*}
  where we wrote $f = \chi_{11} \in BV_\nu(\Omega;\{1-\alpha,-\alpha\})$.

\Cref{lem:Localization} then implies for $\mu \in (0,\frac{|\nu_1|}{2})$, $\lambda >1$
\begin{align*}
    \sum_{k \notin C_{1,\mu,\lambda}} |\hat{f}(k)|^2 \leq C(d,\alpha) \Big( (\mu^{-2} + (\lambda \epsilon)^{-1}) E_\epsilon(\chi) + \lambda^{-1} \Per(\Omega) \Big),
\end{align*}
where the truncated cone $C_{1,\mu,\lambda}$ is defined in \cref{eq:TruncCone_1}.

Moreover, by an application of \cref{lem:LowFreqEst} for $\bar{\lambda} = \frac{2}{|\nu_1|} \lambda > 1$, for $\lambda > 1 > \frac{|\nu_1|}{2}$,
\begin{align}\label{eq:LowFreqEst_2wells}
    \sum_{|k_1| \leq 2 \lambda/|\nu_1|} |\hat{f}(k)|^2 \leq C \frac{\lambda^2}{|\nu_1|^2} E_\epsilon(\chi).
\end{align}
With this, we have control over the Fourier mass of $f$ in the whole space partitioned into $\Z^d =(C_{1,\mu,\lambda}^c\cap\Z^d)\cup \{k \in \Z^d :|k_1| \leq \frac{2}{|\nu_1|} \lambda\}$, cf. \cref{fig:TwoWellFourier}, as follows
\begin{align*}
    \sum_{k \in \Z^d} |\hat{f}(k)|^2 & \leq \sum_{|k_1| \leq 2\lambda/|\nu_1|} |\hat{f}(k)|^2 + \sum_{k \notin C_{1,\mu,\lambda}} |\hat{f}(k)|^2 \\
    & \leq C \left( \frac{\lambda^2}{|\nu_1|^2} + \mu^{-2} + (\lambda \epsilon)^{-1} \right) E_{\epsilon}(\chi) + C \lambda^{-1} \Per(\Omega).
\end{align*}

  \begin{figure}
    \centering
    \tdplotsetmaincoords{10}{0}
    \tdplotsetrotatedcoords{0}{30}{0}
    \begin{tikzpicture}[thick,tdplot_rotated_coords]
      \draw[->] (-5.1,0,0) -- (5.1,0,0);
      \draw[->] (0,-5.1,0) -- (0,5.1,0);
      \draw[->] (0,0,-5.1) -- (0,0,5.1);

      \draw (-5,-5/4,0) -- (0,0,0) -- (5,5/4,0);
      \draw (-5, 5/4,0) -- (0,0,0) -- (5,-5/4,0);
      \draw[dashed, thin] (-5,0,-5/4) -- (5,0,5/4);
      \draw[dashed, thin] (-5,0,5/4) -- (5,0,-5/4);

      \draw[dashed, canvas is yz plane at x=5, thin, gray] (0,0) circle (5/4);
      \draw[dashed, canvas is yz plane at x=-5, thin, gray] (0,0) circle (5/4);

      \draw[dashed, thin, canvas is yz plane at x=3, fill = Blue, fill opacity = 0.15] (0,0) circle (3/4);
      \draw[dashed, thin, canvas is yz plane at x=-3, fill = Blue, fill opacity = 0.15] (0,0) circle (3/4);

      \draw[Blue] (3,-3/4) -- (3,3/4);
      \draw[Blue] (-3,-3/4) -- (-3,3/4);

      \draw[->] (0,0,0) -- (0.6,0.8,0) node[above] {$\nu$};

      \draw[fill = Red, opacity = 0.1] (-25/6,5,5) -- (5,15/8-15/4,5) -- (5,15/8-15/4,-5) --
      (-25/6,5,-5) -- cycle;
      \draw[fill = Red, opacity = 0.1] (25/6,-5,5) -- (-5,-15/8+15/4,5) -- (-5,-15/8+15/4,-5) --
      (25/6,-5,-5) -- cycle;

      \draw[Red] (-25/6,5,0) -- (5,15/8-15/4,0);
      \draw[Red] (25/6,-5,0) -- (-5,-15/8+15/4,0);

      \draw[Orange] (15/4,-15/16,0) -- (15/4,15/16,0);
      \draw[Orange] (-15/4,-15/16,0) -- (-15/4,15/16,0);
      \draw[dotted, canvas is yz plane at x=-15/4, fill = Orange, fill opacity = 0.15] (0,0) circle (15/16);
      \draw[dotted, canvas is yz plane at x= 15/4, fill = Orange, fill opacity = 0.15] (0,0) circle (15/16);
    \end{tikzpicture}
    \caption{Illustration of regions of different Fourier mass control in the two-well
      setting. Choosing $\bar{\lambda} = \bar{\lambda}(\lambda)$ such that the blue (dashed) and orange (dotted) circles coincide we control the Fourier mass everywhere. }
    \label{fig:TwoWellFourier}
  \end{figure}
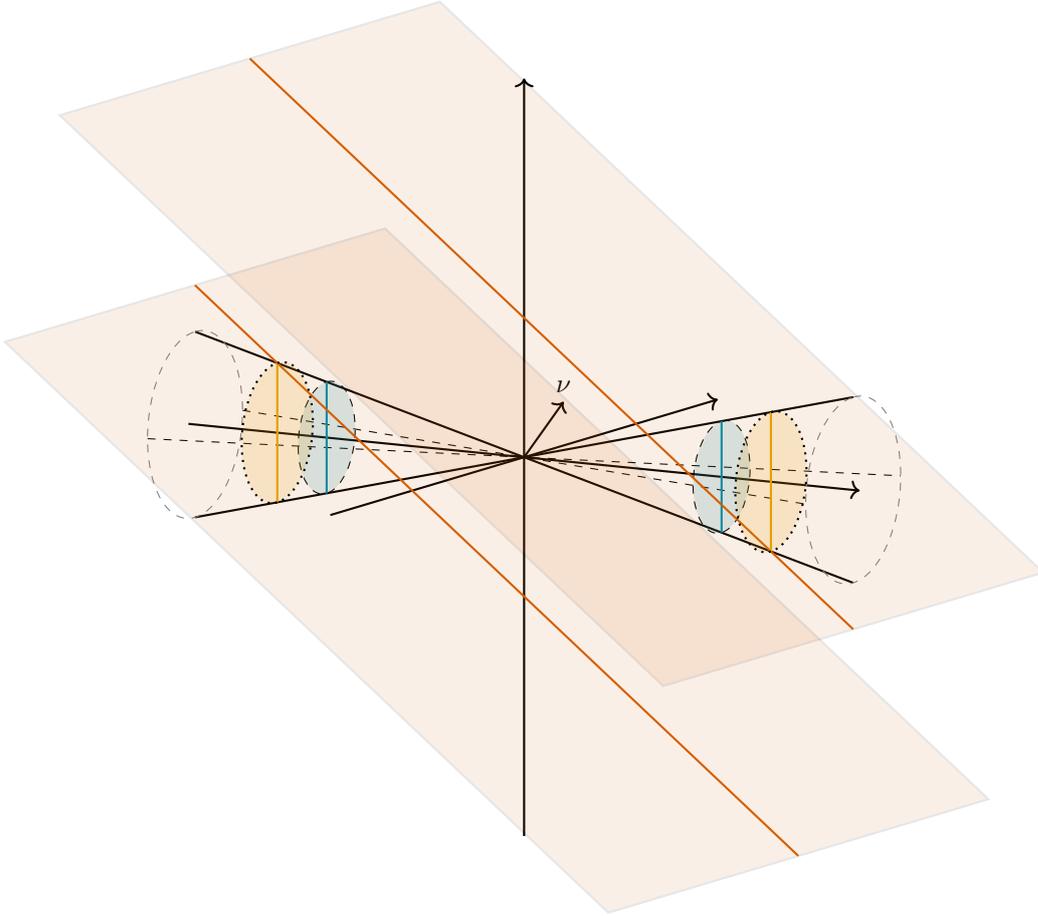

  To balance the first two terms $\lambda^2\nu_1^{-2} + \mu^{-2}$, the optimal choice of $\mu$ is $\mu \sim
  \lambda^{-1}|\nu_1|$, e.g. $\mu = \frac{|\nu_1|}{2} \lambda^{-1} < \frac{|\nu_1|}{2}$. This yields
  \begin{align*}
    \sum_{k \in \Z^d} |\hat{f}(k)|^2 \leq C\left( \frac{\lambda^2}{\nu_1^2} + (\lambda \epsilon)^{-1} \right)
    E_\epsilon(\chi) + C \lambda^{-1} \Per(\Omega),
  \end{align*}
  which again we optimize in $\lambda$, by choosing $\lambda =
  |\nu_1|^\frac{2}{3}\epsilon^{-\frac{1}{3}} > 1$ for $\epsilon < |\nu_1|^2$.

  By this choice, and as $\sum_{k \in \Z^d} |\hat{f}(k)|^2 = \int_\Omega |f(x)|^2 dx \geq
  \min\{1-\alpha,\alpha\}^2$, after an absorption of the perimeter term into the
  left-hand side,
  we derive the lower scaling bound for $\epsilon < \epsilon_0(|\nu_1|,\alpha,d)$
  \begin{align*}
    E_\epsilon(\chi) \geq  C |\nu_1|^{\frac{2}{3}} \epsilon^{\frac{2}{3}}.
  \end{align*}
  Fixing $\epsilon_0 < |\nu_1|^2$ we can also ensure $\lambda > 1$ as required above.

For the general case $F_\alpha\neq0$ and $A-B = a \otimes e_1$ for $a \in \R^d \setminus \{0\}$ we consider the functions $\tilde{u} = R(u-F_\alpha x)$, $\tilde{\chi} = R(\chi-F_\alpha)$ with a rotation $R \in SO(d)$ such that $Ra = |a|e_1$, then $\tilde{\chi}= |a| f e_1 \otimes e_1$, and we can apply the above arguments as $E_\epsilon(\tilde\chi)\sim E_\epsilon(\chi)$.\\

\emph{Step 2: Upper bound.}
  Since simple branching constructions are well-understood (cf. \cite{KM1,CC15}), our sole goal in this step is to make the $\nu_1$ dependence of the prefactor explicit.
  For this, our proof is only a minor adaptation of the ``usual'' branching construction (see, for instance,  \cite[Lemma 3.2]{RT23}), and we work only in two dimensions for simplicity.
  For later use (cf. proof of \cref{lem:scal-second} below) we provide the main estimate on a general rectangular domain $Q=(0,L)\times(0,H)$, which in particular includes the case of the unit square.

  For the reader's convenience, we recall that the domain $Q$ is subdivided in cells $\{\omega_{j,k} : j=0,\dots,j_0+1, k=1,\dots,2^jN\}$ for some $N\in\N$ sufficiently large, where the cells $\omega_{j,k}$ coincide (up to translations) with $(0,\ell_j)\times(0,h_j)$, where
  \begin{equation}\label{eq:parameters}
    \ell_j:=\frac{L}{N 2^j}, \quad h_j:=\frac{(1-\theta)H}{2}\theta^j,
  \end{equation}
  for some $ \theta  \in (1/4, 1/2)$. We refer to \cref{fig:branching} for an illustration of this.
  One produces a lamination (which doubles the frequency from the bottom to the top) in a reference rectangular cell $\omega = (0,\ell) \times (0,h)$, cf. \cref{fig:branching_unitcell}.
  This lamination is transferred on every $\omega_{j,k}$ via rescaling, then obtaining the global construction by attaching all the self-similar copies together, see \cref{fig:branching_combined}.

  In particular for the constructed functions $u \in W^{1,\infty}(\Omega;\R^2),\chi \in BV(\Omega;\{A,B\})$ it holds
  \begin{align} \label{eq:PropertiesBranching}
      \nabla u \in BV(\Omega;\R^{2 \times 2}), \ \Vert \nabla u \Vert_{L^\infty} \leq C(\alpha,\Omega), \Vert D\nabla u \Vert_{TV(\Omega)} \leq C(\alpha,\Omega)(\Vert D \chi \Vert_{TV(\Omega)} + \Per(\Omega)).
  \end{align}

  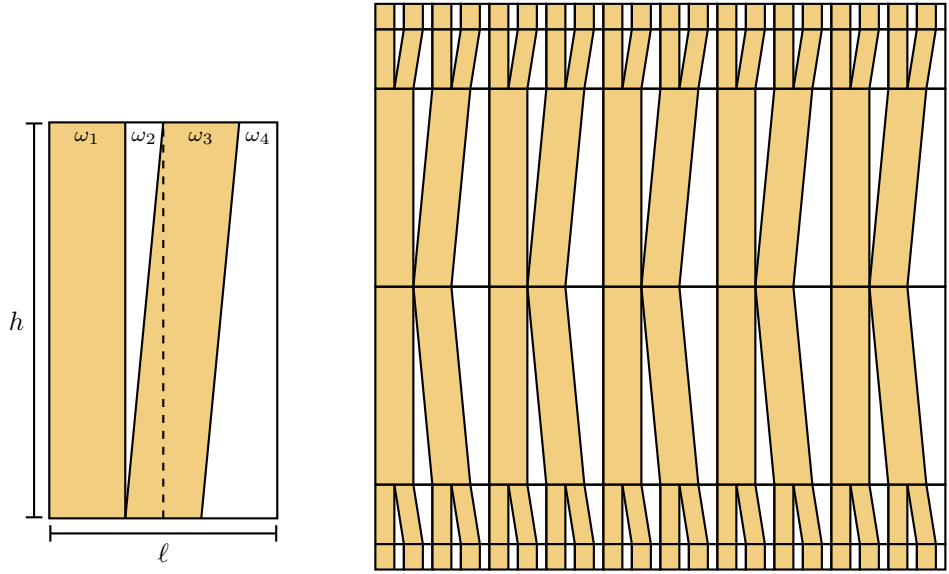
\begin{figure}[t!]
    \centering
    \begin{subfigure}[t]{0.3\textwidth}
        \centering
        \begin{tikzpicture}[thick]
            \fill[draw = none, color = Orange, opacity = 0.5] (0,0) rectangle (1,5.25);
            \fill[draw = none, color = Orange, opacity = 0.5] (1,0) -- (1.5,5.25) -- (2.5,5.25) -- (2,0) -- cycle;

            \draw (0,0) rectangle (3,5.25);
            \draw (1,0) -- (1,5.25);
            \draw (1,0) -- (1.5,5.25);
            \draw (2,0) -- (2.5,5.25);
            \draw[dashed] (1.5,0) -- (1.5,5.25);
            \draw[|-|] (0,-0.2) -- (3,-0.2) node[midway, below] {$\ell$};
            \draw[|-|] (-0.2,0) -- (-0.2,5.25) node[midway, left] {$h$};

            \node at (0.5,5.25)[below] {\footnotesize$\omega_1$};
            \node at (1.25,5.25)[below] {\footnotesize$\omega_2$};
            \node at (2,5.25)[below] {\footnotesize$\omega_3$};
            \node at (2.75,5.25)[below] {\footnotesize$\omega_4$};
        \end{tikzpicture}
        \caption{Unit cell $\omega$ for the branching construction. At the top the oscillation is twice as fast compared to the bottom.}
        \label{fig:branching_unitcell}
    \end{subfigure}%
    ~
    \begin{subfigure}[t]{0.6\textwidth}
        \centering
        \begin{tikzpicture}[thick, scale = 0.5]
        \foreach \y in {-1,1}{
            \begin{scope}[yscale = \y]
                \foreach \x in {0,3,6,9,12}{
                    \fill[draw = none, color = Orange, opacity = 0.5] (\x,0) rectangle (\x+1,5.25);
                    \fill[draw = none, color = Orange, opacity = 0.5] (\x+1,0) -- (\x+1.5,5.25) -- (\x+2.5,5.25) -- (\x+2,0) -- cycle;

                    \draw (\x,0) rectangle (\x+3,5.25);
                    \draw (\x+1,0) -- (\x+1,5.25);
                    \draw (\x+1,0) -- (\x+1.5,5.25);
                    \draw (\x+2,0) -- (\x+2.5,5.25);
                }

                \begin{scope}[yshift = 5.25cm, xscale=0.5, yscale = 0.3]
                    \foreach \x in {0,3,6,9,12,15,18,21,24,27}{
                        \fill[draw = none, color = Orange, opacity = 0.5] (\x,0) rectangle (\x+1,5.25);
                        \fill[draw = none, color = Orange, opacity = 0.5] (\x+1,0) -- (\x+1.5,5.25) -- (\x+2.5,5.25) -- (\x+2,0) -- cycle;

                        \draw (\x,0) rectangle (\x+3,5.25);
                        \draw (\x+1,0) -- (\x+1,5.25);
                        \draw (\x+1,0) -- (\x+1.5,5.25);
                        \draw (\x+2,0) -- (\x+2.5,5.25);

                    }
                \end{scope}

                \begin{scope}[yshift = 6.825cm, xscale=0.25]
                    \foreach \x in {0,3,6,9,...,45,48,51,54,57}{
                        \fill[draw = none, color = Orange, opacity = 0.5] (\x,0) rectangle (\x+2,0.675);

                        \draw (\x,0) rectangle (\x+3,0.675);
                        \draw (\x+2,0) -- (\x+2,0.675);
                    }
                \end{scope}
            \end{scope}
        }
        \end{tikzpicture}
        \caption{Self similar refinement of the unit cell.}
        \label{fig:branching_combined}
    \end{subfigure}
    \caption{Branching construction in the unit cube.}
    \label{fig:branching}
\end{figure}

  In what follows, we now make the $\nu$ dependence explicit by considering the unit-cell construction in more detail.
  The interfaces of the optimal (in the sense of scaling) construction in $\omega$
  (see \cite[Lemma 3.1]{RT23} for details) are either given by a straight line
  with unit normal $e_1$, or normal in direction $(-h, \frac{1-\alpha}{2}\ell)^T$,
  and thus, denoting the normal by $n$, in the first case  it holds that $|n
  \cdot \nu| = |\nu_1|$  and in the second case
  \begin{align*}
  |n \cdot \nu| \leq C(\alpha)( |\nu_1| + \frac{\ell}{h} |\nu_2|).
  \end{align*}
  Hence, through a scaling argument and by summing all the self-similar contributions of $\omega_{j,k}$, we get for sufficiently large $N \in \N$
  \begin{equation}\label{eq:bra1-nu}
      \int_{Q} |\nabla u - \chi|^2 dx + \epsilon \Vert D_\nu \chi \Vert_{TV(Q)}
      \leq C(\alpha) \left( \frac{L^3}{N^2 H}+\epsilon HN|\nu_1|+\epsilon L j_0|\nu_2| \right),
  \end{equation}
  for
  \begin{align*}
      j_0 + 1 \sim \log\Big(\frac{HN}{L}\Big).
  \end{align*}
  We now turn to the case $Q=\Omega$.
  Hence, for $H=L=1$, after optimizing the first two terms in $N$, i.e. choosing $N \sim |\nu_1|^{-1/3} \epsilon^{-1/3}$, we get
  \begin{align*}
      \int_\Omega |\nabla u - \chi|^2 dx + \epsilon \Vert D_\nu \chi \Vert_{TV(\Omega)} \leq C(\alpha) \left(|\nu_1|^{\frac{2}{3}} \epsilon^{\frac{2}{3}} + \epsilon |\log (|\nu_1|\epsilon)| |\nu_2| \right).
  \end{align*}
  Using that $\epsilon |\log(|\nu_1|\epsilon)| |\nu_2| \leq C |\nu_1|^{2/3} \epsilon^{2/3}$ for small $\epsilon < \epsilon_0(|\nu_1|)$, we deduce the desired upper bound.
  \end{proof}

\begin{proof}[Proof of \cref{prop:L1_2wells} for $\nu \cdot e_1 = 0$]
In the case of $\nu \cdot e_1 = 0$ the argument uses that for a simple laminate in $e_1$ direction, we do not pay surface energy, and thus can do an infinitely fine oscillation of the two phases.
Indeed, $E_\epsilon(u,\chi) \geq 0$ is direct. To obtain a suitable upper bound on the energy, we now choose for any $N \in \N$ the functions
\begin{align*}
    \tilde{u}_N(x_1) &= \begin{cases} (1-\alpha) x_1 a & x_1 \in (0,\frac{\alpha}{N}), \\
    - \alpha x_1 a + \frac{\alpha}{N} a& x_1 \in [\frac{\alpha}{N},\frac{1}{N}),
    \end{cases} \\
    \chi_N(x) & = \begin{cases} A & x_1 \in (0,\frac{\alpha}{N}), \\
    B & x_1 \in [\frac{\alpha}{N},\frac{1}{N}),
    \end{cases}
\end{align*}
and extend both $\frac{1}{N}$-periodically.
We fix $u_N \in W^{1,\infty}(\Omega;\R^{d})$ as
\begin{align*}
    u_N(x_1,x') = \tilde{u}_N(x_1) \psi(N \dist(x',\p (0,1)^{d-1})) + F_\alpha x
\end{align*}
for some bump function $\psi \in C^\infty(\R;[0,1])$ such that $\psi(t) = 1$ for $t  \geq 1$ and $\psi(t) = 0$ for $t \leq \frac{1}{2}$.
Here $a \in \R^d \setminus \{0\}$ is given by the relation $A-B = a \otimes e_1$.
We have that $\chi_{N} \in BV(\Omega;\{A,B\})$ and $u_{N} \in W^{1,\infty}(\Omega;\R^{d})$ fulfils the Dirichlet data.
Moreover, as $\nu \cdot e_1 = 0$, it holds
\begin{align*}
    \int_\Omega \chi_N(x) \p_\nu \phi(x) dx = \int_{(0,1)} \chi_N\Big(x_1,\frac{1}{2}\Big) \int_{(0,1)^{d-1}} {\rm div}'(\phi(x_1,x')\nu) dx' dx_1 =  0
\end{align*}
for every $\phi \in C^1_c(\Omega;\R)$, thus $\Vert D_\nu \chi \Vert_{TV(\Omega)} =0$.
Hence, we infer that
\begin{align*}
    E_\epsilon(u_N,\chi_N) = \int_\Omega |\nabla u_N(x) - \chi_N(x)|^2 dx \leq \frac{C}{N}.
\end{align*}
    As the functions $u$ and $\chi$ are admissible for any $N \in \N$, passing to the limit yields the optimal energy $\inf_{\chi \in BV_\nu(\Omega;\{A,B\})} \inf_{u \in \mathcal{A}_{F_{\alpha}}} E_\epsilon(u,\chi) = 0$.
\end{proof}

\subsection{The three-well problem of Lorent}
In this section we now consider the three-well setting due to Lorent. As outlined in \cref{sec:introL13}, in this situation one obtains both first and second order laminates, depending on the boundary condition.

We split the proof of \cref{thm:L1_3wells} into several steps. We start by showing the scaling for $F \in \K_3^1$ for both $\nu \cdot e_1 \neq 0$ and $\nu \cdot e_1 = 0$, both of which are essentially covered by \cref{prop:L1_2wells}. In contrast to the argument given above, due to our specific choices of the possible boundary data, we do not only work with one of the entries of the matrix $\chi-F$ (see the comments in \cref{rmk:lower_bound_first_laminate} below).

In a second step, we will consider $F \in \K_3^2$ where we can exploit the determinedness of $\chi_{22}$ in terms of $\chi_{11}$ to obtain the scaling of second order laminates. Here as long as $\nu \cdot e_1 \neq 0$, we obtain a scaling law of the order $\epsilon^{\frac{1}{2}}$ as in the isotropic setting. If, however, $\nu \cdot e_1 = 0$, the scaling does not change with respect to the one for first order laminates, as we can construct an infinitely fine simple laminate within a branching construction, without paying surface energy for the simple laminate. This then yields the same scaling as in the two-well setting in \cref{prop:L1_2wells}, cf. \cref{fig:LaminateInBranching}.

\begin{lem}
\label{lem:scal-first}
    Under the same assumptions as in \cref{thm:L1_3wells}, let $F_\alpha = \diag(\alpha,0) \in \K_3^1$ for $\alpha \in (0,1)$.
    If $\nu \cdot e_1 \neq 0$ there is a constant $C = C(\alpha) >0$ and $\epsilon_0 = \epsilon_0(\alpha,|\nu_1|) > 0$ such that
    \begin{align*}
    C^{-1} |\nu_1|^{\frac{2}{3}} \epsilon^{\frac{2}{3}} \leq \inf_{\chi \in BV_\nu(\Omega;\K_3)} \inf_{u \in \mathcal{A}_{F_\alpha}} E_\epsilon(u,\chi)  \leq C |\nu_1|^{\frac{2}{3}} \epsilon^{\frac{2}{3}}.
    \end{align*}
    If instead $\nu \cdot e_1 = 0$, we have
    \begin{align*}
        \inf_{\chi \in BV_\nu(\Omega;\K_3)} \inf_{u \in \mathcal{A}_{F_\alpha}} E_\epsilon(u,\chi) = 0.
    \end{align*}
\end{lem}
\begin{proof}

\emph{Upper bounds:}
    As $A_1$ and $A_2$ are rank-one connected in $e_1$ direction, we directly note that the upper bound construction in the proof of \cref{prop:L1_2wells} also yields an upper bound construction in this setting (after an adaptation of parameters). Thus, we have
    \begin{align*}
        \inf_{\chi \in BV_\nu(\Omega;\K_3)} \inf_{u \in \mathcal{A}_{F_\alpha}} E_\epsilon(u,\chi) \leq \begin{cases}
            C(\alpha)|\nu_1|^{\frac{2}{3}} \epsilon^{\frac{2}{3}}, & \nu \cdot e_1 \neq 0, \\
            0, & \nu \cdot e_1 = 0.
        \end{cases}
    \end{align*}

    \emph{Lower bound for $\nu \cdot e_1 \neq 0$:} For any $u\in\mathcal{A}_{F_\alpha}, \chi\in BV_\nu(\Omega;\K_3)$, by considering $\nabla u - F_\alpha$ and $\chi - F_\alpha$, we can reduce to the case $F_\alpha = 0$ (cf.\ proof of \cref{prop:L1_2wells}).
    With a slight abuse of notation, we still write $u$ and $\chi$ for the modified functions, hence $u\in H^1_0(\Omega;\R^2)$ and $\chi \in \K_3 - F_\alpha$.
    For the function $\chi_{11}$ we exploit the same ideas as in the two-well case but will make use of both the $\chi_{11}$ and the $\chi_{22}$ components.
    To be more precise by \cref{lem:ElasticFourier} it holds
    \begin{align*}
        \int_\Omega |\nabla u - \chi|^2 dx \geq \sum_{k \in \Z^2} \left( \frac{|k|^2 - k_1^2}{|k|^2} |\hat{\chi}_{11}(k)|^2 + \frac{|k|^2 - k_2^2}{|k|^2} |\hat{\chi}_{22}(k)|^2 \right) \geq \sum_{k \in \Z^2} \frac{|k|^2 - k_1^2}{|k|^2} |\hat{\chi}_{11}(k)|^2,
    \end{align*}
    where we fix the multipliers to be equal to one in $k = 0$.
    Moreover, it holds
    \begin{align*}
        \Vert D_\nu \chi \Vert_{TV(\Omega)} \geq \Vert D_\nu \chi_{11} \Vert_{TV(\Omega)},
    \end{align*}
    and thus, following the ideas of \cref{prop:L1_2wells}, by \cref{lem:Localization}, we deduce for $\lambda>1$ and $\mu\in(0,\frac{|\nu_1|}{16})$, that
    \begin{align*}
        \Vert \chi_{11} - m_{1,\mu,\lambda}(D) \chi_{11} \Vert_{L^2}^{2} + \Vert \chi_{22} - m_{1,\mu}(D) \chi_{22} \Vert_{L^2}^{2} &\leq \sum_{k \notin C_{1,\mu,\lambda}} |\hat{\chi}_{11}(k)|^2 + \sum_{k \notin C_{2,\mu}} |\hat{\chi}_{22}(k)|^2 \\
        &\leq C (\mu^{-2} + (\lambda \epsilon)^{-1}) E_\epsilon(u,\chi) + C \lambda^{-1} \Per(\Omega),
    \end{align*}
    where we consider the smooth cut-off multipliers as in \cref{eq:multiplier}.
    Applying moreover \cref{cor:CommutatorEst_Applied} with $\chi_{22} = -4 (\chi_{11}+\alpha)(\chi_{11}-1+\alpha)$, yields for $\lambda_2=M\mu\lambda < \lambda$ 
    \begin{align*}
        \Vert \chi_{22} - m_{2,\mu,\lambda_2}(D) \chi_{22} \Vert_{L^2}^{2} &\leq C(\alpha) \psi_{1-\gamma}(\Vert \chi_{11} - m_{1,\mu,\lambda}(D) \chi_{11} \Vert_{L^2}^{2}) + 4 \Vert \chi_{22} - m_{2,\mu}(D) \chi_{22} \Vert_{L^2}^{2} \\
        & \leq C(\alpha) \psi_{1-\gamma}\Big((\mu^{-2} + (\lambda \epsilon)^{-1}) E_\epsilon(u,\chi) + \lambda^{-1} \Per(\Omega)\Big),
    \end{align*}
    with $\psi_{1-\gamma}(x) = \max\{|x|,|x|^{1-\gamma}\}$.

    In conclusion, using that $1-m_{2,\mu,\lambda_2}(k) \geq 1 - m_{2,\mu,\lambda}(k) \geq 0$,
    \begin{align*}
       & \Vert \chi_{11}-m_{1,\mu,\lambda}(D) \chi_{11} \Vert_{L^2}^{2} + \Vert \chi_{22} - m_{2,\mu,\lambda}(D) \chi_{22} \Vert_{L^2}^{2} \\
       &\leq C \psi_{1-\gamma}\Big((\mu^{-2} + (\lambda \epsilon)^{-1}) E_\epsilon(u,\chi) + \lambda^{-1} \Per(\Omega)\Big).
    \end{align*}
    Applying \cref{lem:LowFreqEst} for both $\chi_{11}$ and $\chi_{22}$ with the frequency cut-off given by $\bar{\lambda} = \frac{4}{|\nu_1|} \lambda$ we get
    \begin{align*}
        \Vert \chi_{11} \Vert_{L^2}^{2} + \Vert \chi_{22} \Vert_{L^2}^{2} & \leq 2 \Vert \chi_{11} - m_{1,\mu,\lambda}(D) \chi_{11} \Vert_{L^2}^{2} + 2 \sum_{|k_1| \leq \frac{4}{|\nu_1|} \lambda} |\hat{\chi}_{11}(k)|^2 + \\
        & \quad + 2 \Vert \chi_{22} - m_{2,\mu,\lambda}(D) \chi_{22} \Vert_{L^2}^{2} + 2 \sum_{|k_2| \leq \frac{4}{|\nu_1|}\lambda} |\hat{\chi}_{22}(k)|^2 \\
        & \leq C \psi_{1-\gamma}\Big((\mu^{-2} + (\lambda \epsilon)^{-1}) E_\epsilon(u,\chi) + \lambda^{-1} \Per(\Omega)\Big)\\
        & \quad + C |\nu_1|^{-2} \lambda^2 E_\epsilon(u,\chi) \\
        & \leq C \psi_{1-\gamma}\Big((|\nu_1|^{-2} \lambda^2 + \mu^{-2} + (\lambda \epsilon)^{-1}) E_\epsilon(u,\chi) + \lambda^{-1} \Per(\Omega)\Big).
    \end{align*}
    Fixing $\mu^{-2} \sim |\nu_1|^{-2} \lambda^2$ and $\lambda \sim |\nu_1|^{2/3} \epsilon^{-1/3}$ (which are compatible with the constrains $\mu<\frac{|\nu_1|}{16}$ and $\bar\lambda>1$) yields
    \begin{align}
    \label{eq:second_trunc}
        \Vert \chi_{11} \Vert_{L^2}^{2} + \Vert \chi_{22} \Vert_{L^2}^{2} \leq C \psi_{1-\gamma}\Big(|\nu_1|^{-\frac{2}{3}} \epsilon^{-\frac{2}{3}} E_\epsilon(u,\chi) + |\nu_1|^{-\frac{2}{3}} \epsilon^{\frac{1}{3}} \Per(\Omega)\Big).
    \end{align}

    By the fact that $\Vert \chi_{11} \Vert_{L^2}^{2} + \Vert \chi_{22} \Vert_{L^2}^{2} = \Vert \chi \Vert_{L^2}^{2} \geq c > 0$, we can show the desired lower bound for $\epsilon < \epsilon_0(|\nu_1|,\alpha)$.
    Indeed, we have either
    \begin{align*}
        C \leq |\nu_1|^{-\frac{2}{3}} \epsilon^{-\frac{2}{3}} E_\epsilon(u,\chi) + |\nu_1|^{-\frac{2}{3}} \epsilon^{\frac{1}{3}} \Per(\Omega),
    \end{align*}
    or
    \begin{align*}
        C^{\frac{1}{1-\gamma}} \leq |\nu_1|^{-\frac{2}{3}} \epsilon^{-\frac{2}{3}} E_\epsilon(u,\chi) + |\nu_1|^{-\frac{2}{3}} \epsilon^{\frac{1}{3}} \Per(\Omega),
    \end{align*}
    depending on the case distinction in $\psi_{1-\gamma}$. In conclusion, after absorbing the perimeter term, we arrive at
    \begin{align*}
        E_{\epsilon}(u,\chi) \geq \frac{1}{2}\min\{C,C^{\frac{1}{1-\gamma}}\} |\nu_1|^{\frac{2}{3}} \epsilon^{\frac{2}{3}}.
    \end{align*}
\end{proof}

\begin{rmk} \label{rmk:lower_bound_first_laminate}
Let us comment on a technical aspect which is specific to our anisotropic surface energies and which does not arise in this form in the isotropic setting.
    We observe that for specific choices of $F_\alpha \in \K^{1}_3$ (both in the isotropic and anisotropic settings) it may happen that $\chi_{11}$ vanishes. Indeed, we recall that $\chi \in \K_3 - F_\alpha = \{\diag(-\alpha, 0), \diag(1-\alpha,0), \diag(1/2-\alpha,1)\}$ for $\alpha \in (0,1)$ and, hence, $\chi_{11}=0$ can occur for $\alpha = 1/2$.
    For this reason, in the above proof, we also used $\chi_{22}$ in the lower bound, in order to deduce a uniform lower bound in \cref{eq:second_trunc}. To this end, in the above argument, we used a commutator estimate already for \emph{first} order laminates (while in the isotropic setting commutators only enter for \emph{second} and \emph{higher} order laminates).
    In the setting of \cite{RT23} (while also here $\chi_{11}=0$ may arise) it is \emph{not} necessary to use a commutator estimate of the form \cref{lem:CommutatorEst} for first order laminates, as -- due to the isotropy of the surface energy -- in that article the first high frequency localization truncates the cones in \emph{all} directions.
    In our setting this truncation is not present due to the anisotropy of the surface energy and we, hence, require an extra step to show the lower bound.

    As an alternative argument, also in our anisotropic setting, one could have avoided an application of the commutator estimate at the expense of using further information on the $\chi_{22}$ component. Indeed, one could have restricted to a bound of the form
    \begin{align}
    \label{eq:single_trunc}
        \Vert \chi_{11} \Vert_{L^2}^{2}  \leq C \psi_{1-\gamma}\Big(|\nu_1|^{-\frac{2}{3}} \epsilon^{-\frac{2}{3}} E_\epsilon(u,\chi) + |\nu_1|^{-\frac{2}{3}} \epsilon^{\frac{1}{3}} \Per(\Omega)\Big),
    \end{align}
    and then, in a second step, invoked information on the average of $\chi_{22}$. To this end, we note that by Jensen's inequality, it holds that
    \begin{align*}
    |\langle \nabla u \rangle_{\Omega} - \langle \chi \rangle_{\Omega}|^{2} \leq E_{el}(u,\chi),
    \end{align*}
    where $\langle \cdot \rangle_{\Omega}$ denotes the average on $\Omega$. In particular, we here use that by the imposed boundary conditions and the fundamental theorem of calculus, $\langle \nabla u \rangle_{\Omega} = F_\alpha$ and that $F_{22} = 0$ for all $F \in \K^1_3$.

    Hence, considering the second component we have
    \begin{align*}
        |\{x \in \Omega: \chi(x) =A_3\}|^2 \leq |\langle \chi_{22}\rangle_{\Omega}|^2 \leq |F_\alpha-\langle \chi \rangle_{\Omega}|^2 \leq E_{el}(u,\chi).
    \end{align*}
    In particular,
    \begin{align*}
        \Vert \chi_{11} \Vert_{L^2}^2 &\geq \min\{\alpha^2,(1-\alpha)^2\} |\{x \in \Omega: \chi \neq A_3\}| \geq C(\alpha) (|\Omega| - |\{x \in \Omega: \chi = A_3\}|) \\
        &\geq C(\alpha)(|\Omega|-E_{el}(u,\chi)^{\frac{1}{2}}).
    \end{align*}
    Returning to \cref{eq:single_trunc} with this additional information and rearranging the inequality, one then infers that
    \begin{align*}
      c \leq  \Vert \chi_{11} \Vert_{L^2}^{2} + C E_{el}(u,\chi)^{\frac{1}{2}} \leq C \psi_{1-\gamma}\Big(|\nu_1|^{-\frac{2}{3}} \epsilon^{-\frac{2}{3}} E_\epsilon(u,\chi) + |\nu_1|^{-\frac{2}{3}} \epsilon^{\frac{1}{3}} \Per(\Omega)\Big) + C E_{el}(u,\chi)^{\frac{1}{2}},
    \end{align*}
    which also concludes the argument after an absorption of the perimeter and additional elastic energy terms.
\end{rmk}

Turning now to the second order laminates, we consider
\begin{align*}
    F = \begin{pmatrix}
        \frac{1}{2} & 0 \\ 0 & \alpha
    \end{pmatrix} \in \K_3^2
\end{align*}
for some $\alpha \in (0,1)$ and note that $\chi_{22} - F_{22} \in \{-\alpha,1-\alpha\} \not\ni 0$. Thus, we now aim to control the Fourier mass of $\chi_{22}$.

\begin{lem}\label{lem:scal-second}
    Under the same assumptions as in \cref{thm:L1_3wells}, let $F_\alpha = \diag(\frac{1}{2},\alpha) \in \K_3^2$ for $\alpha \in (0,1)$.
    If $\nu \cdot e_1 \neq 0$ there are constants $C = C(\alpha) > 0$ and $\epsilon_0 = \epsilon_0(\alpha,|\nu_1|) > 0$ such that for any $\epsilon \in (0,\epsilon_0)$
    \begin{align*}
        C^{-1} |\nu_1|^{\frac{1}{2}} \epsilon^{\frac{1}{2}} \leq \inf_{u \in \mathcal{A}_{F_\alpha}} \inf_{\chi \in BV_\nu(\Omega;\K_3)} E_\epsilon(u,\chi) \leq C |\nu_1|^{\frac{1}{2}} \epsilon^{\frac{1}{2}}.
    \end{align*}

    If $\nu \cdot e_1 = 0$, we have with $C = C(\alpha) > 0$ and $\epsilon_0 = \epsilon_0(\alpha,|\nu_2|) > 0$ for all $\epsilon \in (0,\epsilon_0)$
    \begin{align*}
        C^{-1} |\nu_2|^{\frac{2}{3}} \epsilon^{\frac{2}{3}} \leq \inf_{u \in \mathcal{A}_{F_\alpha}} \inf_{\chi \in BV_\nu(\Omega;\K_3)} E_\epsilon(u,\chi) \leq C |\nu_2|^{\frac{2}{3}} \epsilon^{\frac{2}{3}}.
    \end{align*}
\end{lem}
 Again, we split the proof into two parts, first for $\nu \cdot e_1 \neq 0$ and the second part for $\nu \cdot e_1 =0$.

\begin{proof}[Proof for $\nu \cdot e_1 \neq 0$]
    \emph{Upper bound:} As for the first-order branching construction, the claimed upper scaling bound in $\epsilon$ is already known (cf. \cite{KW16,RT23}).
    We just need to focus on the dependence of the prefactor on $\nu$.
    In particular, we will follow the strategy of the proof of \cite[Thm. 1.2 (ii)]{RT23}. This consists in concatenating two orders of branching constructions; an outer one between gradients $A_3$ and $\diag(1/2,0)$, and then replacing the regions in which $\nabla u\approx \diag(1/2,0)$ with an inner branched lamination between $A_1$ and $A_2$.

    Let $\{\omega_{j,k}\}$ be the (outer) first-order branching covering as in the Step 2 of the proof of \cref{prop:L1_2wells} (now with switched roles between $x_1$ and $x_2$ because of the structure of the wells).
    In each cell $\omega_{j,k}$ we can produce an inner branching construction
    (see \cite{RT23} for details) so that estimate \cref{eq:bra1-nu} applies as follows
    \begin{align*}
    \int_{\omega_{j,k}} |\nabla u-\chi|^2 dx + \epsilon\Vert D_\nu\chi\Vert_{TV(\omega_{j,k})} \leq C(\alpha)\Big( \frac{h_j^3}{M^2\ell_j}+\epsilon\ell_j(M|\nu_1|+k_0|\nu_2|) \Big),
    \end{align*}
    where $M\in\N$, $M\sim(2\theta)^jN^2$ denotes the number of oscillations of the zeroth generation of this inner branching construction (the dependence of $M$ on $j$ is dropped for notational simplicity) and $k_0\sim\log(N)$.
    Summing this for every $k$ and every (outer) generation $j$ and by also adding the surface energy term of the outer branching construction (which comes by attaching the $\omega_{j,k}$-cells together), by the relations \cref{eq:parameters} we obtain
    \begin{align*}
    \int_{\Omega} |\nabla u-\chi|^2 dx + \epsilon\Vert D_\nu\chi\Vert_{TV(\Omega)} \leq C(\alpha) \Big( \frac{1}{N^2}+\epsilon (N^2+j_0)|\nu_1|+\epsilon (k_0+N)|\nu_2| \Big).
    \end{align*}
    Optimizing as $N\sim(\epsilon|\nu_1|)^{-\frac{1}{4}}$ we deduce the claimed upper scaling bound.
    Due to the construction, the bounds from  \cref{eq:PropertiesBranching} still hold for $u$ and $\chi$.

    \emph{Lower bound:}
    Analogously as in the previous proofs, by subtracting the boundary conditions we can assume that $u\in H^1_0(\Omega;\R^2)$ and $\chi\in\K_3-F_\alpha$.
    For the readers' convenience we recall the (truncated) cones from \cref{eq:TruncCone_1,eq:Cone_j}: for $\mu \in (0,\frac{|\nu_1|}{16})$,
$\lambda > 0$, $j=1,2$
\begin{align*}
  C_{j,\mu,\lambda} = \{ k \in \Z^d: |k|^2 - k_{j}^2 \leq \mu^2 |k|^2, \ |k_{j}| \leq \frac{2}{|\nu_1|} \lambda\} , \quad C_{2,\mu} = \{ k \in
  \Z^d: |k|^2 - k_2^2 \leq \mu^2 |k|^2\}.
\end{align*}
For the multipliers defined in \cref{eq:multiplier} 
we infer by \cref{lem:Localization}
\begin{align*}
  \Vert \chi_{11} - m_{1,\mu,\lambda}(D) \chi_{11} \Vert_{L^2}^2 + \Vert \chi_{22} -
  m_{2,\mu}(D) \chi_{22} \Vert_{L^2}^2 &\leq \sum_{k \notin C_{1,\mu,\lambda}} |\hat{\chi}_{11}(k)|^2 + \sum_{k \notin C_{2,\mu}} |\hat{\chi}_{22}(k)|^2 \\
  & \leq C (\mu^{-2} + (\lambda \epsilon)^{-1}) E_\epsilon(u,\chi) + C \lambda^{-1} \Per(\Omega).
\end{align*}
Here the constant $C>0$ only depends on $\alpha$.

Using that $\chi_{22} = 1-\alpha-4\chi_{11}^2$, we deduce from an application of \cref{cor:CommutatorEst_Applied} with $\lambda_2 = M \mu \lambda$ for any $\gamma \in (0,1)$
\begin{align} \label{eq:Lorent_iteration}
  \Vert \chi_{22} - m_{2,\mu,\lambda_2}(D) \chi_{22} \Vert_{L^2}^{2} \leq C(\alpha) \psi_{1-\gamma}(\Vert
  \chi_{11} - m_{1,\mu,\lambda}(D) \chi_{11} \Vert_{L^2}^{2}) + 4 \Vert \chi_{22} - m_{2,\mu}(D)
  \chi_{22} \Vert_{L^2}^{2},
\end{align}
with $\psi_t(x) = \max\{|x|^t,|x|\}$ for $t > 0$.

Now we use this estimate on $\chi_{22} - m_{2,\mu,\lambda_2}(D) \chi_{22}$ to improve the lower
bound in comparison to the two-well case.
Using \cref{lem:LowFreqEst} for $\chi_{22}$, we get for $\bar{\lambda} = \frac{4}{|\nu_1|} \lambda_2$
\begin{align} \label{eq:Lorent_low_freq}
\sum_{|k_2| \leq 4 \lambda_2/|\nu_1|} |\hat{\chi}_{22}(k)|^2 \leq C \frac{\lambda_2^2}{|\nu_1|^2} E_{el}(u,\chi).
\end{align}
Thus, as in the proof of the two-well problem in \cref{prop:L1_2wells}, the idea is to combine \cref{eq:Lorent_iteration,eq:Lorent_low_freq} to have an estimate of $\Vert \chi_{22} \Vert_{L^2}^2$ in terms of the energy depending on the parameters $\mu, \lambda_2$ and then to optimize in these parameters.
To be precise, we have
\begin{align*}
    \Vert \chi_{22} \Vert_{L^2}^2 & \leq 2 \sum_{k \in C_{2,2\mu,2\lambda_2}} |\hat{\chi}_{22}(k)|^2 + 2 \Vert \chi_{22} - m_{2,\mu,\lambda_2}(D) \chi_{22} \Vert_{L^2}^2 \\
    & \leq 2 \sum_{|k_2| \leq 4 \lambda_2/|\nu_1|} |\hat{\chi}_{22}(k)|^2 + 2 \Vert \chi_{22} - m_{2,\mu,\lambda_2}(D) \chi_{22} \Vert_{L^2}^2.
\end{align*}
Thus, plugging in \cref{eq:Lorent_iteration,eq:Lorent_low_freq}, yields with $\lambda_2 = M \mu \lambda$ and a constant $C = C(\alpha) > 0$

\begin{align*}
    \Vert \chi_{22} \Vert_{L^2}^2 & \leq C |\nu_1|^{-2} \lambda_2^2 E_{el}(u,\chi) + C \psi_{1-\gamma}(\Vert \chi_{11}- m_{1,\mu,\lambda}(D) \Vert_{L^2}^2) + C \Vert \chi_{22} - m_{2,\mu}(D) \chi_{22} \Vert_{L^2}^2 \\
    & \leq C \left( |\nu_1|^{-2} \lambda_2^2 E_\epsilon(u,\chi) + \psi_{1-\gamma}\Big((\mu^{-2} + (\lambda \epsilon)^{-1}) E_\epsilon(u,\chi) + \lambda^{-1} \Per(\Omega) \Big) \right) \\
    & \leq C \psi_{1-\gamma}\Big((|\nu_1|^{-2} \mu^2 \lambda^2 + \mu^{-2} + (\lambda \epsilon)^{-1}) E_\epsilon(u,\chi) + \lambda^{-1} \Per(\Omega) \Big).
\end{align*}
As $\Vert \chi_{22} \Vert_{L^2}^2 \geq C \min\{\alpha^2,(1-\alpha)^2\} > 0$, fixing $\mu^{-2} \sim |\nu_1|^{-1} \lambda$ and $\lambda \sim |\nu_1|^{1/2} \epsilon^{-1/2}$, we argue as in the proof of \cref{lem:scal-first}, i.e. considering the two cases for $\psi_{1-\gamma}$ and carrying out an absorption argument for the perimeter term, we arrive at
\begin{align*}
    E_\epsilon(u,\chi) \geq \frac{1}{2} \min\{C,C^{\frac{1}{1-\gamma}}\} |\nu_1|^{\frac{1}{2}} \epsilon^{\frac{1}{2}}.
\end{align*}
\end{proof}

\begin{proof}[Proof for $\nu \cdot e_1 = 0$]
\emph{Upper bound:}
We construct a (cut-off) simple laminate of $A_1$ and $A_2$ within a branching construction using $A_3$.
See \cref{fig:LaminateInBranching} for an illustration of a laminate within a branching construction.
As a first-order branching construction has been already explained in the proof of \cref{prop:L1_2wells}, we only give an outline of the argument.

Considering a reference cell $\omega = (0,h) \times (0,\ell)$ with $0 < \ell < h \leq 1$, we decompose this into further subdomains given by
\begin{align*}
  \omega_1 &= \{ (x_1,x_2) \in \omega: x_2 \in (0,(1-\alpha) \frac{\ell}{2})\},\\
  \omega_2 &= \{ (x_1,x_2) \in \omega: x_2 \in [(1-\alpha) \frac{\ell}{2}, (1-\alpha) \frac{\ell}{2} + \alpha \frac{\ell}{2}\frac{x_1}{h}) \}, \\
  \omega_3 & = \{ (x_1,x_2) \in \omega: x_2 \in [(1-\alpha) \frac{\ell}{2} + \alpha \frac{\ell}{2} \frac{x_1}{h}, (1-\alpha) \ell + \alpha \frac{\ell}{2} \frac{x_1}{h}) \}, \\
  \omega_4 &= \{ (x_1,x_2) \in \omega: x_2 \in [(1-\alpha) \ell + \alpha \frac{\ell}{2} \frac{x_1}{h},\ell)\}.
\end{align*}

For $r < h$ such that $\frac{h}{r} \in \N$ and a bump function $\phi \in C^\infty(\R;[0,1])$ with $\phi(t) = 0$ for $t \leq 0$ and $\phi(t) = 1$ for $t \geq 1$, we then define the continuous function
\begin{align*}
 \tilde{u}(x_1,x_2) =
  \begin{cases}
    \begin{pmatrix}
      \operatorname{Lam}_r(x_1) \phi(\frac{x_2}{r}) \phi(\frac{(1-\alpha) \frac{\ell}{2}-x_2}{r}) \\ -\alpha x_2
    \end{pmatrix}, & (x_1,x_2) \in \omega_1, \\[15pt]
    \begin{pmatrix}
      0 \\ (1-\alpha)x_2 - (1-\alpha) \frac{\ell}{2}
    \end{pmatrix}, & (x_1,x_2) \in \omega_2, \\[15pt]
    \begin{pmatrix}
      \operatorname{Lam}_r(x_1) \phi(\frac{x_2-(1-\alpha) \frac{\ell}{2} - \alpha \frac{\ell}{2} \frac{x_1}{h}}{r}) \phi(\frac{(1-\alpha)\ell + \alpha
      \frac{\ell}{2} \frac{x_1}{h}-x_2}{r}) \\ - \alpha x_2 + \alpha \frac{\ell}{2} \frac{x_1}{h}
    \end{pmatrix}, & (x_1,x_2) \in \omega_3, \\[15pt]
    \begin{pmatrix}
      0 \\ (1-\alpha)x_2 - (1-\alpha) \ell
    \end{pmatrix}, & (x_1,x_2) \in \omega_4.
  \end{cases}
\end{align*}
Here we used
\begin{align*}
  \operatorname{Lam}_r(t) = \begin{cases} - \frac{1}{2}t, & t \in [0,\frac{r}{2}), \\ \frac{1}{2}t-\frac{r}{2}, & t \in [\frac{r}{2},r),
  \end{cases}
\end{align*}
  and extended it $r$-periodically.
Setting also
\begin{align*}
  \tilde{\chi}(x_1,x_2) =
  \begin{cases}
    \begin{pmatrix}
       \operatorname{Lam}'_r(x_1) & 0 \\ 0 & -\alpha
    \end{pmatrix}, & (x_1,x_2) \in \omega_1 \cup \omega_3, \\
    \begin{pmatrix}
      0 & 0 \\ 0 & 1-\alpha
    \end{pmatrix}, & (x_1,x_2) \in \omega_2 \cup \omega_4,
  \end{cases}
\end{align*}
we can calculate the energy contribution of $u(x) = \tilde{u}(x) + F_\alpha x$ and $\chi = \tilde{\chi} + F_\alpha$. Note that $F_\alpha = \operatorname{diag}(1/2,\alpha)$.
We have
\begin{align*}
  \int_\omega |\nabla u (x) - \chi(x)|^2 dx = \int_\omega |\nabla\tilde{u}(x) - \tilde{\chi}(x)|^2 dx \leq C(\alpha)(rh + \frac{\ell^3}{h}),
\end{align*}
where the first term is determined by the size of the cut-off areas in $\omega_1$ and $\omega_2$ (i.e., by the inner laminate), and the second term is due to the error we make by adjusting the interfaces away from $e_2$ to achieve
the refinement in $e_1$ direction (i.e., the usual elastic energy originating from branching).
Moreover, as we do not penalize oscillation in the $e_1$ direction, it holds
\begin{align*}
  \Vert D_{e_2} \chi \Vert_{TV(\omega)} + \Per(\omega) = \Vert D_{e_2} \tilde{\chi} \Vert_{TV(\omega)} + \Per(\omega) \leq C(\alpha) h,
\end{align*}
and thus in total
\begin{align*}
  \int_\omega |\nabla u(x) - \chi(x)|^2 dx + \epsilon \Vert D_{e_2} \chi \Vert_{TV(\omega)} + \epsilon \Per(\omega) \leq C(\alpha)(r h + \frac{\ell^3}{h} + \epsilon h).
\end{align*}
As for any $h > 0$ we can choose $r<\epsilon$ 
the contribution of the $r h$ term is negligible.
Hence,
\begin{align*}
  \int_\omega |\nabla u(x) - \chi(x)|^2 dx + \epsilon \Vert D_{e_2} \chi \Vert_{TV(\omega)} + \epsilon \Per(\omega) \leq C(\alpha) (\frac{\ell^3}{h} + \epsilon h).
\end{align*}
This is the same energy contribution as for the standard branching construction, thus concluding as in the proof of \cref{prop:L1_2wells} (see also \cite[Sec. 3]{RT23}), we get
\begin{align*}
 \inf_{\chi \in BV_\nu(\Omega;\K_3)} \inf_{u \in \mathcal{A}_{F_\alpha}} \int_{(0,1)^2} |\nabla u - \chi|^2 dx + \epsilon \Vert D_{e_2} \chi \Vert_{TV((0,1)^2)} \leq C(\alpha) \epsilon^{\frac{2}{3}}.
\end{align*}

\begin{figure}
  \centering
  \tdplotsetmaincoords{-60}{-110}
  \tdplotsetrotatedcoords{0}{0}{0}
  \begin{tikzpicture}[thick,tdplot_rotated_coords]
    \draw[->] (-0.5,-0.5,0) --++ (0.5,0,0) node[above] {$x_1$};
    \draw[->] (-0.5,-0.5,0) --++ (0,0.5,0) node[left] {$x_2$};
    \draw[->] (-0.5,-0.5,0) --++ (0,0,0.5);

    \draw (0,0,0) -- (6,0,0) -- (6,4,0) -- (0,4,0) -- cycle;
    \draw (0,1,0) -- (6,1,0);
    \draw (0,1,0) -- (6,2,0);
    \draw (0,2,0) -- (6,3,0);

    \draw[fill = LightBlue!50, draw = black] (0,1,0) -- (6,1,0) -- (6,2,0) -- (0,1,0); 
    \draw[fill = LightBlue!50, draw = black] (0,2,0) -- (6,3,0) -- (6,4,0) -- (0,4,0) -- (0,2,0); 
    \foreach \x in {5,4,3,...,0}{
        \draw[fill = Orange!50, draw = black] (\x,0,0) -- (\x+0.5,0,0) -- (\x+0.5,1,0) -- (\x,1,0) -- (\x,1,0) -- (\x,0,0); 
        \draw[fill = Orange!50, draw = black] (\x,1+\x/6,0) -- (\x+0.5,1+\x/6+1/12,0) -- (\x+0.5,2+\x/6+1/12,0) -- (\x,2+\x/6,0) -- (\x,1+\x/6,0); 
    }

    \foreach \x in {5,4,3,...,0}{
      \draw[fill = gray!40, draw = black] (\x+1,1,-1) -- (\x+1,0,0) -- (\x+0.5,0.2,-0.7) -- (\x+0.5,0.8,-1.3) -- (\x+1,1,-1); 
      \draw[draw = none, fill = gray]  (\x+1,0,0) -- (\x,0,0) -- (\x+0.5,0.2,-0.7) -- (\x+1,0,0); 
      \draw[draw = black, fill = gray] (\x+1,1,-1) -- (\x,1,-1) -- (\x+0.5,0.8,-1.3) -- (\x+1,1,-1); 
      \draw[fill = Orange, draw = black] (\x,0,0) -- (\x+0.5,0.2,-0.7) -- (\x+0.5,0.8,-1.3) -- (\x,1,-1) -- (\x,0,0); 
    }

    \draw[fill = LightBlue, draw = black] (0,1,-1) -- (6,1,-1) -- (6,2,0) -- (0,1,-1);

    \foreach \x in {5,4,3,...,0}{
        \draw[fill = gray!40, draw = black] (\x+1,2+\x/6+1/6,-2+\x/6+1/6) -- (\x+1,1+\x/6+1/6,-1+\x/6+1/6) -- (\x+0.5,1.2+\x/6+1/12,-1.7+\x/6+1/12) -- (\x+0.5,1.8+\x/6+1/12,-2.3+\x/6+1/12) -- (\x+1,2+\x/6+1/6,-2+\x/6+1/6); 
        \draw[draw = black, fill = gray] (\x,1+\x/6,-1+\x/6) -- (\x+0.5,1.2+\x/6+1/12,-1.7+\x/6+1/12) -- (\x+1,1+\x/6+1/6,-1+\x/6+1/6) -- (\x,1+\x/6,-1+\x/6); 
        \draw[draw = black, fill = gray] (\x+1,2+\x/6+1/6,-2+\x/6+1/6) -- (\x,2+\x/6,-2+\x/6) -- (\x+0.5,1.8+\x/6+1/12,-2.3+\x/6+1/12) -- (\x+1,2+\x/6+1/6,-2+\x/6+1/6); 
        \draw[fill = Orange, draw = black] (\x,1+\x/6,-1+\x/6) -- (\x+0.5,1.2+\x/6+1/12,-1.7+\x/6+1/12) -- (\x+0.5,1.8+\x/6+1/12,-2.3+\x/6+1/12) -- (\x,2+\x/6,-2+\x/6) -- (\x,1+\x/6,-1+\x/6); 
    }

    \draw[fill = LightBlue, draw = black] (0,4,0) -- (0,2,-2) -- (6,3,-1) -- (6,4,0) -- (0,4,0);

    \node[canvas is xy plane at z=0, rotate = 270] at (0.25,1.5) {\scriptsize$A_1$};
    \node[canvas is xy plane at z=0, rotate = 270] at (0.75,1.5+1/12) {\scriptsize$A_2$};
    \node[canvas is xy plane at z=0, rotate = 270] at (0.5,3) {\scriptsize$A_3$};
  \end{tikzpicture}

  \caption{Illustration of a laminate within a branching with boundary condition $F_{\alpha} = 0$. This does not depict the exact same situation as in \cref{thm:L1_3wells} for $\nu = e_2$, as here we only show a scalar-valued map.}
  \label{fig:LaminateInBranching}
\end{figure}
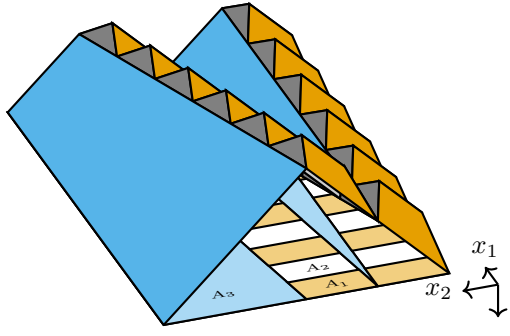

As it is used below in the proof of \cref{cor:Lp_3wells}, we emphasize that the above constructed function satisfies that $u \in W^{1,\infty}(\Omega;\R^2), \chi \in BV(\Omega;\K_3)$ with
\begin{align*}
    \nabla u \in BV(\Omega;\R^{2\times2}), \ \Vert D\nabla u\Vert_{TV(\Omega)} \leq C(\alpha) (\Vert D \chi \Vert_{TV(\Omega)} + \Per(\Omega)),
\end{align*}
and thus the bounds from \cref{eq:PropertiesBranching} also hold in this setting.

\emph{Lower bound:}
Again, by considering $\K_3-F_\alpha$, we reduce to the case of zero boundary conditions.
In the case of $\nu \cdot e_1 = 0$ we have $\nu \cdot e_2 \neq 0$, instead of the high frequency control in $k_1$, we consider a high frequency control in $k_2$ for $\chi_{22}$.
Indeed, let $\mu < \frac{|\nu_2|}{16}$, $\lambda >0$, as in \cref{eq:ineq_inclusion_cones} we have
\begin{align*}
    C_{2,\mu,\lambda} := \{ k \in \Z^{2}: |k|^2 - k_2^2 \leq \mu^2 |k|^2, \ |k_2| \leq \frac{2}{|\nu_2|} \lambda\} \supset \{k \in \Z^{2}: |k|^2 - k_1^2 \leq \mu^2 |k|^2,\ |k \cdot \nu| \leq \lambda\}
\end{align*}
    and thus we can bound the Fourier mass of $\chi_{22}$ outside of $C_{2,\mu,\lambda}$ in terms of the energy, cf. \cref{lem:Localization},
    \begin{align*}
        \sum_{k \notin C_{2,\mu,\lambda}} |\hat{\chi}_{22}(k)|^2 \leq C (\mu^{-2} + (\lambda \epsilon)^{-1}) E_\epsilon(u,\chi) + C \lambda^{-1} \Per(\Omega),
    \end{align*}
    with a constant $C = C(\alpha) > 0$.
    Notice that the definition of the cone is analogous to that in \cref{eq:TruncCone_1}, but different as the second direction plays the role of the first. This is not the cone given in \cref{eq:TruncCone_1} for $j=2$, as there the truncation is dependent on the parameter $|\nu_1|$, here on $|\nu_2|$.
    Following now the proof of \cref{prop:L1_2wells} for $f = \chi_{22}$ with a suitable change of coordinates, we get
    \begin{align*}
        E_{\epsilon}(u,\chi) \geq C |\nu_2|^{\frac{2}{3}} \epsilon^{\frac{2}{3}}.
    \end{align*}
\end{proof}
With \cref{lem:scal-first,lem:scal-second} proved, we combine their estimates to deduce the desired lower bound in \cref{thm:L1_3wells}.

\begin{proof}[Conclusion of \cref{thm:L1_3wells}]
We finally combine the above discussion by rewriting the lower bound estimate in a concise way.
We notice that for $\nu_1 = 0$ there is nothing to prove as \cref{itm:L1_3wells_i,itm:L1_3wells_ii} are given by \cref{lem:scal-first,lem:scal-second} respectively.
If $\nu_1 \neq 0$, we choose $\epsilon_0$ to be small depending on $\nu$, i.e.\
fix $\epsilon_0$ to fulfil
\begin{align*}
    |\nu_2|^{\frac{2}{3}} \epsilon_0^{\frac{2}{3}} \leq |\nu_1|^{\frac{1}{2}} \epsilon_0^{\frac{1}{2}}.
\end{align*}
With this choice, for every $\epsilon<\epsilon_0$, it holds
\begin{align*}
    \inf_{u \in \mathcal{A}_F} \inf_{\chi \in BV_\nu(\Omega;\K_3)} E_\epsilon(u,\chi) \geq C |\nu_1|^{\frac{1}{2}} \epsilon^{\frac{1}{2}} \geq \frac{C}{2} \left( |\nu_1|^{\frac{1}{2}} \epsilon^{\frac{1}{2}} + |\nu_2|^{\frac{2}{3}} \epsilon^{\frac{2}{3}} \right).
\end{align*}
The upper bound follows by adding the two upper bounds from \cref{lem:scal-first,lem:scal-second}.
\end{proof}

\begin{rmk} \label{rmk:lower_bound_component}
    As can be seen in the above proof of \cref{lem:scal-second} (and later analogously in the proof of \cref{thm:L1_mult_wells}), it would be possible to deduce the above scaling behaviour for $\nu \cdot e_1 \neq 0$ also for an even more degenerate anisotropic surface energy of the following type
    \begin{align*}
        E_{\textup{surf}}^{r}(\chi) = \Vert D_\nu (\chi : r) \Vert_{TV(\Omega)},
    \end{align*}
    for $r\in \R^{d \times d}$ such that $|r| = 1$ and
    \begin{align*}
        \Vert D_\nu (\chi : r) \Vert_{TV(\Omega)} \geq C(\K,d, r) \Vert D_{\nu} \chi_{11} \Vert_{TV(\Omega)}.
    \end{align*}
    Here we denoted the Frobenius scalar product of matrices by $\chi : r = \sum_{i,j=1}^{d} \chi_{ij} r_{ij}$.
    A sufficient condition on $r \in \R^{d \times d}$ is the following:
    \begin{align*}
        (A_{j})_{11} \neq (A_{k})_{11} \text{ if and only if } A_j :r \neq A_k : r.
    \end{align*}
    For the three-wells from $\K_3$ in \cref{eq:Lorent_wells} this translates into the condition that
    \begin{align*}
        r \notin \{r_{11} = 0\} \cup \{r_{11} = 2 r_{22}\} \cup \{ r_{11} = -2 r_{22}\}.
    \end{align*}
\end{rmk}

\subsection{Higher order laminates}
We next turn to the scaling behaviour of higher order laminates.
Due to the anisotropic energies, we cannot immediately invoke the argument from \cite{RT23}. As an important technical novelty, we have to treat the relevant nonlinear relations between the components of the phase indicator substantially more carefully. This is due to the fact that, initially, our surface energy (potentially) only controls high frequencies in the $k_1$ direction. Thus, particularly the first step in the localization is crucial, as this will iteratively allow us the further high frequency reduction steps in the other conical directions.

As in the setting of the three-well problem of Lorent in the previous section, the key idea is to use that $\chi_{11}$ determines the other diagonal entries. In \cite{RT23} $\chi_{22}$ is written as a nonlinear polynomial of a combination of the remaining diagonal entries $\chi_{11},\chi_{33},\dots,\chi_{N-1,N-1}$.
  In the case of $\nu = e_1$, if only such a relation were available, this would lead to an issue with our argument as none of the cones in the $e_3, \dots, e_{N-1}$ directions is initially localized in the high frequencies (due to our anisotropic energies).
  The central novel idea is to rely on additional structure: More precisely, we will instead use that $\chi_{22}$ is given by a \emph{nonlinear} polynomial of $\chi_{11}$ plus a \emph{linear} one in $\chi_{33},\dots,\chi_{N-1,N-1}$.

To elaborate on the strategy of the proof of \cref{thm:L1_mult_wells}, let us analyse the four well setting in three dimensions in the case of $\nu \cdot e_1 \neq 0$ first.
  Let $N=4$ and $d=3$, and thus consider $\K_4 = \{A_1, A_2, A_3, A_4\}$ with
  \begin{align*}
  A_1 & = \begin{pmatrix}
      0 & 0 & 0 \\ 0 & 0 & 0 \\ 0 & 0 & 0
  \end{pmatrix}, &
  A_2 &= \begin{pmatrix}
      1 & 0 & 0 \\ 0 & 0 & 0 \\ 0 & 0 & 0
  \end{pmatrix}, &
  A_3 & = \begin{pmatrix}
      \frac{1}{2} & 0 & 0 \\ 0 & 1 & 0 \\ 0 & 0 & 0
  \end{pmatrix}, &
  A_4 &= \begin{pmatrix}
      \frac{1}{2} & 0 & 0 \\ 0 & \frac{1}{2} & 0 \\ 0 & 0 & 1
  \end{pmatrix}.
  \end{align*}
We will subsequently generalize the lower bound to the family of wells from \cref{sec:introL14}.

\begin{prop} \label{prop:L1_4wells}
    Let $\alpha \in (0,1)$ and $F_\alpha = \diag(1/2,1/2,\alpha) \in \K_4^3$. Let $\nu \in \S^2$ with $\nu \cdot e_1 \neq 0$.
    There exist constants $C= C(\alpha)>0$ and $\epsilon_0 = \epsilon_0(\alpha,|\nu_1|) > 0$ such that for any $\epsilon \in (0,\epsilon_0)$
    \begin{align*}
        \inf_{u \in \mathcal{A}_{F_\alpha}} \inf_{\chi \in BV_\nu(\Omega;\K_4)} E_\epsilon (u,\chi) \geq C |\nu_1|^{\frac{2}{5}} \epsilon^{\frac{2}{5}},
    \end{align*}
    where $\mathcal{A}_{F_\alpha}$ is given in  \cref{eq:AdmissibleFunctionsL2}.
\end{prop}

Let us outline the strategy of proof for the derivation of the lower bound in \cref{prop:L1_4wells}.
To show this result, we first use the nonlinear relation for $\tilde{\chi} = \chi - F_\alpha: \Omega \to \K_3 - F_\alpha$
  \begin{align} \label{eq:nonlinRel4Wells}
    2 \tilde{\chi}_{22} +  \tilde{\chi}_{33} = 1-8\tilde{\chi}_{11}^2 - \alpha.
  \end{align}
  After a first application of the first localization \cref{lem:Localization} with
  \begin{align*}
    C_{1,\mu,\lambda} = \{k \in \Z^3: |k|^2-k_1^2 \leq \mu^2 |k|^2, |k_1| \leq \frac{2}{|\nu_1|} \lambda\}, \ C_{j,\mu} = \{ k \in \Z^3: |k|^2-k_j^2 \leq \mu^2 |k|^2\},
  \end{align*}
  for $\mu < \frac{|\nu_1|}{16}$, $\lambda > 0$ and $j=2,3$ and their smoothed out multipliers, cf.\ \cref{eq:multiplier}, we have
  \begin{align*}
    &\Vert\tilde{\chi}_{11} - m_{1,\mu,\lambda}(D) \tilde{\chi}_{11} \Vert_{L^2}^2 + \sum_{j=2,3} \Vert \tilde{\chi}_{jj} - m_{j,\mu}(D) \tilde{\chi}_{jj} \Vert_{L^2}^2\\
    &\qquad\qquad \leq C(\alpha)(\mu^{-2} + (\lambda \epsilon)^{-1}) E_\epsilon(u,\chi) + C(\alpha) \lambda^{-1} \Per(\Omega).
  \end{align*}
  Using the above nonlinear relation \cref{eq:nonlinRel4Wells}, we aim to truncate $C_{2,\mu}$ and $C_{3,\mu}$ by means of \cref{cor:CommutatorEst_Applied}.

\begin{proof}[Proof of \cref{prop:L1_4wells}]
Without loss of generality, by subtracting $F_\alpha$, we assume $F_\alpha = 0$ and $\chi \in \K_4-F_\alpha$.
    We observe that differently from the situation in \cref{rmk:lower_bound_component}, it suffices to only deduce a lower bound for $\chi_{33}$. 
    We aim to use \cref{cor:CommutatorEst_Applied}. After an application of  the first localization \cref{lem:Localization} for $\mu < \frac{|\nu_1|}{16}$, $\lambda > 0$ and a constant $C = C(\alpha) > 0$
    \begin{align}\label{eq:estimate}
    \begin{split}
        &\Vert \chi_{11} - m_{1,\mu,\lambda}(D) \chi_{11} \Vert_{L^2}^2 + \sum_{j=2,3} \Vert \chi_{jj} - m_{j,\mu}(D) \chi_{jj} \Vert_{L^2}^2 \\
        &\qquad\qquad \leq C (\mu^{-2} + (\lambda \epsilon)^{-1}) E_\epsilon(u,\chi) + C \lambda^{-1} \Per(\Omega).
    \end{split}
    \end{align}
    To achieve the energy bound, we use \cref{cor:CommutatorEst_Applied} to truncate the cone corresponding to $\chi_{22}$, which in turn is used to reduce the size of $C_{3,\mu}$ even further.
    In detail we apply \cref{cor:CommutatorEst_Applied} compounded with \cref{eq:estimate} for $2\chi_{22} +\chi_{33} = 1 - \alpha - 8 \chi_{11}^2$ and get for $\lambda_2 = M \mu \lambda$
    \begin{align} \label{eq:iteration_1_4wells}
    \Vert \chi_{22} - m_{2,\mu,\lambda_2}(D) \chi_{22} \Vert_{L^2}^{2} \leq C \psi_{1-\gamma}\Big( (\mu^{-2} + (\lambda \epsilon)^{-1}) E_\epsilon(u,\chi) + C \lambda^{-1} \Per(\Omega)\Big).
    \end{align}
    Here we used the truncated cone $C_{2,\mu,\lambda_2}$ defined in \cref{eq:TruncCone_1} with the corresponding smooth multiplier defined in \cref{eq:multiplier}.

    Iterating now the comparison argument, using $\chi_{33} = 1- \alpha - 4\chi_{22}^2$ we can improve the estimate on the Fourier mass of $\chi_{33}$ in the sense that we can truncate the corresponding cone $C_{3,\mu}$ on a scale $\lambda_3 = M \mu \lambda_2 = M^2 \mu^2 \lambda$.
    Indeed, by a variant of the commutation estimate from \cref{cor:CommutatorEst_Applied} (see \cref{rmk:corollary-variant}) and \cref{eq:iteration_1_4wells}
\begin{align*}
    \Vert \chi_{33} - m_{3,\mu,\lambda_3}(D) \chi_{33} \Vert_{L^2}^{2} &\leq C \psi_{1-\gamma}\Big(\Vert \chi_{22}-m_{2,\mu,\lambda_2}(D) \chi_{22} \Vert_{L^2}^{2}\Big) + 4 \Vert \chi_{33} - m_{3,\mu}(D) \chi_{33} \Vert_{L^2}^{2} \\
    & \leq C \psi_{(1-\gamma)^2}\Big((\mu^{-2}+(\lambda \epsilon)^{-1})E_\epsilon(u,\chi) + \lambda^{-1} \Per(\Omega)\Big).
\end{align*}
As above, the multiplier is of the form as in \cref{eq:multiplier} for the truncated cone $C_{3,\mu,\lambda_3}$ defined in \cref{eq:TruncCone_1}.

    By \cref{lem:LowFreqEst} with $\bar{\lambda} = 4 \lambda_3/|\nu_1|$ we can control the ``missing'' (low frequency) region and infer
    \begin{align*}
        \Vert \chi_{33} \Vert_{L^2}^2 &\leq 2 \Vert \chi_{33}- m_{3,\mu,\lambda_3} \chi_{33} \Vert_{L^2}^2 + 2 \sum_{|k_3| \leq 4 \lambda_3/|\nu_1|} |\hat{\chi}_{33}(k)|^2 \\
        & \leq C \psi_{(1-\gamma)^2}\Big((\mu^{-2}+(\lambda \epsilon)^{-1}) E_\epsilon(u,\chi) + \lambda^{-1} \Per(\Omega)\Big) + C \frac{\lambda_3^2}{|\nu_1|^2} E_\epsilon(u,\chi).
    \end{align*}
    We now choose the parameters $\mu$ and $\lambda_3$, that is, $\lambda$, in an optimal way.
    To be more precise, we fix $\mu^{-1} \sim \frac{\lambda_3}{|\nu_1|} \sim \frac{\mu^2 \lambda}{|\nu_1|}$, i.e. $\mu \sim |\nu_1|^{1/3} \lambda^{-1/3}$ and get
    \begin{align*}
        \Vert \chi_{33} \Vert_{L^2}^{2} & \leq C \psi_{(1-\gamma)^2}\Big((|\nu_1|^{-2/3} \lambda^{2/3} + (\lambda \epsilon)^{-1})E_\epsilon(u,\chi) + \lambda^{-1} \Per(\Omega)\Big) + |\nu_1|^{-2/3} \lambda^{2/3} E_\epsilon(u,\chi) \\
        & \leq C \psi_{(1-\gamma)^2} \Big((|\nu_1|^{-2/3} \lambda^{2/3} + (\lambda \epsilon)^{-1}) E_\epsilon(u,\chi) + \lambda^{-1} \Per(\Omega)\Big).
    \end{align*}
    Optimizing the energy contributions in $\lambda$ by choosing $\lambda \sim |\nu_1|^{2/5} \epsilon^{-3/5}$ yields, after absorbing the perimeter term,
    \begin{align*}
        E_\epsilon(u,\chi) \geq C |\nu_1|^{\frac{2}{5}} \epsilon^{\frac{2}{5}}.
    \end{align*}
    Here we used that $\Vert \chi_{33} \Vert_{L^2}^{2} \geq C \min^2\{\alpha,1-\alpha\}$ and the fact that the function $\psi_{(1-\gamma)^2}$ does not influence the scaling behaviour, cf. the proof of \cref{lem:scal-first}.
\end{proof}

\begin{proof}[Proof of \cref{thm:L1_mult_wells}]
With the previous results in hand, the remainder of the proof of \cref{thm:L1_mult_wells} is exactly the same as in \cite{RT23} with the same modifications as above for four wells. For the convenience of the reader we recall the main ideas.
Without loss of generality, we reduce to the case of $F = 0$ by considering $\chi - F \in \K_N - F$.
We first clarify the nonlinear relations. Each component determines the following ones, that is
\begin{align*}
    \sum_{n=j+1}^{N-1} \alpha_k \chi_{nn} = g_j(\chi_{jj}).
\end{align*}
The relations are given as follows

\begin{align*}
    \sum_{n=j+1}^{N-1} 2^{-n+j+1} \chi_{nn}= 4 F_{jj}-4F_{jj}^2 - \sum_{n=j+1}^{N-1} 2^{-n+j+1} F_{nn} +4\chi_{jj} - 8 F_{jj} \chi_{jj} - 4 \chi_{jj}^2 =: g_j(\chi_{jj}).
\end{align*}

\emph{The case $\nu \cdot e_1 \neq 0$:}
By \cref{lem:Localization} we obtain for a constant $C = C(d,F) > 0$
\begin{align} \label{eq:Localization_mult_wells}
\begin{split}
    &\Vert \chi_{11} - m_{1,\mu,\lambda}(D) \chi_{11} \Vert_{L^2}^{2} + \sum_{j=2}^{d} \Vert \chi_{jj} - m_{j,\mu}(D) \chi_{jj} \Vert_{L^2}^{2} \\
    &\qquad \leq C (\mu^{-2} + (\lambda \epsilon)^{-1}) E_\epsilon(u,\chi) + C \lambda^{-1} \Per(\Omega).
    \end{split}
\end{align}
To facilitate the reading, we recall the definition of the truncated cones $C_{j,\mu,\lambda_2}$ in \cref{eq:TruncCone_1}
\begin{align*}
    C_{j,\mu,\lambda_2} = \{k \in \Z^d: |k|^2 - k_j^2 \leq \mu^2 |k|^2, |k_j| \leq \frac{2}{|\nu_1|} \lambda_2\},
\end{align*}
and that $m_{j,\mu,\lambda_2}(D)$ denote their smooth Fourier multipliers as in \cref{eq:multiplier}.
By \cref{cor:CommutatorEst_Applied}, in combination with \cref{eq:Localization_mult_wells}, we have for $\lambda_2 = M \mu \lambda$
\begin{align*}
\sum_{j=2}^{d} \Vert \chi_{jj} - m_{j,\mu,\lambda_2}(D) \chi_{jj} \Vert_{L^2}^{2} \leq C 2^d \psi_{1-\gamma}\Big((\mu^{-2}+(\lambda \epsilon)^{-1}) E_\epsilon(u,\chi)  + \lambda^{-1} \Per(\Omega)\Big).
\end{align*}
Here we exploited that $\alpha_k \in [2^{-d},1]$ for $k \in \{2,3,\dots,d\}$.
Then an iterative application of the higher-order variant of \cref{cor:CommutatorEst_Applied} (\cref{rmk:corollary-variant}) yields, in combination with \cref{eq:Localization_mult_wells}, after $\ell-1$ many iterations
\begin{align} \label{eq:iterated_commutator_mult_wells}
    \sum_{j=\ell}^{d} \Vert \chi_{jj} - m_{j,\mu,\lambda_{\ell}}(D) \chi_{jj} \Vert_{L^2}^{2} \leq  C \psi_{(1-\gamma)^{\ell-1}}\Big((\mu^{-2} + (\lambda \epsilon)^{-1}) E_\epsilon(u,\chi) +\lambda^{-1} \Per(\Omega)\Big),
\end{align}
with $\lambda_\ell = M \mu \lambda_{\ell-1} = M^{\ell-1} \mu^{\ell-1} \lambda$ and a constant $C = C(d,F,\ell) > 0$.

As we have already seen in the proof of \cref{thm:L1_3wells} for $F \in \K_3^1$, it might happen, that we need information on more than one $\chi_{jj}$, cf. \cref{rmk:lower_bound_first_laminate}.
Depending on $\ell \in \{1,\dots,N-2\}$ we have for some $\alpha \in (0,1)$, given by the boundary data,
\begin{align*}
    (\chi_{\ell \ell},\chi_{\ell+1,\ell+1}) \in \{ (-\alpha,0), (1-\alpha,0), (\frac{1}{2}-\alpha,1), (\frac{1}{2}-\alpha,\frac{1}{2})\}.
\end{align*}
In particular, $(\chi_{\ell \ell}, \chi_{\ell+1,\ell+1}) \neq 0$.
For $\ell = N-1$, we already have $\chi_{N-1,N-1} \in \{-\alpha, 1-\alpha\}$ with $\alpha \in (0,1)$, it holds $\chi_{N-1,N-1} \neq 0$.

Thus, if $N = d+1$, setting $\chi_{NN} \equiv 0$, we obtain that if $F \in \K_N^{\ell}$ for some $\ell \in \{1,2,\dots,N-1\}$, then it holds $(\chi_{\ell \ell},\chi_{\ell+1,\ell+1}) \neq 0$. Hence, we note that the control for two diagonal components $\chi_{jj}$ is sufficient for deducing the desired lower bound.

By \cref{eq:iterated_commutator_mult_wells} we have
\begin{align*}
    \Vert \chi_{\ell \ell} - m_{\ell,\mu,\lambda_{\ell}}(D) \chi_{\ell \ell} \Vert_{L^2}^{2} &+ \Vert \chi_{\ell+1,\ell+1} - m_{\ell+1,\mu,\lambda_{\ell}}(D) \chi_{\ell+1,\ell+1} \Vert_{L^2}^{2} \\
    & \leq C \psi_{(1-\gamma)^{\ell-1}}\Big((\mu^{-2} + (\lambda \epsilon)^{-1}) E_\epsilon(u,\chi) + \lambda^{-1} \Per(\Omega)\Big).
\end{align*}

As in the proof of \cref{prop:L1_4wells} we use \cref{lem:LowFreqEst} to gain control over $\Vert m_{\ell,\mu,\lambda_{\ell}}(D) \chi_{\ell \ell} \Vert_{L^2}^2$ and $\Vert m_{\ell+1,\mu,\lambda_{\ell}}(D) \chi_{\ell+1,\ell+1} \Vert_{L^2}^2$.
For this let $\bar{\lambda} = 4 \lambda_{\ell}/|\nu_1|$, then
\begin{align*}
    \Vert m_{\ell,\mu,\lambda_{\ell}}(D) \chi_{\ell \ell} \Vert_{L^2}^2 &+ \Vert m_{\ell+1,\mu,\lambda_{\ell}}(D) \chi_{\ell+1,\ell+1} \Vert_{L^2}^2 \\
    & \leq \sum_{|k_{\ell}|  \leq 4 \lambda_{\ell}/|\nu_1|} |\hat{\chi}_{\ell \ell}(k)|^2 + \sum_{|k_{\ell+1}| \leq 4 \lambda_{\ell}/|\nu_1|} |\hat{\chi}_{\ell+1,\ell+1}(k)|^2 \\
    & \leq C \frac{\lambda_{\ell}^2}{|\nu_1|^2} E_{\epsilon}(u,\chi).
\end{align*}
These two estimates combined give
\begin{align*}
    \Vert \chi_{\ell \ell} \Vert_{L^2}^2 + \Vert \chi_{\ell+1,\ell+1} \Vert_{L^2}^2
    &\leq 2 \Vert \chi_{\ell \ell} - m_{\ell,\mu,\lambda_{\ell}}(D) \chi_{\ell \ell} \Vert_{L^2}^2 + 2 \Vert \chi_{\ell+1,\ell+1} - m_{\ell+1,\mu,\lambda_{\ell}}(D) \chi_{\ell+1,\ell+1} \Vert_{L^2}^2 \\
    & \quad + 2\Vert m_{\ell,\mu,\lambda_{\ell}}(D) \chi_{\ell \ell} \Vert_{L^2}^2 + 2\Vert m_{\ell+1,\mu,\lambda_{\ell}}(D) \chi_{\ell+1,\ell+1} \Vert_{L^2}^2  \\
    & = C \psi_{(1-\gamma)^{\ell-1}}\Big((\mu^{-2} + (\lambda \epsilon)^{-1}) E_\epsilon(u,\chi) + \lambda^{-1} \Per(\Omega)\Big) + C \frac{\lambda_{\ell}^2}{|\nu_1|^2} E_\epsilon(u,\chi) \\
    & \leq C \psi_{(1-\gamma)^{\ell-1}}\Big((\frac{\lambda_{\ell}^2}{|\nu_1|^2} + \mu^{-2} + (\lambda \epsilon)^{-1}) E_\epsilon(u,\chi) + \lambda^{-1} \Per(\Omega)\Big).
\end{align*}
Arguing as in the proof of \cref{prop:L1_4wells}, we optimize this in $\mu$ and $\lambda$ by fixing $\mu \sim |\nu_1|^{1/\ell} \lambda^{-1/\ell}$ and $\lambda \sim |\nu_1|^{2/(\ell+2)} \epsilon^{-\ell/(\ell+2)}$ (such a choice of parameters is compatible with the constraints on $\mu$ and $\bar\lambda$) and thus
\begin{align*}
    \Vert \chi_{\ell \ell} \Vert_{L^2}^2 + \Vert \chi_{\ell+1,\ell+1} \Vert_{L^2}^2 \leq C \psi_{(1-\gamma)^{\ell-1}}\Big(|\nu_1|^{-2/(\ell+2)} \epsilon^{-2/(\ell+2)} E_\epsilon(u,\chi) + |\nu_1|^{2/(\ell+2)} \epsilon^{\ell/(\ell+2)} \Per(\Omega)\Big).
\end{align*}
Using that
\begin{align*}
    \Vert \chi_{\ell \ell} \Vert_{L^2}^2 + \Vert \chi_{\ell+1,\ell+1} \Vert_{L^2}^2 \geq C(F) > 0,
\end{align*}
after arguing similarly as in the proof of \cref{prop:L1_2wells}, that is considering the two cases for $\psi_{(1-\gamma)^{\ell-1}}$ and absorbing the perimeter, we obtain with $C = C(d,F,\ell) > 0$ and for $\epsilon < \epsilon_0(d,F,\ell,|\nu_1|)$

\begin{align} \label{eq:HigherOrderScaling_1}
    E_\epsilon(u,\chi) \geq C |\nu_1|^{\frac{2}{\ell+2}} \epsilon^{\frac{2}{\ell+2}}.
\end{align}

\emph{The case $\nu \cdot e_1 = \dots = \nu \cdot e_{n} = 0$ and $\nu \cdot e_{k+1} \neq 0$, with $0 < n < \ell$:}
We argue similarly, but we start the iterative application of \cref{cor:CommutatorEst_Applied} for the $n+1$ diagonal entry $\chi_{n+1,n+1}$.

Note that due to the change of roles of the coordinate directions, the truncated cones in this setting are given by
\begin{align*}
    C_{j,\mu,\lambda}= \{k \in \Z^d: |k|^2-k_j^2 \leq \mu^2 |k|^2, |k_j| \leq \frac{2}{|\nu_n|} \lambda\},
\end{align*}
i.e. the truncation depends on $|\nu_n|$ instead of $|\nu_1|$ as above in \cref{thm:L1_3wells} in the case $\nu \cdot e_1 = 0$.

By \cref{lem:Localization}, with the roles of the axes changed such that $n+1$ is the first coordinate direction, we get
\begin{align*}
    &\quad \Vert \chi_{n+1,n+1} - m_{n+1,\mu,\lambda}(D) \chi_{n+1,n+1} \Vert_{L^2}^{2} + \sum_{j=2}^d \Vert \chi_{n+j,n+j} - m_{n+j,\mu}(D) \chi_{n+j,n+j} \Vert_{L^2}^{2} \\
    &\leq C (\mu^{-2} + (\lambda \epsilon)^{-1}) E_\epsilon(\chi) + C \lambda^{-1} \Per(\Omega).
\end{align*}
Applying an analogous iteration of a variant of \cref{cor:CommutatorEst_Applied} as above, now starting at $j = n+1$ instead of $j=1$, we deduce
\begin{align*}
    & \quad \sum_{j=\ell-n}^d \Vert \chi_{n+j,n+j} - m_{k+j,\mu,\lambda_{\ell-n}}(D) \chi_{n+j,n+j} \Vert_{L^2}^{2} \\
    &\leq C \psi_{(1-\gamma)^{\ell-n-1}}\Big((\mu^{-2} + (\lambda \epsilon)^{-1}) E_\epsilon(u,\chi) + \lambda^{-1} \Per(\Omega)\Big).
\end{align*}
Concluding as above for $\nu \cdot e_1 \neq 0$ to get \cref{eq:HigherOrderScaling_1}, while taking the off-set in the index into account, i.e. having $\ell-n$ instead of $\ell$, yields (where the constant can be chosen to be independent of $n$)
\begin{align*}
    E_\epsilon(u,\chi) \geq C |\nu_{n+1}|^{\frac{2}{\ell-n+2}} \epsilon^{\frac{2}{\ell-n+2}}.
\end{align*}
Here the scaling depends on $\nu_{n+1}$ instead of $\nu_1$ as above in \cref{eq:HigherOrderScaling_1} as we start our arguments with $\chi_{n+1,n+1}$ instead of $\chi_{11}$ and with the corresponding high frequency control in the direction $e_{n+1}$. This is due to the assumption that $\nu \cdot e_{n+1} \neq 0$.

\emph{The case $\nu \cdot e_1 = \dots = \nu \cdot e_{\ell} = 0$:}
If $\nu \cdot e_j = 0$ for $j=1,\dots,\ell$, we consider an $\ell$-th order simple laminate of arbitrary fine oscillations, as those directions are not penalized, yielding a minimizing sequence, with energies converging to zero.
We refer to the proof of \cref{prop:L1_2wells} for $\nu \cdot e_1 = 0$ for a similar setting.

\emph{Conclusion of the proof.}
As above, to combine the derived estimates we choose $\epsilon$ sufficiently small depending on $\nu$, that is for $0 \leq n < \ell$ being the index such that $\nu_{n+1} \neq 0$, $\nu_j = 0$ for $j \leq n$, we fix $\epsilon$ such that for all $n < j < \ell \leq d$
\begin{align*}
    |\nu_{j+1}|^{\frac{2}{\ell-j+2}} \epsilon^{\frac{2}{\ell-j+2}} \leq |\nu_{n+1}|^{\frac{2}{\ell-n+2}} \epsilon^{\frac{2}{\ell-n+2}},
\end{align*}
which is possible as $2/(\ell-j+2) > 2/(\ell-n+2)$.
Hence, we arrive at
\begin{align*}
    |\nu_{n+1}|^{\frac{2}{\ell-n+2}} \epsilon^{\frac{2}{\ell-n+2}} \geq \frac{1}{\ell-n} \left( \sum_{j=n}^{\ell-1} |\nu_{j+1}|^{\frac{2}{\ell-j+2}} \epsilon^{\frac{2}{\ell-j+2}} \right) \geq \frac{1}{\ell} \left( \sum_{j=0}^{\ell-1} |\nu_{j+1}|^{\frac{2}{\ell-j+2}} \epsilon^{\frac{2}{\ell-j+2}} \right),
\end{align*}
where we used that $\nu_{j+1} = 0$ for $0 \leq j < n$.
\end{proof}

We conclude our discussion of sharp interface models by highlighting that the above arguments also allow us to treat situations in which the set $T_{4}$ is given by the Tartar square (see \cref{sec:intro_appl}). This is of particular interest as it is an instance of an extremely rigid phase transition involving laminates of infinite order.

\begin{rmk}[On the Tartar square with anisotropic sharp interface energies]
    With the same arguments as above, it is possible to generalize the scaling of the Tartar square, cf. \cref{eq:Tartar}, that was derived in \cite{RT22} to a setting involving anisotropic surface energies.
    To be precise, for any $F \in T_{4}^{qc} \setminus T_4$, $\nu \in \S^1$ and $\eta \in (0,\frac{1}{2})$ it holds
    \begin{align*}
        \inf_{u \in \mathcal{A}_F} \inf_{\chi \in BV_\nu((0,1)^2;T_4)} \int_\Omega |\nabla u - \chi|^2 dx + \epsilon \Vert D_\nu \chi \Vert_{TV((0,1)^2)} \geq C \exp ( -  c_\eta |\log(\epsilon)|^{\frac{1}{2}+\eta}),
    \end{align*}
    for some constant $c_\eta>0$.
    Compared to the settings from above, there are no degenerate directions for the lower scaling bound in the Tartar square.
    This is due to the fact, that there are no rank-one connections present in the Tartar square and that, hence, each diagonal entry determines the corresponding other entry uniquely.
    As a consequence, if $\nu \cdot e_1 \neq 0$, we start by truncating the cone in $e_1$ direction, else we have $\nu \cdot e_2 \neq 0$ and thus can start by truncating the cone in $e_2$ direction.
\end{rmk}

\subsection{Proof of \cref{thm:frac_energy}}
\label{sec:frac_proof}

Building on the observations from \cref{sec:prelim_frac}, we conclude the proof of \cref{thm:frac_energy} analogously as in the (local) sharp interface arguments.

\begin{proof}[Proof of \cref{thm:frac_energy}]
\emph{Lower bound.}
We note that for $t_1,t_2 \in (0,1)$
\begin{align*}
    \psi_{t_1} \circ \psi_{t_2} (x) = \psi_{t_1t_2}(x),
\end{align*}
and thus
\begin{align*}
    \psi_{1-\gamma} \circ \psi_{2s} = \psi_{(1-\gamma)2s}.
\end{align*}
Setting $\gamma_s = 1-2s+2s \gamma \in (0,1)$, we combine \cref{cor:CommutatorEst_Applied,lem:LocalizationFrac} and get with a constant $C = C(d,F,s) > 0$
\begin{align*}
    \sum_{j=2}^d |\alpha_j| \Vert \chi_{jj}-m_{j,\mu,\lambda}(D)\chi_{jj}\Vert_{L^2}^{2} &\leq C \psi_{1-\gamma} \circ \psi_{2s}((\mu^{-2}+(\lambda \epsilon)^{-1})E_{\epsilon,s}(\chi)) \\
    &= C \psi_{1-\gamma_s}((\mu^{-2}+(\lambda\epsilon)^{-1})E_{\epsilon,s}(\chi)).
\end{align*}
Therefore, the remainder of the proof follows as the proof of \cref{thm:L1_mult_wells} also for the fractional surface energy.

\emph{Upper bound.}
Thanks to an interpolation argument, we prove that we can directly exploit the $BV$-regular upper bounds $\chi$ of \cref{thm:L1_mult_wells}.
For this, we argue similarly to \cite[Prop. 1.3]{BO13} but taking into account the anisotropy of the surface energy.
Let $\chi\in BV_\nu(\Omega;\K_N)$.
We first claim that for $\chi$ seen as a function on the torus $\T^d$ it holds
\begin{align*}
    \sum_{k \in \Z^d \setminus \{0\}} |k \cdot \nu|^{2s} |\hat{\chi}(k)|^2 \leq C(s) \int_{-\frac{1}{2}}^{\frac{1}{2}} \int_{\T^d} \frac{|\chi(x+h\nu)-\chi(x)|^2}{|h|^{1+2s}} dx dh + \Vert \chi \Vert_{L^2}^2.
\end{align*}
Indeed, this follows from observing that by Plancherel's theorem
\begin{align*}
    \int_{-\frac{1}{2}}^{\frac{1}{2}} \int_{\T^d} \frac{|\chi(x+h\nu)-\chi(x)|^2}{|h|^{1+2s}} dx dh = \int_{-\frac{1}{2}}^{\frac{1}{2}} \sum_{k \in \Z^d} \frac{|e^{2\pi i h k \cdot \nu}-1|^2}{|h|^{1+2s}} |\hat{\chi}(k)|^2 dh,
\end{align*}
and by noting that for $k \in \Z^d$ with $|k \cdot \nu|>1$ we have $(2|k\cdot \nu|)^{-1} \leq 2^{-1}$ and hence
\begin{align*}
    B_\nu(k) := |k \cdot \nu|^{-2s} \int_{-\frac{1}{2}}^{\frac{1}{2}} \frac{|e^{2\pi i h k \cdot \nu}-1|^2}{|h|^{1+2s}} dh \geq 8 |k \cdot \nu|^{-2s} \int_{0}^{\frac{1}{2|k \cdot \nu|}} \frac{\sin^2(\pi h |k \cdot \nu|)}{|h|^{1+2s}} dh \geq C(s)^{-1} > 0
\end{align*}
for a constant $C(s)$ independent of $k$.
With this in hand, by monotone convergence we have
\begin{align*}
    \sum_{k \in \Z^d \setminus\{0\}} |k \cdot \nu|^{2s} |\hat{\chi}(k)|^2 & = \sum_{|k \cdot \nu|>1} |k \cdot \nu|^{2s} |\hat{\chi}(k)|^2 + \sum_{|k \cdot \nu|\leq 1} |k \cdot \nu|^{2s} |\hat{\chi}(k)|^2 \\
    & \leq C(s) \sum_{|k \cdot \nu|>1} |k \cdot \nu|^{2s} B_\nu(k) |\hat{\chi}(k)|^2 + \sum_{|k \cdot \nu|\leq 1} |\hat{\chi}(k)|^2 \\
    & \leq C(s) \int_{-\frac{1}{2}}^{\frac{1}{2}} \int_{\T^d} \frac{|\chi(x+h\nu)-\chi(x)|^2}{|h|^{1+2s}} dx dh + \Vert \chi \Vert_{L^2}^2.
\end{align*}
We now fix $\tilde{\chi}: \R^d \to \K_N$ by setting
\begin{align*}
    \tilde{\chi}(x) := \begin{cases} \chi(x) & x \in (-1,2)^d, \\
    0 & \text{else}.
    \end{cases}
\end{align*}
Noting that for all $x \in (0,1)^d$ and $h \in [-\frac{1}{2},\frac{1}{2}]$ it holds that $x + h \nu \in (-1,2)^d$, by the definition of $\tilde{\chi}$ we infer that
\begin{align*}
    \sum_{k \in \Z^d \setminus \{0\}} |k \cdot \nu|^{2s} |\hat{\chi}(k)|^2 \leq C(s) \int_{\R} \int_{\R^d} \frac{|\tilde{\chi}(x+h\nu)-\tilde{\chi}(x)|^2}{|h|^{1+2s}} dxdh + \Vert \chi \Vert_{L^2}^2.
\end{align*}
Using a slicing argument together with a Gagliardo-Nirenberg type inequality \cite[Theorem 1]{BM18} (see also \cite[Theorem 2]{RZZ19})
\begin{align*}
    \int_\R \int_{\R^d} \frac{|\tilde{\chi}(x+h\nu) - \tilde{\chi}(x)|^2}{|h|^{1+2s}} &= \int_{\nu^\perp} \int_{\R} \int_{\R} \frac{|\tilde{\chi}_y^\nu(t+h)-\tilde{\chi}_y^\nu(t)|^2}{|h|^{1+2s}} dhdtdy \\
    &\leq C(s,d,\K) \int_{\nu^\perp} (\Vert \tilde{\chi}^\nu_y\Vert_{L^1(\R)} + \Vert D\tilde{\chi}^\nu_y \Vert_{TV(\R)})^{2s} dy.
\end{align*}
By the uniformly compact support of the functions $\tilde{\chi}$ on each slice and as $2s < 1$, we can further bound this by Jensen's inequality
\begin{align*}
    \int_{\nu^\perp} (\Vert \tilde{\chi}^\nu_y\Vert_{L^1(\R)} + \Vert D\tilde{\chi}^\nu_y \Vert_{TV(\R)})^{2s} dy &\leq C \Big( \int_{\nu^\perp} \Vert \tilde{\chi}^\nu_y\Vert_{L^1(\R)} + \Vert D\tilde{\chi}^\nu_y \Vert_{TV(\R)} dy \Big)^{\frac{1}{2s}} \\
    &= C \Big( \Vert \tilde{\chi} \Vert_{L^1(\R^d)} + \Vert D_\nu \tilde{\chi} \Vert_{TV(\R^d)} \Big)^{\frac{1}{2s}}.
\end{align*}
Thus, by the relation between $\chi$ and $\tilde{\chi}$, we conclude
\begin{align*}
    \sum_{k \in \Z^d \setminus \{0\}} |k \cdot \nu|^{2s} |\hat{\chi}(k)|^2 \leq C(s,d,\K)(1 + \Per(\Omega) + \Vert D_\nu \chi \Vert_{TV(\Omega)})^{\frac{1}{2s}}.
\end{align*}
In particular, this implies
\begin{align*}
    E_{\epsilon,s}(\chi) \leq C(s,d,\K) E_{\epsilon}(\chi) + C(s,d,\K) \epsilon (1+\Per(\Omega)).
\end{align*}
Thus, the matching upper bounds for the sharp interface model in \cref{thm:L1_mult_wells} are still applicable which yields the desired result.

\end{proof}

\section{Diffuse surface energies}
\label{sec:diffuse}

\subsection{Diffuse to sharp interface model -- the lower bound} \label{sec:diffuse1}
Following the arguments of \cite{KK11}, we show that in the continuous model,
the minimal diffuse energy can be controlled by the lower bound for the sharp interface energy. Hence,
the lower scaling estimates for the sharp interface model which had been deduced in the previous section for various anisotropic situations also give rise to lower scaling estimates for the diffuse energy.
We show this lower bound for any \emph{discrete set} of wells and also for a more general framework covering both the gradient and symmetrized gradient settings.
Combined with the upper bounds from the next subsection, this will lead to sharp scaling results in our model problems also for anisotropic, diffuse interface energies.

\subsubsection{The scalar-valued setting}
We first show the lower bound in the scalar-valued case, and afterwards will reduce the general vector-valued setting to the one-dimensional one by a projection argument.
\begin{lem} \label{lem:LowerDiffSharp_1D}
  Let $p,q \in [1,\infty)$ and $\Omega \subset \R^d$ be bounded and let $\K = \{A_{1}, \dots, A_{N}\} \subset \R$.
  For any $U \in L^p(\Omega;\R)\cap W^{1,q}(\Omega;\R)$, $f \in L^\infty(\Omega;\K)$ there exist $\tilde{f} \in BV(\Omega;\K)$ and a constant $C = C(\K,p,q) > 0$ such that for any $\epsilon >0$
  \begin{align}
  \label{eq:comparison_diff_sharp}
    \int_{\Omega} |U-f|^{p} + \epsilon^q |\nabla U|^{q} dx \geq C \epsilon \Vert D \tilde{f} \Vert_{TV(\Omega)}
    \quad\text{and}\quad
    |U-f| \geq C |U-\tilde{f}|.
  \end{align}
  For $q = 1$, the same result holds with $\Vert D U \Vert_{TV(\Omega)}$ instead of $\|\nabla U\|_{L^1(\Omega)}$ (on the left-hand side of \cref{eq:comparison_diff_sharp}), for $U \in BV(\Omega;\R)$ instead of $U \in W^{1,1}(\Omega;\R)$.
\end{lem}

\begin{proof}
  Up to relabelling, we can assume that $-\infty < A_1 < A_2 < \dots < A_N < + \infty$.
We start by giving the argument for the first inequality in \cref{eq:comparison_diff_sharp}.
We assume for simplicity that $U$ coincides with its Lebesgue representative.
In this case, for scalar-valued functions, the validity of the coarea formula
for Sobolev functions is known (cf.\ \cite{MSZ03}).
We, hence, invoke Young's inequality for $q>1$ and the coarea formula to get
\begin{align*}
  \int_\Omega |U-f|^p + \epsilon^q |\nabla U|^q dx & \geq \epsilon \int_\Omega |U-f|^{\frac{p(q-1)}{q}} |\nabla U| dx = \epsilon \int_\R \int_{U^{-1}(t) \cap \Omega} |t-f|^{\frac{p(q-1)}{q}}
                                                        d \mathcal{H}^{d-1}(x) dt \\
  & \geq \epsilon \sum_{k=1}^{N-1} \int_{A_{k}+\frac{c}{4}}^{A_{k+1}-\frac{c}{4}} \int_{U^{-1}(t) \cap \Omega} |t-f|^{\frac{p(q-1)}{q}} d\mathcal{H}^{d-1}(x) dt,
\end{align*}
where in the last inequality we introduced $c = \min \{ A_{k+1}-A_k : k = 1, \dots, N-1\}$.

As for $t \in (A_k + \frac{c}{4},A_{k+1}-\frac{c}{4})$ we have $\dist(t,\K) \geq \frac{c}{4}$, we can control $|t-f(x)|^{\frac{p(q-1)}{q}} \geq C$ with a constant $C = C(\K,p,q) > 0$.
Plugging this lower bound into the above inequality yields
\begin{align*}
  \int_\Omega |U-f|^p + \epsilon^q |\nabla U|^q dx & \geq C \epsilon \sum_{k=1}^{N-1} \int_{A_k + \frac{c}{4}}^{A_{k+1} - \frac{c}{4}} \mathcal{H}^{d-1}(\{x \in \Omega: U(x) =t\}) dt.
\end{align*}

For every $k \in \{1,\dots,N-1\}$ there is $t_k \in (A_k+\frac{c}{4}, A_{k+1}-\frac{c}{4})$, which, without loss of generality, we may assume to be a Lebesgue point for the function $t \mapsto \mathcal{H}^{d-1}(\{x
\in \Omega: U(x) = t \})$, such that
\begin{align*}
  \int_{A_k+\frac{c}{4}}^{A_{k+1}-\frac{c}{4}} \mathcal{H}^{d-1}(\{x \in \Omega: U(x) = t\}) dt &\geq (A_{k+1}-A_k - \frac{c}{2}) \mathcal{H}^{d-1}(\{x \in \Omega: U(x) = t_k\}) \\
  &\geq \frac{c}{2}
  \mathcal{H}^{d-1}(\{x \in \Omega: U(x) = t_k\}).
\end{align*}
Therefore, combined with the previous bound, we obtain
\begin{align*}
  \int_\Omega |U-f|^p + \epsilon^q |\nabla U|^q dx \geq C \epsilon \sum_{k=1}^{N-1} \mathcal{H}^{d-1}(\{x \in \Omega: U(x) = t_k\}).
\end{align*}
In particular, we obtain a sequence of $t_k$ such that $t_k < t_{k+1}$ for all $k \in \{1,\dots,N-2\}$.

The idea now is to define a new phase indicator with BV-seminorm which is exactly determined by these measures.
To this end let $h: \R \to \R$ be given as $h(t)= \sum_{k=1}^{N} A_k \chi_{(t_{k-1},t_k]}(t) \in \K$ and with slight abuse of notation, we set $t_0 := -\infty, t_N := \infty$. We define $\tilde{f} = (h \circ U) : \Omega \to \K$, and
claim that $\tilde{f} \in BV(\Omega;\K)$ and that it satisfies
\begin{align*}
  \Vert D \tilde{f} \Vert_{TV(\Omega)} \leq C \sum_{k=1}^{N-1} \mathcal{H}^{d-1}(\{x \in \Omega: U(x) = t_k\}).
\end{align*}

The claim will be shown by an approximation argument. To this end, let $h_j \in C^\infty(\R;\R)$ be a mollification of $h$ fulfilling $\supp(h-h_j) \subset \bigcup_{k=1}^{N-1}
(t_k-\frac{1}{j},t_k+\frac{1}{j})$, $|h_j'| \leq C j \chi_{\supp(h-h_j)}$, and $|h_j| \leq \max\K$.
By construction, it holds that $h_j(t) \to h(t)$ for every $t\neq t_k$, $k=1,\dots, N-1$.
Since the sets $\{U=t_k\}$ have Hausdorff dimension $d-1$, by the boundedness of $\Omega$, and
boundedness of $h_j$, the dominated convergence theorem implies $h_j \circ U  \to h \circ U = \tilde{f}$ in $L^1(\Omega)$.
As the BV-seminorm is lower semicontinuous with respect to the (strong) $L^1$ convergence, we deduce
\begin{align*}
  \Vert D \tilde{f} \Vert_{TV(\Omega)} \leq \liminf_{j \to \infty} \Vert D (h_j \circ U) \Vert_{TV(\Omega)}.
\end{align*}

As $h_j$ is smooth, we can use the chain rule to further bound the total variation norm.
Indeed, by the coarea formula
\begin{align*}
  \Vert D (h_j \circ U) \Vert_{TV(\Omega)} &= \int_\Omega |h_j'(U(x))| |\nabla U(x)| dx = \int_\R |h_j'(t)| \mathcal{H}^{d-1}(\{x \in \Omega: U(x) = t\}) dt \\
  & \leq C j \sum_{k=1}^{N-1} \int_{t_k-\frac{1}{j}}^{t_k + \frac{1}{j}} \mathcal{H}^{d-1}(\{x \in \Omega: U(x) = t\}) dt \\
  & \to \frac{C}{2} \sum_{k=1}^{N-1} \mathcal{H}^{d-1}(\{x \in \Omega: U(x) = t_k\}).
\end{align*}
Here we used that the $t_k$ are Lebesgue points of the function $t \mapsto \mathcal{H}^{d-1}(\{x \in \Omega: U(x) = t \})$.

Thus, with this definition, $\tilde{f} \in BV(\Omega;\K)$ and also the desired upper bound from \cref{eq:comparison_diff_sharp} follows.

If $q = 1$, we do not use Young's inequality, but neglect the first term. We give the argument for $U \in BV(\Omega;\R)$ and note that the statement for $U \in W^{1,1}(\Omega;\R)$ is also covered by this.
By the coarea formula for BV functions \cite[Thm 3.40]{AFP00}
\begin{align*}
    \int_\Omega |U - f|^p dx &+ \epsilon \Vert DU \Vert_{TV(\Omega)} \geq \epsilon \Vert DU \Vert_{TV(\Omega)} = \epsilon \int_{\R} \Vert D \chi_{\{U > t\}} \Vert_{TV(\Omega)} dt \\
    &\geq \epsilon \sum_{k=1}^{N-1} \int_{A_{k}}^{A_{k+1}} \Vert D \chi_{\{U>t\}} \Vert_{TV(\Omega)} dt \geq \epsilon \sum_{k=1}^{N-1} |A_{k+1}-A_k| \Vert D \chi_{\{U> t_{k}\}} \Vert_{TV(\Omega)},
\end{align*}
for some $t_k \in (A_k,A_{k+1})$.
Considering the function $h: \R \to \R$ and $\tilde{f} = h \circ U$ as above, we write $h$ as
\begin{align*}
    h(t) = A_1 + \sum_{k=1}^{N-1} (A_{k+1}-A_k) \chi_{(t_k,\infty)}(t),
\end{align*}
and thus it holds
\begin{align*}
    \tilde{f}(x) = h \circ U(x) &= A_{1} + \sum_{k=1}^{N-1} (A_{k+1}-A_k) \chi_{\{U>t_k\}}(x).
\end{align*}
In particular, we have
\begin{align*}
    \epsilon \Vert D \tilde{f} \Vert_{TV(\Omega)} & \leq \epsilon \sum_{k=1}^{N-1} |A_{k+1}-A_k| \ \Vert D \chi_{\{U>t_k\}}\Vert_{TV(\Omega)} \\
    & \leq \int_\Omega |U-f|^p dx + \epsilon \Vert DU\Vert_{TV(\Omega)}.
\end{align*}

Now we prove the second inequality from \cref{eq:comparison_diff_sharp}. 
For this, let $x \in \Omega$ and $k \in \{1,\dots,N-1\}$ be such that $U(x) \in (\frac{A_k+A_{k-1}}{2}, \frac{A_k+A_{k+1}}{2}]$, where we set $A_0 = -\infty$ and $A_{N+1} = \infty$. In particular, $|U(x)-A_k| =
\dist(U(x),\K)$.
We distinguish two cases, which are illustrated in \cref{fig:LowerDiffSharp_1D}.

On the one hand, if $U(x) \in (t_{k-1},t_k]$, we have $\tilde{f}(x) = A_k$ and thus $|U(x)-\tilde{f}(x)| = \dist(U(x),\K) \leq |U(x)-f(x)|$.

On the other hand, if $U(x) \notin (t_{k-1},t_k]$, we have either
\begin{align*}
U(x) > t_k > A_k + \frac{c}{4}\quad \text{or} \quad U(x) \leq t_{k-1} < A_k - \frac{c}{4},
\end{align*}
and hence $\dist(U,\K) = |U(x)-A_k| > \frac{c}{4}$.
This means that 
$|U(x)- f(x)| \geq \frac{c}{4}$ and furthermore
\begin{align*}
    |U(x) - \tilde{f}(x)| \leq |U(x)-f(x)| + |f(x) -\tilde{f}(x)| \leq |U(x)-f(x)| + \diam(\K) \leq C |U(x)-f(x)|.
\end{align*}
This concludes the proof. 

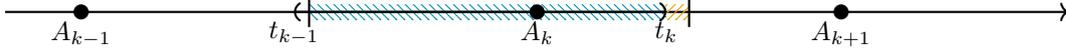
\begin{figure}

 \centering
  \begin{tikzpicture}[thick]
    \draw[->] (-7,0) -- (7,0);

    \fill[draw=none, pattern = north west lines, pattern color = Blue] (-3,-0.1) rectangle (1.7,0.1);
    \fill[draw=none, pattern = north east lines, pattern color = Orange] (1.7,-0.1) rectangle (2,0.1);

    \fill (0,0) circle (0.1) node[below] {$A_{k}$};
    \fill (-6,0) circle (0.1) node[below] {$A_{k-1}$};
    \fill (4,0) circle (0.1) node[below] {$A_{k+1}$};

    \draw[{(-)}] (-3.2,0) node[below] {$t_{k-1}$} -- (1.7,0) node[below] {$t_k$};

    \draw (-3,0.2) -- (-3,-0.2);
    \draw (2,0.2) -- (2,-0.2);

  \end{tikzpicture}
  \caption{Picture of the two cases which arise in our proof of the second estimate in \cref{eq:comparison_diff_sharp}. The blue hashed region depicts the first case, where $\tilde{f}(x) = A_k$, the orange region the second case, in which $\dist(U(x),\K)
    \geq \frac{c}{4}$. The vertical dashes are the center points between $A_{k-1}$ and $A_k$ and $A_{k}$ and $A_{k+1}$ respectively.}
   \label{fig:LowerDiffSharp_1D}
\end{figure}
\end{proof}

Seeking to deduce the same scaling laws from \cref{sec:intro_L1} for diffuse surface energies,
we also generalize \cref{lem:LowerDiffSharp_1D} to anisotropic surface energies:
\begin{cor} \label{cor:DiffuseToSharp_Aniso_1D}
  Let $p,q \in [1,\infty)$, $\nu \in \S^{d-1}$ and $\Omega \subset \R^d$ be a bounded Lipschitz domain and let $\K = \{A_1,\dots,A_N\} \subset \R$. For any $U \in L^p(\Omega;\R)$ such that the weak directional
  derivative satisfies $\p_{\nu} U \in L^q(\Omega;\R)$ and $f \in L^\infty(\Omega; \K)$ there exist $\tilde{f} \in BV_{\nu}(\Omega;\K)$ and a constant $C = C(\K,p,q) > 0$ such that for any $\epsilon >0$
  \begin{align*}
    \int_\Omega |U-f|^p + \epsilon^q |\p_{\nu} U |^q dx \geq C \epsilon \Vert D_{\nu} \tilde{f} \Vert_{TV(\Omega)},
    \quad \text{and}\quad
    \int_\Omega|U-f|^p dx \geq C \int_\Omega|U - \tilde{f}|^p dx.
  \end{align*}
\end{cor}

\begin{proof}
The proof uses similar methods as above, with the addition of a splitting in direction $\nu$ and $\nu^\perp$.
We start by considering $U\in W^{1,p}(\R^d;\R)$.
For almost every $y \in \Omega_\nu$ (see \cref{sec:prelim_directBV} for the notation), we follow the arguments from the previous proof and obtain
  \begin{align*}
    \int_\Omega |U-f|^p + \epsilon^q |\p_\nu U|^q dx& \geq \epsilon \int_{\Omega_\nu} \int_{\Omega^\nu_y} |U^\nu_y(t)-f^\nu_y(t)|^{\frac{p(q-1)}{q}} |(U_y^\nu)'(t)| dt dy \\
    &\geq C \epsilon
    \int_{\Omega_\nu} \sum_{k=1}^{N-1} \mathcal{H}^0(\{s \in \Omega^\nu_y: U^\nu_y(s) = t_k^{y}\}) dy,
  \end{align*}
  where, for a.e.\ $y\in\Omega_\nu$, $t_k^y\in(A_k+\frac{c}{4},A_{k+1}-\frac{c}{4})$, with $c$ defined as in the proof of \cref{lem:LowerDiffSharp_1D}.

  The main difference with respect to the isotropic case, is that here we use a slicing argument, hence in the definition of the modified phase indicator $\tilde f$, the thresholds $t_k^y$ change from slice to slice, giving in principle measurability issues.
  To avoid this, we define $\tilde f$ to be piecewise constant in the variable $y$ as follows.

  Given $\delta>0$, we define $\mathcal{L}_\delta:=\{z\in\delta\Z^{d-1} : Q_\delta(z)\subset\Omega_\nu\}$, where $Q_\delta(z)=(0,\delta)^{d-1}+z$, and we use the notation $\mathcal{L}_\delta=\{z_j\}_{j=1}^{M}$, with $M=\#\mathcal{L}_\delta$ (the dependence on $\nu$ is omitted for the sake of clarity of exposition).
  From the above estimate, by an average argument, we get
  \begin{align}\label{eq:slicing-MM1}
  \begin{split}
      \int_\Omega |U-f|^p + \epsilon^q |\p_\nu U|^q dx& \geq C\epsilon \sum_{j=1}^M \int_{Q_\delta(z_j)}\sum_{k=1}^{N-1}\mathcal{H}^0(\{s \in \Omega^\nu_y: U^\nu_y(s) = t_k^y\}) dy\\
      &\ge C\epsilon\delta^{d-1} \sum_{j=1}^M \sum_{k=1}^{N-1}\mathcal{H}^0(\{s \in \Omega^\nu_{y_j}: U^\nu_{y_j}(s) = t_k^{y_j}\})
      \end{split}
  \end{align}
  with $y_j\in Q_\delta(z_j)$.
  In every cube, we now work as in the proof of \cref{lem:LowerDiffSharp_1D}, namely we define $h_j=\sum_{k=1}^N A_k\chi_{(t_{k-1}^{y_j},t_k^{y_j}]}$ and $\tilde f_j^\nu:=h_j\circ U_{y_j}^\nu$.
  Notice that as $U \in W^{1,p}(\R^d;\R)$ the functions $\tilde f_j^\nu$ are defined on the whole $\R$.
  Working as in the previous proof, for every $j$ we have
  \begin{equation}\label{eq:slicing-MM2}
    \|D\tilde f_j^\nu\|_{TV(\Omega_{y_j}^\nu)} \le C\sum_{k=1}^{N-1}\mathcal{H}^0(\{s \in \Omega^\nu_{y_j}: U^\nu_{y_j}(s)  = t_k^{y_j}\}).
    \end{equation}
  Defining
  \begin{align*}
  \tilde f(x) :=\begin{cases}
    \tilde f^\nu_j(t), &\Pi_{\nu^\perp}x=y\in Q_\delta(z_j),\, x=y+t\nu,\, j\in\{1,\dots,M\} \\
    A_1 & \text{otherwise},
  \end{cases}
  \end{align*}
  by slicing $\tilde f\in BV_\nu(\Omega;\K)$ and combining \cref{eq:slicing-MM1,eq:slicing-MM2} we get
  \begin{align*}
  \epsilon\|D_\nu\tilde f\|_{TV(\Omega)} &= \epsilon\sum_{j=1}^M \|D_\nu\tilde f\|_{TV(C_\delta(z_j))} = \epsilon\delta^{d-1}\sum_{j=1}^M \|D\tilde f_j^{\nu}\|_{TV(\Omega_{y_j}^\nu)}\\
  &\le C\int_\Omega|U-f|^p+\epsilon^q|\p_\nu U|^qdx,
  \end{align*}
  where $C_\delta(z_j):=\{x\in\Omega : \Pi_{\nu^\perp}x\in Q_\delta(z_j)\}.$

  For $U \in BV_\nu(\R^d;\R)$ we can then also follow the arguments in the proof of \cref{lem:LowerDiffSharp_1D} combined with the methods above to show the estimate.

We now prove the estimate regarding elastic energies.
Notice preliminarily that we can assume that $U\neq f$ in $L^p$ otherwise we can simply choose $\tilde
f=U$.

By definition of $\tilde f_j^\nu$, arguing as in the proof of \cref{lem:LowerDiffSharp_1D} the pointwise estimate holds $|U_{y_j}^\nu(t)-\tilde f_j^\nu(t)|\le C\dist (U_{y_j}^\nu(t),\K)$ a.e.\ in $\R$.
Hence, applying the triangle inequality twice, we may write
\begin{align*}
    &\int_{Q_\delta(z_j)}\int_{\Omega_y^\nu}|U_y^\nu(t)-(\tilde f)_y^\nu(t)|^p dt dy \\
    &\qquad \le 2^{p-1}\int_{Q_\delta(z_j)}\int_{\Omega_y^\nu}|U_y^\nu(t)-U_{y_j}^\nu(t)|^p dt dy + 2^{p-1}\int_{Q_\delta(z_j)}\int_{\Omega_y^\nu}|U_{y_j}^\nu(t)-\tilde f_j^\nu(t)|^p dt dy \\
    &\qquad \le 2^p\int_{Q_\delta(z_j)}\int_{\Omega_y^\nu}|U_y^\nu(t)-U_{y_j}^\nu(t)|^p dt dy + 2^{p-1}\int_{Q_\delta(z_j)}\int_{\Omega_{y}^\nu}\dist^p(U_y^\nu(t),\K) dt dy
\end{align*}
Eventually, since $U\in W^{1,p}(\R^d;\R)$ by assumption, we infer
\begin{align*}
    &\int_{Q_\delta(z_j)}\int_{\Omega_y^\nu}|U_y^\nu(t)-U_{y_j}^\nu(t)|^pdtdy = \int_{Q_\delta(z_j)}\int_{\Omega_y^\nu} |U(y+t\nu)-U(y_j+t\nu)|^p dtdy \\
    &\qquad\le \sup_{|h|\le\delta\sqrt{d-1}}\int_{C_\delta(z_j)\cap\Omega_\delta}|U(x)-U(x+h)|^pdx+C\int_{C_\delta(z_j)\setminus\Omega_\delta}|U(x)|^pdx \\
    &\qquad\le C \delta^p\|\nabla U\|_{L^p(\Omega_\delta^j)}^p+C\|U\|^p_{L^p(C_\delta(z_j)\setminus\Omega_\delta)},
\end{align*}
where $\Omega_\delta:=\{x\in\Omega:\dist(x,\p\Omega)>2\sqrt{d-1}\delta\}$,
and where we used the notation $\Omega_\delta^j=(B^{d-1}_{2\delta\sqrt{d-1}}(z_j)\times\R\nu)\cap\Omega$ which intersect a finite number of times.
Gathering the two above inequalities, and summing over $j$ we obtain
\begin{align*}
    \int_\Omega|U(x)-\tilde f(x)|^pdx\le C\int_\Omega\dist^p(U(x),\K)dx+C\delta^p\|\nabla U\|_{L^p(\Omega)}^p+C\int_{\Omega\setminus\Omega_\delta}|U(x)|^pdx + C |\Omega \setminus \Omega_\delta|.
\end{align*}
Taking $\delta$ sufficiently small, the result follows for every $U\in W^{1,p}(\Omega;\R)$.
By an approximation argument, let $U_n\in W^{1,p}(\R^d;\R)$ converge to $U$ strongly in $L^p(\Omega)$ and $\p_\nu U_n$ converge to $\p_\nu U$ strongly in $L^q(\Omega)$ and let $\tilde f_n$ be their phase indicators (defined as above), we have that
\begin{align*}
\int_\Omega |U-\tilde f_n|^pdx \le C\int_\Omega|U-U_n|^pdx +C\int_\Omega|U_n-\tilde f_n|^pdx.
\end{align*}
By taking $n$ sufficiently large (and $\delta>0$ sufficiently small such
  that $C\delta^p\|\nabla U\|_{L^p(\Omega)}^p+C\|U\|_{L^p(\Omega\setminus\Omega_\delta)} + C |\Omega \setminus \Omega_\delta|<\frac{\eta}{2}$), the result follows:
\begin{align*}
&\int_\Omega |U-\tilde f_n|^pdx \le \eta +C\int_\Omega\dist^p(U_n(x),\K)dx\\
&\le \eta +C\int_\Omega\dist^p(U(x),\K)dx + C\int_\Omega|U-U_n|^pdx
\le 2 \eta +C\int_\Omega\dist^p(U(x),\K)dx ,
\end{align*}
after choosing $\eta \leq \int_\Omega\dist^p(U(x),\K)dx$.
Analogously, for $n$ sufficiently large we get
\begin{align*}
\|D_\nu\tilde f_n\|_{TV(\Omega)} \le C\int_\Omega|U_n-f|^p+\epsilon^q|\p_\nu U_n|^qdx \le C\int_\Omega|U-f|^p+\epsilon^q|\p_\nu U|^qdx+C \eta
\end{align*}
which gives the result.
\end{proof}

\subsubsection{The vector-valued setting}
In this section, we translate the results from the previous section to the vector-valued setting.

\begin{prop} \label{prop:DiffToSharp_Full}
  Let $p,q \in [1,\infty)$ and $\Omega \subset \R^d$ be bounded and let $\K =
  \{A_1,\dots,A_N\} \subset \R^n$. For any $U \in L^p(\Omega;\R^n)\cap W^{1,q}(\Omega;\R^n)$ and $\chi \in L^\infty(\Omega;\K)$ there exist $\tilde{\chi} \in BV(\Omega;\K)$ and a constant $C = C(\K,p,q) > 0$ such that for any $\epsilon >0$
  \begin{align*}
    \int_\Omega |U-\chi|^p + \epsilon^q |\nabla U|^q dx \geq C \epsilon \Vert D \tilde{\chi} \Vert_{TV(\Omega)}
    \quad\text{and}\quad
    |U-\chi|  \geq C  |U-\tilde{\chi}|.
  \end{align*}
  In particular,
  \begin{align*}
    \int_\Omega |U-\chi|^p + \epsilon^q |\nabla U|^q dx & \geq C \left( \int_\Omega |U-\tilde{\chi}|^p dx + \epsilon \Vert D \tilde{\chi} \Vert_{TV(\Omega)} \right).
  \end{align*}
\end{prop}

\begin{proof}
  The main idea is to invoke \cref{lem:LowerDiffSharp_1D} for a suitable projection of the wells onto a one-dimensional subspace in which there are still $N$ distinct wells.
  For this we choose
  \begin{align*}
    \zeta \in \mathbb{S}^{n-1} \setminus \left( \bigcup_{i \neq j} (A_i-A_j)^\perp \right)\neq \emptyset.
  \end{align*}
  By this choice of $\zeta$, we know that $A_i \cdot \zeta \neq A_j \cdot \zeta$ for any $i,j \in \{1,\dots,N\}, i \neq j$, hence there is a one-to-one correspondence of $\K' = \{A_1 \cdot \zeta, A_2 \cdot
  \zeta, \dots, A_N \cdot \zeta\}$ and $\K$.
  For further use in what follows below, we note that due to the discreteness of $\K$ and $\K'$ we can define a Lipschitz map
  \begin{align}
  \label{eq:map_h}
  h\in C^1(\R;\R^n) \mbox{ with } h(A_j \cdot \zeta) = A_j \mbox{ for all } j \in \{1,2,\dots,N\}.
  \end{align}

  Now we project the energy onto the direction $\zeta$ in the image:
  \begin{align*}
    \int_\Omega |U-\chi|^p + \epsilon^q |\nabla U|^q dx \geq \int_\Omega |U \cdot \zeta - \chi \cdot \zeta|^p + \epsilon^q |\nabla (U \cdot \zeta)|^q dx.
  \end{align*}

  The functions $U\cdot \zeta: \R^d \to \R$ and $\chi \cdot \zeta: \R^d \to \K'$ are admissible for \cref{lem:LowerDiffSharp_1D}, and thus there is $\tilde{f} \in BV(\Omega;\K')$ and $C =C(\K,p,q)>0$ such that
  \begin{align}
  \label{eq:first_step}
    \int_\Omega |U-\chi|^p + \epsilon^q |\nabla U|^q dx \geq C \epsilon \Vert D \tilde{f} \Vert_{TV(\Omega)}, \quad |U\cdot \zeta- \chi \cdot \zeta| \geq C |U \cdot \zeta - \tilde{f}|.
  \end{align}
  Using the determinedness of $\K$ in terms of $\K'$, i.e. using the function $h:\R \to \R^n$ from \cref{eq:map_h}, and the chain rule in BV \cite[Thm 3.96]{AFP00} hence shows that $\tilde{\chi} = h \circ \tilde{f} \in BV(\Omega;\K)$ and
  \begin{align*}
    \Vert D \tilde{\chi}\Vert_{TV(\Omega)} \leq C(\K) \Vert D \tilde{f} \Vert_{TV(\Omega)}.
  \end{align*}
  Combining this bound with \cref{eq:first_step} yields
  \begin{align*}
    \int_\Omega |U-\chi|^p + \epsilon^q |\nabla U|^q dx \geq C \epsilon \Vert D \tilde{\chi} \Vert_{TV(\Omega)}.
  \end{align*}

  The pointwise bound follows with the same argument as in the one-dimensional case. Let $x \in \Omega$ and $k \in \{1,\dots,N\}$ be such that $\dist(U(x),\K) = |U(x)-A_k|$.
  If $\tilde{\chi}(x) = A_k$, it is direct that $|U(x)-\chi(x)| \geq |U(x)-A_k| = |U(x)-\tilde{\chi}(x)|$.
  If $\tilde{\chi}(x) \neq A_k$, also $\tilde{f}(x) \neq A_k \cdot \zeta$, and we can deduce that $|U(x) \cdot \zeta - A_{k} \cdot \zeta| \geq C$ and thus also $\dist(U(x),\K) \geq C$ and hence
  \begin{align*}
    |U(x)-\tilde{\chi}(x)|& \leq |U(x)-\chi(x)| + |\chi(x)-\tilde{\chi}(x)| \leq |U(x)-\chi(x)| + \diam(\K) \\
    & \leq |U(x) - \chi(x)| + c \dist(U(x),\K) \leq C |U(x)-\chi(x)|.
  \end{align*}
  The two cases are illustrated in \cref{fig:LowerDiffSharp}.
  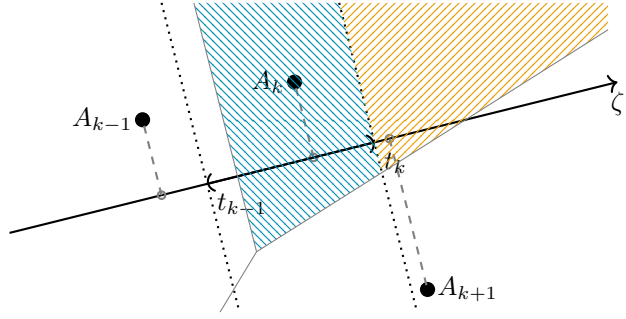
\begin{figure}
    \centering
    \begin{tikzpicture}[thick]
      \draw[->] (-4,-1) -- (4,1) node[below] {$\zeta$};

      \fill (-1/4,1) circle (0.1) node[left] {$A_{k}$};
      \fill (-2.25,0.5) circle (0.1) node[left] {$A_{k-1}$};
      \fill (1.5,-1.75) circle (0.1) node[right] {$A_{k+1}$};

      \draw[dashed, gray] (0,0) -- (-1/4,1); 
      \draw[dashed, gray] (1,0.25) -- (1.5,-1.75); 
      \draw[dashed, gray] (-2,-0.5) -- (-2.25,0.5); 

      \draw[gray] (0,0) circle (0.05);
      \draw[gray] (1,0.25) circle (0.05);
      \draw[gray] (-2,-0.5) circle (0.05);

      \draw[{(-)}] (-1.4,-0.35) node[below right] {$t_{k-1}$} -- (0.8,0.2) node[below right] {$t_k$};
      \draw[dotted] (-119/80-41/80,2.05) -- (-119/80+41/80,-2.05);
      \draw[dotted] (17/20-41/80,2.05) -- (17/20+41/80,-2.05);

      \fill[draw = none, pattern = north west lines, pattern color = Blue] (-63/40,2.05) -- (-3/4,-5/4) -- (9/10,-1/5) -- (17/20-41/80,2.05) -- cycle;
      \fill[draw = none, pattern = north east lines, pattern color = Orange] (17/20-41/80,2.05) -- (31/8,2.05) -- (31/8,149/88) -- (9/10,-1/5) -- cycle;

      \draw[gray,thin] (-63/40,2.05) -- (-3/4,-5/4);

      \draw[gray,thin] (31/8,149/88) -- (-3/4,-5/4);

      \draw[gray,thin] (-1.23,-2.05) -- (-3/4,-5/4);

    \end{tikzpicture}
    \caption{Picture of the two cases to show the estimate on the elastic energy in \cref{prop:DiffToSharp_Full}. The hashed region is the set in which $U(x)$ lies within. In particular, the blue region illustrates the first case, in which $\tilde{\chi} =
    A_k$. The orange region illustrates the second case, in which $\dist(U(x),\K) \geq C$. The smaller circles denote the projections $A_{k-1} \cdot \zeta, A_k \cdot \zeta, A_{k+1} \cdot \zeta$. The thin lines are the
    boundaries of the Voronoi-regions. }
    \label{fig:LowerDiffSharp}
  \end{figure}
\end{proof}

Also this result can be translated to anisotropic surface energies which proves \cref{thm:comparison}.

\begin{proof}[Proof of \cref{thm:comparison}]
  We split the energy into the $r$ component and the remainder.
  By an application of \cref{cor:DiffuseToSharp_Aniso_1D} it holds
  \begin{align*}
    \int_\Omega |U-\chi|^p + \epsilon^q |\p_\nu (U \cdot r)|^q dx
    & \geq \int_\Omega |U\cdot r -
    \chi \cdot r|^p + \epsilon^q |\p_\nu (U \cdot r)|^q dx\\
    & \geq C (\int_\Omega |U \cdot r -
    f_{r}|^p dx + \epsilon\Vert D_\nu f_{r} \Vert_{TV(\Omega)}),
  \end{align*}
  for some $f_{r}: \Omega \to \{A_1 \cdot r, A_2 \cdot r, \dots, A_N \cdot r\}$.
  We now define $\tilde{\chi} \in BV_{\nu}(\Omega, \K)$ by setting $\tilde{\chi} \cdot r = f_{r}$. To fix the part
  perpendicular to $r$, we choose $\Pi_{r^\perp} \tilde{\chi} = \tilde{\chi} - (\tilde{\chi} \cdot r) r$ as the projection
  of $\Pi_{r^\perp} U$ onto $\Pi_{r^\perp} \K = \{A_1 - (A_1 \cdot r) r, \dots, A_N -
  (A_N \cdot r) r\}$ with the additional constraint that $\tilde{\chi} = f_{r} r + \Pi_{r^\perp} \tilde{\chi} \in \K$.
  That is $\tilde{\chi}$ is chosen such that
  \begin{align*}
      |U(x) - \tilde{\chi}(x)| = \min \{|U(x)-A_k|: A_k \cdot r = f_{r}(x)\}, \quad \text{for a.e. } x\in\Omega.
  \end{align*}
  Then we have
  \begin{align*}
    \int_\Omega\left| U- \chi \right|^p dx \geq c \int_\Omega\left| U - \tilde{\chi} \right|^pdx.
  \end{align*}
    Indeed, by \cref{cor:DiffuseToSharp_Aniso_1D} and the pointwise estimate in the orthogonal part we get
\begin{align*}
\int_\Omega|U-\chi|^pdx &= \int_\Omega|U\cdot r-\chi\cdot r|^p+|\Pi_{r^\perp}U-\Pi_{r^\perp}\chi|^pdx \\
&\ge \int_\Omega|U\cdot r-f_r|^p + \dist^p(\Pi_{r^\perp}U,\Pi_{r^\perp}\K) dx = \int_\Omega |U-\tilde\chi|^pdx.
\end{align*}
  In conclusion, by \cref{cor:DiffuseToSharp_Aniso_1D}, we obtain
  \begin{align*}
    \int_\Omega |U-\chi|^p + \epsilon^q |\p_\nu (U\cdot r)|^q dx &\geq C \int_\Omega
    |U - \tilde{\chi}|^p + |U \cdot r - \chi \cdot r|^p + \epsilon^q
                                                                      |\p_\nu (U \cdot r)|^q dx \\
    & \geq C \int_\Omega |U - \tilde{\chi}|^p dx + C \epsilon \Vert D_\nu f_{r} \Vert_{TV(\Omega)} \\
    & \geq C \left( \int_\Omega |U - \tilde{\chi}|^p dx + \epsilon \Vert D_{\nu} (\tilde{\chi}
      \cdot r) \Vert_{TV(\Omega)} \right).
  \end{align*}
\end{proof}

\subsection{Diffuse to sharp interface model -- the upper bound} \label{sec:diffuse2}

Based on the estimates in \cref{prop:DiffToSharp_Full,thm:comparison} in this section we provide a complementary upper bound. Due to the presence of ``lower order errors'', on its own, this upper bound does not show that the diffuse and sharp interface models display the same scaling behaviour. However, with the knowledge of known upper bounds (e.g. in our prototypical model scenarios), it implies that for all our applications the diffuse and sharp interface models admit the same $\epsilon$ scaling.

\begin{lem} \label{lem:DiffToSharp_Mollified}
    Let $\Omega = (0,1)^d$, $p,q \in [1,\infty)$. For any $U \in L^p(\R^d;\R^n)\cap BV(\R^d;\R^n)$ with $U = 0$ outside $\Omega$, $\chi \in L^\infty(\R^d;\R^n) \cap BV(\R^d;\R^n)$, and $\epsilon >0$ there exist $U_\epsilon \in C^\infty_0(\Omega;\R^n)$ 
    and a constant $C = C(\Vert \chi \Vert_{L^\infty}, p, q, d) > 0$ such that
    \begin{align*}
        \int_{\Omega} |U_\epsilon - \chi|^p + \epsilon^q |\nabla U_\epsilon|^q dx &\leq C \Big( \int_{\Omega} |U - \chi|^p dx + \epsilon (1+\max\{\Vert U \Vert_{L^\infty}^{p-1}, \Vert U \Vert_{L^\infty}^{q-1}\})\Vert D \chi \Vert_{TV(\Omega)} \\
        & \quad +  \epsilon \max\{\Vert U \Vert_{L^\infty}^{p-1}, \Vert U \Vert_{L^\infty}^{q-1}\} ( \Vert DU - D\chi \Vert_{TV(\R^d)} +  \Per(\Omega)) \Big).
    \end{align*}
    Moreover, if we have $\operatorname{curl} U= 0$ or $\operatorname{curl} \operatorname{curl} U = 0$ in $\R^d$ it also holds $\operatorname{curl} U_{\epsilon} = 0$ or $\operatorname{curl} \operatorname{curl} U_\epsilon = 0$ in $\R^d$, respectively.
\end{lem}

We highlight that similar arguments would also work for more general $\mathcal{A}$-free differential inclusions as in \cite{RRTT24}.

\begin{proof}
We split the proof into two parts, first presenting the construction of the main estimate without adjusting the boundary conditions of the function $U_{\epsilon}$ and then correcting for this.

\emph{Step 1: The estimate without the boundary conditions.}
    Let us start by giving the argument for $U \in L^{p}(\R^d;\R^n) \cap BV(\R^d;\R^n)$ with $U = 0$ outside $\Omega$ and $\chi \in L^\infty(\R^d;\R^n) \cap BV(\R^d;\R^n)$, but without preserving the ``boundary condition''. We use this to show the full result with the preserved boundary conditions in what follows below.
    We set $U_\epsilon = U \ast \psi_\epsilon$, where $\psi_\epsilon$ is a radially symmetric mollifier with $\psi_\epsilon(x)=\epsilon^{-d}\psi(\epsilon^{-1}x)$, $\int_{\R^d} \psi_\epsilon dx = 1$, $\supp( \psi_\epsilon) \subset B_\epsilon(0)$.

    First let us consider some auxiliary results.
    By Young's convolution inequality, it holds for any $f \in L^r(\R^d;\R^n)$, $1 \leq r \leq \infty$, that
    \begin{align} \label{eq:DiffToSharp_Young}
        \Vert \nabla (f \ast \psi_\epsilon) \Vert_{L^r} \leq \Vert \nabla \psi_\epsilon \Vert_{L^1} \Vert f \Vert_{L^r} \leq \epsilon^{-1} \Vert \nabla \psi \Vert_{L^1} \Vert f \Vert_{L^r},
    \end{align}
    where we used
    \begin{align*}
        \int_{\R^d} |\nabla \psi_\epsilon(x)| dx = \int_{\R^d} \epsilon^{-d-1} |\nabla \psi(x/\epsilon)| dx = \epsilon^{-1} \Vert \nabla \psi \Vert_{L^1}.
    \end{align*}

    Next, since $U \in BV(\R^d;\R^n)$,
    \begin{align*}
        \int_{\R^d} U_\epsilon(x) \cdot \di \phi(x) dx & = \int_{\R^d} \int_{\R^d} U(x) \psi_\epsilon(y) \cdot \di \phi(x+y) dy dx \\ & = \int_{\R^d} U(x) \cdot \di(\phi \ast \psi_\epsilon)(x) dx \leq \Vert D U \Vert_{TV(\R^d)},
    \end{align*}
    and thus
    \begin{align} \label{eq:DiffToSharp_Gradient}
        \Vert \nabla U_\epsilon \Vert_{L^1} \leq \Vert D U \Vert_{TV(\R^d)}.
    \end{align}
    The same then holds for $\chi$, and $U - \chi$.
    As a last ingredient, we observe that
    \begin{align} \label{eq:DiffToSharp_ConvRate}
        \Vert \chi \ast \psi_\epsilon - \chi \Vert_{L^1} \leq C \epsilon \Vert D \chi \Vert_{TV(\R^d)}.
    \end{align}

    With this we can prove the claim.
    Starting with the diffuse elastic energy, we have
    \begin{align*}
        \int_{\R^d} |U_\epsilon - \chi|^p dx &\leq C(p) \left( \Vert (U-\chi) \ast \psi_\epsilon \Vert_{L^p}^p + \Vert \chi \ast \psi_\epsilon - \chi \Vert_{L^p}^p \right) \\
        & \leq C(p) \left( \Vert U - \chi\Vert_{L^p}^p + \Vert \chi \ast \psi_\epsilon - \chi \Vert_{L^\infty}^{p-1} \Vert \chi \ast \psi_\epsilon - \chi \Vert_{L^1} \right) \\
        & \leq C(p) \left( \Vert U - \chi \Vert_{L^p}^p + \epsilon \Vert \chi \Vert_{L^\infty}^{p-1} \Vert D \chi \Vert_{TV(\R^d)}\right),
    \end{align*}
    where in the last step we applied \cref{eq:DiffToSharp_ConvRate}.

    For the diffuse surface energy, 
    by \cref{eq:DiffToSharp_Young,eq:DiffToSharp_Gradient} we obtain that

    \begin{align*}
        \int_{\R^d} |\nabla U_\epsilon|^q dx & \leq \Vert \nabla U_\epsilon \Vert_{L^\infty}^{q-1} \Big( \Vert \nabla(U-\chi) \ast \psi_\epsilon \Vert_{L^1} + \Vert \nabla \chi \ast \psi_\epsilon \Vert_{L^1} \Big) \\ & \leq \epsilon^{1-q} \Vert \nabla \psi \Vert_{L^1}^{q-1} \Vert U \Vert_{L^\infty}^{q-1} \Big( \Vert DU - D\chi \Vert_{TV(\R^d)} + \Vert D\chi \Vert_{TV(\R^d)} \Big).
    \end{align*}

    As a consequence of the two inequalities above we obtain that
    \begin{align} \label{eq:DiffToSharp_NoBdry}
    \begin{split}
        \int_{\R^d} |U_\epsilon - \chi|^p + \epsilon^q |\nabla (U_\epsilon)|^q dx & \leq C(\Vert \chi \Vert_{L^\infty},p,q) \Big( \int_{\R^d} |U-\chi|^p dx + \epsilon \Vert D \chi \Vert_{TV(\R^d)} \\
        & \qquad  + \epsilon \Vert U \Vert_{L^\infty}^{q-1} \Vert D \chi \Vert_{TV(\R^d)} + \epsilon \Vert U \Vert_{L^\infty}^{q-1} \Vert DU - D \chi \Vert_{TV(\R^d)} \Big).
        \end{split}
    \end{align}

\emph{Step 2: Adjusting the boundary conditions.}

    Now we turn to the construction preserving the exterior data, i.e. we aim to construct $U_\epsilon \in L^p(\R^d;\R^n) \cap W^{1,q}(\R^d;\R^n)$ such that $U_\epsilon = 0$ outside $\Omega$.
    Let $U \in L^p(\R^d;\R^n) \cap BV(\R^d;\R^n)$ such that $U = 0$ outside $\Omega$ and $\chi \in L^\infty(\R^d;\R^n) \cap BV(\R^d;\R^n)$.

    By a translation we consider $\Omega' = (-1/2,1/2)^d$, $V := U(\cdot + 1/2(1,\dots,1)):\Omega'\to\R^n$, and $\chi' := \chi(\cdot + 1/2(1,\dots,1)):\Omega' \to \R^n$ in the following.
    We fix
    \begin{align*}
        \Omega'_\epsilon := (1-2\epsilon) \Omega' = (-\frac{1}{2}+\epsilon,\frac{1}{2}-\epsilon)^d \subset \Omega'.
    \end{align*}
    In particular, we have
    \begin{align*}
        \dist(\p\Omega',\p\Omega'_\epsilon) \geq \epsilon.
    \end{align*}

    We now choose $V_\epsilon = V^{(\epsilon)} \ast \psi_\epsilon$, where
    \begin{align*}
        V^{(\epsilon)}(y) = V(\frac{1}{1-2\epsilon}y), \ y \in \R^d.
    \end{align*}
    As $V = 0$ outside $\Omega'$, 
    we can infer $V_\epsilon = 0$ outside $\Omega'$.
    We now compare the energy for $V_\epsilon$ to that of $V \ast \psi_\epsilon$, which we can control with the arguments from step 1, as follows

    \begin{align*}
        \int_{\Omega'} |V_\epsilon - \chi'|^p + \epsilon^q |\nabla V_{\epsilon}|^q dy & \leq C(p,q) \Big( \int_{\Omega'} |V \ast \psi_{\epsilon} - \chi'|^p + \epsilon^q |\nabla (V \ast \psi_\epsilon)|^q dy \\
        & \quad + \Vert V_\epsilon - V \ast \psi_\epsilon \Vert_{L^p}^p + \epsilon^q \Vert \nabla V_\epsilon - \nabla (V \ast \psi_\epsilon) \Vert_{L^q}^q \Big).
    \end{align*}

    Hence, it remains to control
    \begin{align*}
        \Vert (V^{(\epsilon)} - V) \ast \psi_\epsilon \Vert_{L^p}^p + \epsilon^q \Vert \nabla (V^{(\epsilon)}-V) \ast \psi_\epsilon \Vert_{L^q}^q &\leq \Vert V^{(\epsilon)} - V \Vert_{L^p}^p + C \Vert V^{(\epsilon)} - V\Vert_{L^q}^q \\
        & \leq C \max\{\Vert U \Vert_{L^\infty}^{p-1}, \Vert U \Vert_{L^\infty}^{q-1}\} \Vert V^{(\epsilon)} - V \Vert_{L^1},
    \end{align*}
    where we invoked Young's convolution inequality and \cref{eq:DiffToSharp_Young}.
    Thus, using that $V, V^{(\epsilon)} = 0$ outside $\Omega'$ it holds
    \begin{align*}
        \int_{\R^d} |V^{(\epsilon)}(x) - V(x) | dx & \leq C \int_1^{\frac{1}{1-2\epsilon}} t^{-d} dt \, \Vert D V \Vert_{TV(\R^d)} \leq C (1- (1-2\epsilon)^{d}) \Vert D V \Vert_{TV(\R^d)} \\
        & \leq C(d) \epsilon \Vert D U \Vert_{TV(\R^d)},
    \end{align*}
    and, therefore, for $U_\epsilon(x) = V_\epsilon(x-1/2(1, \dots,1))$
    \begin{align*}
        \int_\Omega |U_\epsilon - \chi|^p + \epsilon^q |\nabla U_\epsilon|^q dx & \leq C(\Vert \chi \Vert_{L^\infty},p,q,d) \Big( \int_\Omega |U-\chi|^p dx \\
        & \quad + \epsilon (1+\max\{\Vert U \Vert_{L^\infty}^{p-1},\Vert U \Vert_{L^\infty}^{q-1}\}) \Vert D\chi \Vert_{TV(\Omega)}\\
        & \quad + \epsilon \max\{\Vert U \Vert_{L^\infty}^{p-1}, \Vert U \Vert_{L^\infty}^{q-1} \} (\Vert D U - D\chi \Vert_{TV(\R^d)} + \Per(\Omega))\Big).
    \end{align*}
    As $U_\epsilon$ is a combination of a rescaling and convolution of $U$, the differential constraint is preserved.
\end{proof}

For completeness, we also derive an estimate of the above type for the anisotropic setting.
\begin{cor} \label{cor:DiffToSharp_Aniso_Upper}
    Let $p,q \in [1,\infty)$, $\nu \in \S^{d-1}$, $\Omega = (0,1)^d$. For any $U, \chi \in L^p(\R^d;\R^n) \cap BV_\nu(\R^d;\R^n)$, with $U = 0$ outside $\Omega$ and $\epsilon > 0$ there exists $U_\epsilon \in L^p(\R^d;\R^n)$ such that $\p_\nu U_\epsilon \in L^q(\R^d;\R^n)$, $U_\epsilon = 0$ outside $\Omega$, and a constant $C = C(\Vert \chi \Vert_{L^\infty},p,q,d) > 0$ such that
    \begin{align*}
        \int_{\Omega} |U_\epsilon - \chi|^p + \epsilon^q |\p_\nu U_\epsilon|^q dx
        & \leq C \biggl( \int_{\Omega} |U-\chi|^p dx + \epsilon (1+\max\{\Vert U \Vert_{L^\infty}^{p-1},\Vert U \Vert_{L^\infty}^{q-1}\})\Vert D_\nu \chi \Vert_{TV(\Omega)} \\
        & \quad + \epsilon \max\{1,\Vert U \Vert_{L^\infty}^{p-1},\Vert U \Vert_{L^\infty}^{q-1}\}( \Vert D_\nu U - D_\nu \chi \Vert_{TV(\Omega)} +\Per(\Omega) \biggr)
    \end{align*}
    Moreover, if we have $\operatorname{curl} U= 0$ or $\operatorname{curl} \operatorname{curl} U = 0$ in $\R^d$ it also holds $\operatorname{curl} U_{\epsilon} = 0$ or $\operatorname{curl} \operatorname{curl} U_\epsilon = 0$, respectively.
\end{cor}

\begin{proof}
    For given $U \in L^p(\R^d;\R^{n})$, we define $U_\epsilon(x) = \int_{\R} U(x-t\nu) \psi_\epsilon(t)dt$ with a one-dimensional mollifier $\psi_\epsilon$. That is, by $U_\epsilon$ we denote the function which is obtained by mollifying $U$ in direction $\nu$ on a scale $\epsilon$; we omit the $\nu$ dependence in the notation.
    Due to the fact that $(U_\epsilon)^\nu_y(t) = (U^\nu_y \ast \psi_\epsilon)(t)$, using the notation introduced in \cref{sec:prelim_directBV}, all the required estimates \cref{eq:DiffToSharp_Young,eq:DiffToSharp_Gradient,eq:DiffToSharp_ConvRate} have their respective analogous inequalities:
    \begin{align*}
        \Vert \p_\nu U_\epsilon \Vert_{L^r} &\leq \epsilon^{-1} \Vert \psi' \Vert_{L^1(\R)} \Vert U \Vert_{L^r}, \\
        \Vert \p_\nu U_\epsilon \Vert_{L^1} & \leq \Vert D_\nu U \Vert_{TV(\R^d)}, \\
        \Vert U_\epsilon - U \Vert_{L^1} & \leq C \epsilon \Vert D_\nu U \Vert_{TV(\R^d)}.
    \end{align*}
    Hence, the proof works as that of \cref{lem:DiffToSharp_Mollified}.  In particular, the changes required to not change the exterior data outside $\Omega$ are done by the same methods.

    As the function $U_\epsilon$ is essentially defined via a convolution and scaling, the differential constraint of the form $\operatorname{curl} U_\epsilon = 0$ or $\operatorname{curl curl} U_\epsilon = 0$ is still fulfilled in $\R^d$.
\end{proof}

\begin{rmk}
    We want to conclude this general analysis with two remarks.
    First, the same results as in \cref{sec:diffuse1,sec:diffuse2} hold if we consider periodic functions instead of prescribing the exterior data outside of $\Omega$.
    Second, we could also consider other differential constraints besides $\operatorname{curl} U = 0$ or $\operatorname{curl} \operatorname{curl} U = 0$ in $\R^d$.
    In general, in \cref{lem:DiffToSharp_Mollified,cor:DiffToSharp_Aniso_Upper} we can preserve the constraint $L(D) U = 0$ in $\R^d$ for any linear, homogeneous, constant coefficient differential operator $L(D)$, e.g. we could also consider $\di U = 0$ in $\R^d$.
    The role of the divergence operator for lower scaling bounds was discussed in \cite{RRT23}.
\end{rmk}

\subsection{Applications -- proof of \cref{cor:Lp_3wells}}
We now use the derived comparison results of the diffuse and sharp interface models to show \cref{cor:Lp_3wells}.
\begin{proof}[Proof of \cref{cor:Lp_3wells}]
Let us begin by recalling the upper bound constructions from \cref{sec:L1AnisotropicSurf} for \cref{thm:L1_3wells} for $U = \nabla u$. We recall that, as highlighted in \cref{eq:PropertiesBranching} and in the proofs of \cref{prop:L1_2wells,thm:L1_3wells}, we always have
\begin{align*}
    \Vert DU - D \chi \Vert_{TV(\Omega)} &\leq C(\K,F) (\Vert D \chi \Vert_{TV(\Omega)} + \Per(\Omega)), \\
    \Vert U \Vert_{L^\infty} &\leq C(\K,F).
\end{align*}
Thus, considering the functions $u_\epsilon,\chi_\epsilon$ defined in the respective upper bound, we use \cref{thm:comparison,cor:DiffToSharp_Aniso_Upper} (with $U = \nabla u_\epsilon$, $\chi = \chi_\epsilon$) to derive the scaling bounds.
Indeed, we obtain that
\begin{align*}
    & \quad c \inf_{\tilde{\chi}  \in BV_\nu(\Omega;\K_3)} \inf_{u \in \mathcal{A}_F}\left( \int_\Omega |\nabla u - \tilde{\chi}|^2 dx + \epsilon \Vert D_{\nu} \tilde\chi \Vert_{TV(\Omega)} \right) \\
    &\leq \inf_{\chi \in L^2(\Omega;\K_3)} \inf_{\substack{u \in \mathcal{A}_F \\ \p_\nu \nabla u \in L^{q}(\Omega;\R^{2 \times 2})}} \left( \int_\Omega |\nabla u - \chi|^2 + \epsilon^q |\p_\nu \nabla u|^q dx \right) \\
    & \leq C \left( \int_{\Omega} |\nabla u_\epsilon - \chi_\epsilon|^2 dx + \epsilon \Vert D_\nu \chi_{\epsilon} \Vert_{TV(\Omega)} + \epsilon \Per(\Omega) + \epsilon \Vert D_\nu\nabla u_{\epsilon} - D_\nu \chi_{\epsilon} \Vert_{TV(\Omega)} \right) \\
    & \leq C \left( \int_{\Omega} |\nabla u_{\epsilon} - \chi_{\epsilon}|^2 dx + \epsilon \Vert D_\nu \chi_{\epsilon} \Vert_{TV(\Omega)} + \epsilon \Per(\Omega) \right),
\end{align*}
where $C$ may depend on $\mathcal{K}$ and $F$.
In particular, after an absorption of the perimeter term, both the upper and lower bounds from \cref{cor:Lp_3wells} are proven by inserting the upper and lower bounds for the associated sharp interface models.
\end{proof}

\section{Discrete models and anisotropic surface energies}
\label{sec:discrete}
As a final prototypical example, we complement our discussion on possible regularizations by now considering the case of discrete models.
Indeed, discrete systems prevent oscillations at scales smaller than the grid size of the lattice.
As already known in the literature (see e.g.\ \cite{CM99,L09}), this produces a regularization effect comparable to the addition of a singular surface energy term in the continuous case. In our discussion, we recover these results in our model settings. Our main focus and novelty here, however, is on the analysis of anisotropic situations.
In this context, the orientation of the lattice naturally introduces an anisotropy which may affect the scaling depending on the geometry of $\K$.
This effect can be tracked into its continuous counterpart. In particular, this will allow us to invoke the scaling results from the previous sections.

\subsection{Quantitative surface penalization in discrete models}
In this section we consider discrete energies in the sense that, for $h \in (0,1)$ and $R \in SO(2)$, we fix a triangulation $\mathcal{T}_h^R$ of $\Omega$ on the scale $h$ and assume that $\nabla u$ and $\chi$ are constant on the triangles $\tau \in \mathcal{T}_h^R$.
By doing this we hence rule out -- and, in particular, ``penalize" -- oscillations on a scale finer than $h$.

For this reason, in certain instances, this discrete energy can be bounded from below by the continuous elastic energy contribution singularly perturbed by a sharp (anisotropic) surface penalization of the form given in \cref{eq:L1_energy}.

To observe this, we will make the energy contribution of three adjacent triangles explicit, and note that when a change of phase occurs, the ``middle" triangle $T_h'$ has to pay elastic energy, giving rise to a contribution which resembles the surface energy $\Vert D_\nu \chi \Vert_{TV(\Omega)}$.
We recall the notation for the lattice structures from \cref{eq:lattice}, \cref{eq:Triangulation}, the admissible deformations from \cref{eq:admissible_discrete} and the discrete energy from \cref{eq:energy_discrete}. We begin by deducing ``interior'' estimates. In the subsequent result, we will also incorporate associated boundary conditions.

In what follows, we focus on the anisotropic setting, more precisely, on the case in which we assume that there is a compatible direction of the wells which is perpendicular to one of the sides of the triangles. We allow for rank-one connections which are in exactly one of the directions $R e_1$, $R e_2$, or $R(e_1+e_2)$ for a fixed matrix $R\in SO(2)$. We note that this cannot yield a full surface penalization for the associated continuum model but will give rise to an anisotropic surface penalization.

\begin{lem}
\label{lem:DiscreteToSharp-aniso}
    Let $\K = \{A_1, A_2, \dots, A_N\} \subset \R^{2 \times 2}$. Let $T_h, T_h'$ denote the triangles from \cref{eq:lattice}. Assume that there exists $v \in \{e_1, e_2, 2^{-1/2}(e_1+e_2)\}$  such that for any $A_j, A_k \in \K $ with $j \neq k$ it holds
    \begin{align}
    \label{eq:assumption-wells-deg-lem}
        (Rw) \times [A_j - A_k] \neq 0, \quad \text{for every } w\in\{e_1,e_2,2^{-1/2}(e_1+e_2)\}\setminus\{v\}.
    \end{align}
    Let $p \in [1,\infty)$, $z \in h \Z^2$ and consider
    \begin{align}\label{eq:DefLargeTriangles}
        \Omega_j := 
        (RT_h \cup RT_h' \cup (RT_h+hRe_j)) + Rz,
    \end{align}
    for $j = 1,2$, illustrated in \cref{fig:UnionTriangles}.
    Then there is a constant $C = C(\K,R,p,v) > 0$ such that for any $u \in W^{1,p}(\Omega_1 \cup  \Omega_2;\R^2)$ and any $\chi \in L^\infty(\Omega_1 \cup \Omega_2;\K)$ which are affine or, respectively, constant on the triangles $R(T_h+z),R(T_h'+z),R(T_h+z+he_1),R(T_h+z+he_2)$,  it holds
    \begin{align*}
        \int_{\Omega_1 \cup \Omega_2} |\nabla u - \chi|^p dx \geq C h \Vert D_{R\nu} \chi \Vert_{TV(\Omega_1 \cup \Omega_2)},
    \end{align*}
    where $\nu \in \S^1$ is such that $\nu \cdot v = 0$.

    Moreover, defining
    \begin{align*}
    \Omega_j' := (RT_h \cup RT_h' \cup (RT_h'-hRe_j)) + Rz,
    \end{align*}
    for $j=1,2$, the same result holds with $\Omega_j'$ in place of $\Omega_j$.
\end{lem}

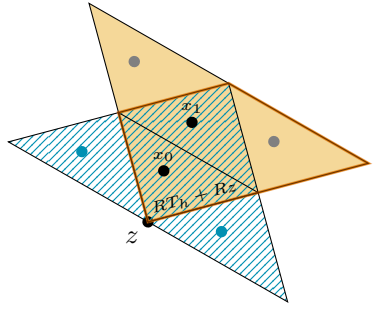
\begin{figure}
    \centering
    \begin{tikzpicture}
    \begin{scope}[rotate=15, scale=1.5]
        \fill[color = Orange, opacity = 0.4] (0,0) -- (2,0) -- (0,2) -- cycle;
        \fill (0,0) circle (0.05) node[below left] {$z$};
        \fill[pattern = north east lines, pattern color = Blue] (-1,1) -- (1,1) -- (1,-1) -- cycle;

        \draw[thick, Orange] (0,0) -- (0,1) -- (1,1) -- (2,0) -- cycle;

        \draw (0,0) -- (2,0) -- (0,2) -- cycle;
        \draw (1,0) -- (1,1) -- (0,1) -- cycle;

        \draw (0,1) -- (-1,1) -- (0,0) -- cycle;
        \draw (1,0) -- (1,-1) -- (0,0) -- cycle;

        \draw[very thick, Red, opacity = 0.5] (0,0) -- (0,1) -- (1,1) -- (2,0) -- cycle;

        \node at (0.45,0.1) [, rotate = 15] {\tiny$RT_h + Rz$};
        \fill (0.25,0.4) circle (0.05) node[above] {\tiny$x_0$};
        \fill (0.6,0.75) circle (0.05) node[above] {\tiny$x_1$};
        \fill[gray] (1.25,0.4) circle (0.05);
        \fill[gray] (0.25,1.4) circle (0.05);
        \fill[Blue] (-0.4,0.75) circle (0.05);
        \fill[Blue] (0.6,-0.25) circle (0.05);
    \end{scope}
    \end{tikzpicture}
    \caption{Illustration of the sets in \cref{lem:DiscreteToSharp-aniso} for a rotation $R$ by $15^\circ$. The set $\Omega_1 \cup \Omega_2$ is filled in orange. The fixed points $x_0$, $x_1$ are marked in black and their translations $x_0 + h e_j$ are shown in gray. The set $\Omega_1$ is highlighted with a red border.
    The corresponding setting for $\Omega_1' \cup \Omega_2'$ is shown in blue.}
    \label{fig:UnionTriangles}
\end{figure}

\begin{proof}
    Without loss of generality, we can assume $z = 0$.
    We also start by considering $R=Id$; the general case will be reduced to this one below.
    We fix $x_0 \in T_h$ and $x_1 \in T_h'$ (see \cref{fig:triangulation}).
    For simplicity of exposition 
    {we prove the estimate for $\Omega_1 \cup \Omega_2$ only,} with the estimate on $\Omega_1' \cup \Omega_2'$ being completely analogous.

\emph{Step 1: The case of $R= \operatorname{Id}$.}
    We first quantify the energy contribution of a possible change of phase in $\Omega_j$.
    Its main contribution concentrates on $T_h'$ and amounts to
    \begin{align}\label{eq:trinagle1}
    \begin{split}
        \int_{T_h'} |\nabla u - \chi|^p dx &\geq c h^2 \big| e_j \times (\chi(x_0+he_j) - \chi(x_1)) \big| - 2^{p} \int_{T_h \cup (T_h + h e_j)} |\nabla u - \chi|^p dx,
    \end{split}
    \end{align}
    for some $c=c(\K,p)>0$.
    To prove \cref{eq:trinagle1}, we first work on $\Omega_1$, i.e. with $j=1$. The estimate for $\Omega_2$ follows by switching the roles of the coordinates.
    Since $u$ is affine on each triangle in $\Omega_1\cup \Omega_2$, by tangential continuity along the edges of $T_{h}'$ it holds that
    \begin{align*}
        \p_1 u (x_1)= \p_1 u (x_0+h e_2), \ \p_2 u (x_1)= \p_2 u (x_0+h e_1), \ \p_1 u (x_1)- \p_2 u (x_1) = \p_1 u (x_0)- \p_2 u (x_0),
    \end{align*}
    and, hence,
    \begin{align}\label{eq:triangle2}
    \begin{split}
        \p_1 u (x_1) & =  \p_2 u(x_0+he_1) + \p_1 u(x_0) - \p_2 u(x_0) = \p_1 u(x_0+h e_2), \\
        \p_2 u (x_1) &= \p_2 u(x_0+he_1) = \p_1 u(x_0+he_2) + \p_2 u(x_0) - \p_1u(x_0).
    \end{split}
    \end{align}
    By exploiting the fact that $A^\frac{p}{2}+B^\frac{p}{2}\le\sqrt{2}(A+B)^\frac{p}{2}$ for $A,B\ge0$, relation \cref{eq:triangle2},
    by Jensen's inequality (in the form $|A-B|^p-2^{p-1}|B|^p\leq 2^{p-1}|A|^p$), and
    writing $\chi_j = \chi e_j$ for the $j$th column of $\chi$, we obtain
    \begin{align*}
        &\int_{T_h'} |\nabla u - \chi|^p dx \\
        & \quad = \frac{h^2}{2} \left( | \p_2 u(x_0+he_1) + \p_1u(x_0) - \p_2 u(x_0) - \chi_1(x_1)|^2 + |\p_2 u(x_0+he_1)-\chi_2(x_1)|^2 \right)^{\frac{p}{2}} \\
        & \quad \geq \frac{h^2}{2} \frac{1}{2^p\sqrt{2}} \left( |\p_1 u(x_0) - \p_2 u(x_0) - (\chi_1(x_1) - \chi_2(x_1)) |^p + |\p_2 u(x_0+he_1) - \chi_2(x_1)|^p \right)\\
        & \quad \geq \frac{h^2}{2} \frac{1}{2^p\sqrt{2}} \Big( 2^{1-p} |\chi_1(x_0) - \chi_1(x_1) - (\chi_2(x_0) - \chi_2(x_1))|^p +2^{1-p} |\chi_2(x_0+he_1) - \chi_2(x_1)|^p \\
        & \qquad - |\p_1 u(x_0) - \p_2 u(x_0) - (\chi_1(x_0) - \chi_2(x_0))|^p  - |\p_2 u(x_0+h e_1) - \chi_2(x_0+he_1)|^p \Big).
    \end{align*}
    Let $C_0 = \min\big\{|w \times(A_j-A_k)| : j\neq k, w\in\{e_1,e_2,2^{-1/2}\begin{psmallmatrix}1\\1\end{psmallmatrix}\} \setminus \{v\}\big\} > 0$ which is strictly positive thanks to \cref{eq:assumption-wells-deg-lem}.
    For simplicity of exposition, in the next two estimates we assume $v=2^{-1/2}(e_1+e_2)$. Hence, we deduce
    \begin{align*}
       & \int_{T_h'} |\nabla u - \chi|^p dx \\
       &\quad\geq 2^{-2p-1} h^2 \Big( | \begin{psmallmatrix} 1 \\ 1 \end{psmallmatrix} \times (\chi(x_0) - \chi(x_1))|^p + |e_1 \times (\chi(x_0+he_1) - \chi(x_1))|^p \Big) \\
        & \qquad - \frac{h^2}{2} \Big( |\p_1 u(x_0) - \p_2 u(x_0) - (\chi_1(x_0) - \chi_2(x_0))|^p  + |\p_2 u(x_0+h e_1) - \chi_2(x_0+he_1)|^p \Big) \\
        & \quad\geq 2^{-2p-1} h^2 \Big( | \begin{psmallmatrix} 1 \\ 1 \end{psmallmatrix} \times (\chi(x_0) - \chi(x_1))|^p + |e_1 \times (\chi(x_0+he_1) - \chi(x_1))|^p \Big) \\
        & \qquad - \frac{h^2}{2} \Big( 2^p|\nabla u(x_0) - \chi(x_0)|^p  + |\nabla u(x_0+h e_1) - \chi(x_0+he_1)|^p \Big) \\
        & \quad\geq 2^{-2p-1} C_0^{p-1} h^2  |e_1 \times (\chi(x_0+he_1)-\chi(x_1))| - 2^{p} \int_{T_h \cup T_h+he_1} |\nabla u - \chi|^p dx.
    \end{align*}
    Here, we used $\begin{psmallmatrix} 1 \\ 1 \end{psmallmatrix} \times \nabla u= \begin{psmallmatrix} -1 \\ 1 \end{psmallmatrix} \cdot \nabla u= \p_2 u - \p_1 u$ to denote the tangential components for the edges with normal $\begin{psmallmatrix} 1 \\ 1 \end{psmallmatrix}$, see the notation introduced in \cref{sec:intro_Tartar}.
    This proves \cref{eq:trinagle1} with $c =2^{-2p-1} C_0^{p-1} > 0$ depending only on $\K$, $v$, and $p$.

    We now combine \cref{eq:trinagle1} with the analogous inequality with $e_1$ and $e_2$ switched.
    We thus have
    \begin{align*}
    & \quad \int_{\Omega_1 \cup \Omega_2} |\nabla u - \chi|^p dx \\
    &  \geq \int_{T_h \cup ( T_h + h e_1) \cup (T_h + h e_2) } |\nabla u - \chi|^p dx + 2^{-p-1}\int_{T_h'} |\nabla u - \chi|^p dx + 2^{-p-1}\int_{T_h'} |\nabla u - \chi|^p dx\\
    &  \geq c h^2 2^{-p-1} \sum_{j=1}^2 \Big|e_j \times [\chi(x_0+he_j)-\chi(x_1)]\Big| \\
    & \quad + (1-2^{-1}-2^{-1}) \int_{T_h} |\nabla u - \chi|^{p} dx + (1-2^{-1}) \int_{(T_h + h e_1) \cup (T_h + h e_2)} |\nabla u - \chi|^{p} dx \\
    & \geq c 2^{-p-1} h^2 \sum_{j=1}^2 \Big| e_j \times [\chi(x_0+he_j) - \chi(x_1)] \Big|.
    \end{align*}

    For general $v \in \{e_1, e_2, 2^{-1/2}(e_1+e_2)\}$ we follow the same arguments as above and get
    \begin{align}\label{eq:lb-bigt}
    \begin{split}
        \int_{\Omega_1 \cup \Omega_2} |\nabla u - \chi|^p dx &\geq c 2^{-p-1} h^2 \Big( \sum_{j=1}^2 \Big| e_j \times [\chi(x_0+he_{j})- \chi(x_1)] \Big| + \Big| \begin{pmatrix}
           1/\sqrt{2} \\ 1/\sqrt{2}
        \end{pmatrix} \times [\chi(x_0) - \chi(x_1)] \Big| \\
        & \quad - \Big| v \times [\chi(x_v)-\chi(x_1)] \Big| \Big),
        \end{split}
    \end{align}
    where $x_v=x_0$ if $v=2^{-1/2}(e_1+e_2)$, $x_v=x_0+he_j$ if $v=e_j$ 
    , such that exactly the term with $v$ cancels.

    To relate this expression to the desired surface energy, we note that for $\nu \in \S^1$ such that $\nu \cdot v = 0$, it follows that
    \begin{align*}
        \Vert D_\nu \chi \Vert_{TV(\Omega_1 \cup \Omega_2)} & = h | \nu \cdot 2^{-1/2}(e_1 + e_2)| \ |\chi(x_0) - \chi(x_1)| + \sum_{j=1}^2 h |\nu \cdot e_j| \ |\chi(x_1)-\chi(x_0+he_j)| \\
        & \leq C h \Big( |\chi(x_0)-\chi(x_1)| + \sum_{j=1}^2 |\chi(x_1)-\chi(x_0+he_j)| - |\chi(x_1) - \chi(x_v)|\Big).
    \end{align*}
    From \cref{eq:lb-bigt}, exploiting the boundedness of $\chi$, we can also infer that
    \begin{align*}
        \int_{\Omega_1 \cup \Omega_2} |\nabla u - \chi|^p dx & \geq c 2^{-p-1} \frac{h^2}{\diam(\K)} \Big( \Big| \begin{pmatrix} 1/\sqrt{2} \\ 1/\sqrt{2} \end{pmatrix} \times [\chi(x_0) - \chi(x_1)] \Big| |\chi(x_0)- \chi(x_1)|  \\
        & \quad  + \sum_{j=1}^2 \Big| e_j \times [\chi(x_0+he_j) - \chi(x_1)] \Big| |\chi(x_0+he_j) -\chi(x_1)| \\
        & \quad - \Big| v \times [\chi(x_1) - \chi(x_v)] \Big| |\chi(x_1) - \chi(x_v)| \Big).
    \end{align*}
    Hence, gathering the two inequalities above and again by \cref{eq:assumption-wells-deg-lem}
    for a constant $C = C(\K,p,v) > 0$ we get
    \begin{align*}
        \int_{\Omega_1 \cup \Omega_2} |\nabla u - \chi|^p dx &\geq C h \left( h|\chi(x_0)-\chi(x_1)| + \sum_{j=1}^2 h |\chi(x_0+he_j)-\chi(x_1)| - h |\chi(x_1)-\chi(x_v)|\right) \\
        & \geq C h \Vert D_\nu \chi \Vert_{TV(\Omega_1 \cup \Omega_2)}.
    \end{align*}

Step 2: General rotations $R \in SO(2)$.
    Now we turn to the general (rotated) case.
    For given $R \in SO(2)$, consider $\tilde{\Omega}_j = R^T \Omega_j = (T_h \cup T_h' \cup (T_h + h e_j)) + z$ and the functions
    \begin{align*}
        \tilde{u}&: \tilde{\Omega}_1 \cup \tilde{\Omega}_2 \to \R^2, & \tilde{u}(x) &= u(Rx), \\
        \tilde{\chi}&: \tilde{\Omega}_1 \cup \tilde{\Omega}_2 \to \K R, & \tilde{\chi}(x) &= \chi(Rx) R.
    \end{align*}
    For those functions it holds that $\nabla \tilde{u}(x) = \nabla u (Rx) R$, and thus
    \begin{align*}
        \int_{\Omega_1 \cup \Omega_2} |\nabla u - \chi|^p dx &= \int_{R^T(\Omega_1 \cup \Omega_2)} |\nabla u (Rx) - \chi(Rx)|^p dx = \int_{\tilde{\Omega}_1 \cup \tilde{\Omega}_2} |\nabla \tilde{u} - \tilde{\chi}|^p dx, \\
        \Vert D_{R \nu} \chi \Vert_{TV(\Omega_1 \cup \Omega_2)} & = \Vert D_\nu \tilde{\chi} \Vert_{TV(\tilde{\Omega}_1 \cup \tilde{\Omega}_2)}.
    \end{align*}
    For the wells $A_j,A_k \in \K$, $j \neq k$, from \cref{eq:assumption-wells-deg-lem} it holds
    \begin{align*}
        e_j \times (A_j R - A_k R) = (Re_j) \times (A_j-A_k), \ \begin{pmatrix} 1 \\ 1 \end{pmatrix} \times (A_j R - A_k R) = (R \begin{pmatrix} 1 \\ 1 \end{pmatrix}) \times (A_j - A_k).
    \end{align*}
    Therefore the set $\tilde{\K} = \K R$ fulfils the compatibility conditions \cref{eq:assumption-wells-deg-lem} for $R = \operatorname{Id}$.
    Hence, the claim follows by applying the argument for $R = \operatorname{Id}$ to the functions $\tilde{u}$, $\tilde{\chi}$, and thus
    \begin{align*}
        \int_{\Omega_1 \cup \Omega_2} |\nabla u - \chi|^p dx &= \int_{\tilde{\Omega}_1 \cup \tilde{\Omega}_2} |\nabla \tilde{u} - \tilde{\chi}|^p dx \geq C(\tilde{\K},p,v) h \Vert D_\nu \tilde{\chi} \Vert_{TV(\tilde{\Omega}_1 \cup \tilde{\Omega}_2)} \\
        & = C(\K,R,p,v) h \Vert D_{R \nu} \chi \Vert_{TV(\Omega_1 \cup \Omega_2)}.
    \end{align*}
\end{proof}

With the above auxiliary result on ``interior estimates'' in hand, we now cover $\Omega$ with copies of $\Omega_1 \cup \Omega_2$ to relate the discretized setting to a sharp interface model. This, together with an estimate of the energy contributions which arise at the boundary allows us to provide the proof of \cref{cor:DiscreteToSharp}.

\begin{proof}[Proof of \cref{cor:DiscreteToSharp} in the anisotropic setting]
    Only the proof for the anisotropic setting is given, for the full isotropic surface penalization the arguments are analogous with the corresponding changes for the full derivative.
    By possibly replacing $\Omega$ by $R^T \Omega$, in the following we assume $R = \operatorname{Id}$.
    The general result follows by a rotation as done in the proof of \cref{lem:DiscreteToSharp-aniso}.
    Without loss of generality and for brevity of exposition, we assume that $v=2^{-\frac{1}{2}}(1,1)^T$ and thus $\nu = 2^{-\frac{1}{2}}(-1,1)^T$.
    We work in multiple steps, separating the analysis in the interior of the domain and close to the boundary.

    \emph{Step 1: Estimate on the interior.}
    We cover $\Omega$ by sets of the form $\Omega_1^z \cup \Omega_2^z$ as in \cref{eq:DefLargeTriangles}, i.e.
    \begin{align*}
        \Omega_1^z \cup \Omega_2^z = (T_h \cup T_h' \cup (T_h+h  e_1) \cup (T_h + e_2)) + z,
    \end{align*}
    where we now keep track of the dependence on $z \in h \Z^2$ in the notation.
    We define the discrete set $\mathcal{J}:=\{z\in h \Z^2 : (\Omega_{1}^z \cup
    \Omega_2^z) \subset \Omega\}$ and the set $\Omega_h := \cup_{z \in
      \mathcal{J}} \Omega_1^z \cup \Omega_2^z$.
    The interior set $\Omega_h\subset \Omega$ approximates $\Omega$ from the inside as $h\downarrow 0$ and may leave an $h$-neighbourhood of the boundary $\p \Omega$ not yet covered, which will be addressed in Step 2.

    As $\Omega_1^z \cup \Omega_2^z$ overlap at most six times for $z\in\mathcal{J}$, by an application of \cref{lem:DiscreteToSharp-aniso} we obtain the lower bound
    \begin{align} \label{eq:DiscreteToSharp_inner}
    \begin{split}
        \int_{\Omega} |\nabla u - \chi|^p dx & \geq \frac{1}{6} \sum_{z \in\mathcal{J}} \int_{(\Omega_1^z \cup \Omega_2^z) \cap \Omega} |\nabla u - \chi|^p dx \\
        & \geq \sum_{z \in \mathcal{J}} C' h \Vert D_{\nu} \chi \Vert_{TV(\Omega_1^z \cup \Omega_2^z)} \geq C' h \Vert D_{ \nu} \chi \Vert_{TV(\Omega_h)},
        \end{split}
    \end{align}
    with $C'>0$ depending on $\K$, $v$, and $p$.

    \emph{Step 2: Boundary layer.}
    To estimate the energy contribution of the boundary layer, we introduce the set of boundary triangles
    \begin{align*}
        \mathcal{I} := \{ \tau \in \mathcal{T}_h:
        |\tau \cap \Omega| > 0,  \
        \mathcal{H}^1(\overline{\tau} \cap \p \Omega) > 0\}.
    \end{align*}

    For a constant $\sigma=\sigma(\Omega) \in (0,\frac{1}{2})$ which will be fixed below, we decompose the set of boundary triangles into $\mathcal{I} = \mathcal{I}_{\textup{small}} \cup \mathcal{I}_1 \cup \mathcal{I}_2$, where
    \begin{align*}
        \mathcal{I}_{\textup{small}} &:= \{\tau \in \mathcal{I}: |\tau \cap \Omega| < \frac{1}{2}\sigma h^2\},\\
        \mathcal{I}_1 &:= \{ \tau \in \mathcal{I}\setminus \mathcal{I}_{\textup{small}}:  |\tau \setminus \Omega| > 0\}, \\
         \mathcal{I}_2  &:= \{ \tau \in \mathcal{I} \setminus \mathcal{I}_{\textup{small}}: |\tau \setminus \Omega| = 0\}.
    \end{align*}
    The triangles in $\mathcal{I}_1$ and $\mathcal{I}_2$ have non-degenerate area inside $\Omega$
    and, in particular, those in $\mathcal{I}_2$ have at least one edge that has a common line segment with $\p \Omega$.

    We will show that the energy on this boundary layer produces a contribution of order (at least) $h$.
    To prove this, the case in which the edges of $\Omega$ are aligned with the grid necessitates a careful treatment, since (a priori) in this case $u\in\mathcal{A}^{p,\Id}_{h,F}$ could laminate with the boundary condition without paying any energy.
    We thus split this analysis into two substeps, depending on the geometry of the domain.

    \emph{Step 2.1: Domains aligned with the grid.}
    Consider first the case in which $\Omega$ is a polygon with edges parallel to the directions of the triangulation, i.e.\ such that $n_{\p\Omega}\in\{\pm e_1,\pm e_2, \pm v\}$ $\mathcal{H}^1$-a.e.\ on $\p\Omega$, where $n_{\p\Omega}$ is the outer unit normal vector to $\p\Omega$.
    Notice that there exist $h_0=h_0(\Omega)>0$ and $C_0>0$ a universal constant such that, for every edge $\Gamma$ of $\Omega$ and for every $h<h_0$ it holds
    \begin{equation}\label{eq:C0}
    \#\{\tau\in\mathcal{I} : \bar{\tau}\cap\Gamma\neq\emptyset\} \ge \frac{C_0}{h}\mathcal{H}^1(\Gamma).
    \end{equation}
    We now split our discussion in two cases.

    \begin{figure}
        \centering
        \begin{subfigure}[t]{0.4\textwidth}
        \centering
            \begin{tikzpicture}
                \draw[very thick, Blue] (-0.4,2.2) -- (-0.4,-2.2) node[below] {$\Gamma$};
                \draw[|-|] (-0.4,2.3) -- (0,2.3) node[above,midway] {$\sigma h$};
                \fill[color = Blue, fill opacity = 0.3] (-0.4,-2.2) rectangle (1.5,2.2);

                \foreach \y in {-2,-1,...,2}{
                    \draw (-1,\y) -- (1.5,\y);
                }

                \foreach \x in {-1,0,1}{
                    \draw (\x,-2.2) -- (\x,2.2);
                }

                \draw (0.8,2.2) -- (1.5,1.5);
                \draw (-0.2,2.2) -- (1.5,0.5);
                \draw (-1,2) -- (1.5,-0.5);
                \draw (-1,1) -- (1.5,-1.5);
                \draw (-1,0) -- (1.2,-2.2);
                \draw (-1,-1) -- (0.2,-2.2);
                \draw (-1,-2) -- (-0.8,-2.2);

                \fill[pattern = north east lines, pattern color = Purple, fill opacity = 0.5] (-1,-2.2) rectangle (0,2.2);
            \end{tikzpicture}
            \caption{Case 1: The distance of $\Gamma$ to the (interior) edges of the triangulation is at least $\sigma h$. The boundary triangles (hashed in purple) have uniform area inside $\Omega$.}
            \label{fig:discrete_case1}
        \end{subfigure}%
        ~
        \begin{subfigure}[t]{0.4\textwidth}
        \centering
            \begin{tikzpicture}
                \draw[very thick, Blue] (1.5,1.2) -- (-0.1,1.2) -- (-0.1,-2.2) node[below] {$\Gamma$};
                \fill[color = Blue, fill opacity = 0.3] (-0.1,-2.2) rectangle (1.5,1.2);
                \draw[|-|] (1.6,1) -- (1.6,1.4) node[right, midway] {$\sigma h$};

                \foreach \y in {-2,-1,...,2}{
                    \draw (-1,\y) -- (1.5,\y);
                }

                \foreach \x in {-1,0,1}{
                    \draw (\x,-2.2) -- (\x,2.2);
                }

                \draw (0.8,2.2) -- (1.5,1.5);
                \draw (-0.2,2.2) -- (1.5,0.5);
                \draw (-1,2) -- (1.5,-0.5);
                \draw (-1,1) -- (1.5,-1.5);
                \draw (-1,0) -- (1.2,-2.2);
                \draw (-1,-1) -- (0.2,-2.2);
                \draw (-1,-2) -- (-0.8,-2.2);

                \draw[very thick, color = Orange] (0,-2.2) -- (0,1) -- (1.5,1) node[below right] {$\p \Omega'$};
                \fill[pattern = north east lines, pattern color = Orange, fill opacity = 0.5] (0,-2.2) rectangle (1.5,1);

            \end{tikzpicture}
            \caption{Case 2: The two edges have distance less that $\sigma h$ to the edges. We replace $\Omega$ with $\Omega'$, which is highlighted in orange (hashed).}
            \label{fig:discrete_case2}
        \end{subfigure}
        \caption{Illustration for cases 1, 2 in step 2. The interior of $\Omega$ is highlighted in blue  with its boundary segment $\Gamma$ parallel to the edges of the triangles.}
    \end{figure}
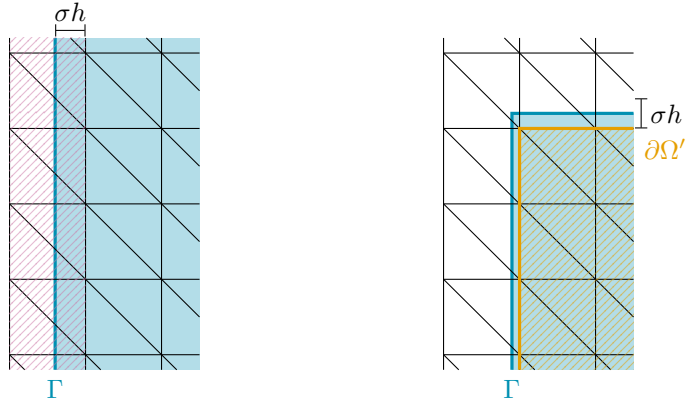

    \textbf{Case 1:} Assume that there exists an edge $\Gamma$ of $\Omega$ whose distance (towards the interior of $\Omega$, i.e.\ oriented in direction $-n_\Gamma$) from the lattice lines orthogonal to $n_\Gamma$ is larger than $\sigma h$, see \cref{fig:discrete_case1}.
    In formulas, we assumed that $d_h(\Gamma)\ge\sigma h$, where
    \begin{equation*}\label{eq:distance-latt}
    d_h(\Gamma):=\min\{t\ge0 : \Gamma-t n_\Gamma\subset\mathbb{L}_h(n_\Gamma)\}, \quad
    \mathbb{L}_h(n_\Gamma):=\begin{cases}
    h\Z\times\R & n_\Gamma=\pm e_1,\\
    \R\times h\Z & n_\Gamma=\pm e_2,\\
    \{(t,z-t):t\in\R,z\in h\Z\} & n_\Gamma=\pm v.
    \end{cases}
    \end{equation*}
    In this case
    \begin{align*}
    \#\{\tau\in\mathcal{I}_1 : \bar{\tau}\cap\Gamma\neq\emptyset\} \ge \frac{1}{2}\#\{\tau\in\mathcal{I} : \bar{\tau}\cap\Gamma\neq\emptyset\}.
    \end{align*}
    Since $u\in\mathcal{A}^{p,\Id}_{h,F}$ and $\nabla u\equiv F$ outside $\Omega$, for every $\tau\in\mathcal{I}_1$ we have that $u(x) = Fx$ on $\tau$ and hence, using $F \notin \K$,
    \begin{align*}
    \int_{\tau\cap\Omega}|\nabla u-\chi|^p dx = \int_{\tau \cap \Omega} |F- \chi|^p dx \geq \sigma C' h^2.
    \end{align*}
    Consequently, by \cref{eq:C0} and the estimates above, we get
    \begin{equation}\label{eq:step3}
        \int_{\Omega} |\nabla u - \chi|^p dx \geq \sum_{\tau \in \mathcal{I}_1} \int_{\tau \cap \Omega} |\nabla u - \chi|^p dx \geq \sum_{\tau \in \mathcal{I}_1} \sigma C' h^2 \geq \frac{1}{2}C' C_0 \sigma h \mathcal{H}^1(\Gamma).
    \end{equation}

    \textbf{Case 2:} If case 1 does not take place, the edges of $\Omega$ have distance (towards the interior) from the lattice lines (parallel to that edge) smaller than $\sigma h$, i.e.\ $d_h(\Gamma)<\sigma h$.
    If there is one edge $\Gamma$ for which this distance is not zero, we define $\Omega'\subset\Omega$ as the polygon obtained by replacing $\Gamma$ with the closest lattice line orthogonal to $n_\Gamma$.
    Notice that, $u(x) = Fx$ also outside of $\Omega'$.
    Up to repeating this procedure a finite number of times, we obtain that all the edges of $\Omega'$ lie on lattice lines, depicted in \cref{fig:discrete_case2}.
    With a slight abuse of notation, we still denote $\Omega'$ as $\Omega$, in particular we have that $\mathcal{I}=\mathcal{I}_2$.

    In this case, due to the anisotropy, there might be an energy-free phase transition between $F$ and $\K$ on the boundary triangles.
    We separate the cases in which this phase transition does not happen (case 2(a)) with the one in which the transition takes place, case 2(b).
    In case 2(b) we need global arguments to prove that the lower bound is non-degenerate.

    For $w \in \{-e_1,-e_2,v\}$ we introduce the set
    \begin{align*}
        \mathcal{I}_2^w := \{ \tau \in \mathcal{I}_2: E_w(\tau) \subset \p \Omega\},
    \end{align*}
    where by $E_w(\tau)$ we denote the edge of $\tau$ with normal $w$.
    These are the sets of lattice triangle touching the edges of $\Omega$ that are orthogonal to $w$.

    \textbf{Case 2(a):}
    First let us assume that there is one edge $\Gamma$ of $\Omega$ with normal $w \in \{-e_1,-e_2,v\}$ such that in the majority of triangles at this edge $\chi$ is in an incompatible well to the boundary condition $F$, i.e.
    \begin{align*}
        \# \{\tau \in \mathcal{I}_2^w: \overline{\tau} \cap \Gamma \neq \emptyset, w \times[\chi - F] \neq 0 \text{ on } \tau\} \geq \frac{C_0}{2h} \mathcal{H}^1(\Gamma).
    \end{align*}
    For any such triangle $\tau$ it then holds
    \begin{align*}
        \int_{\tau \cap \Omega} |\nabla u - \chi|^p dx \geq \int_\tau |w \times (F-\chi)|^p dx \geq C' h^2.
    \end{align*}
    Here we have used that, as there is a line segment in $\p \tau \cap \Gamma$ perpendicular to $w$, by the boundary condition and the fact that $\nabla u$ is constant on $\tau$, we have
    \begin{align*}
        w \times \nabla u = w \times F.
    \end{align*}
    This yields that
    \begin{align} \label{eq:step2a}
        \int_\Omega |\nabla u- \chi|^p dx \geq \frac{1}{2}C' C_0 h \mathcal{H}^1(\Gamma).
    \end{align}
    This case can either happen, if all the wells are incompatible to $F$ in direction $w$, or if $\chi$ is ``badly'' chosen.

    \textbf{Case 2(b):} We are then left to consider the case in which for every edge of $\Omega$ (orthogonal to $w$) there exists (at least) one $A_w$ compatible to $F$ in direction $w$, i.e. $w \times (A_w-F) = 0$, and such that $\chi=A_w$ in the majority of the boundary triangles.
    Hence, there exist at least two (different) vectors $w_1,w_2\in\{-e_1,-e_2,v\}$, two edges $\Gamma_1,\Gamma_2\subset\p\Omega$ (orthogonal to $w_1$ and $w_2$ respectively) and two matrices (up to relabelling) $A_1,A_2\in\K$ such that $w_1\times(F-A_1)=0$, $w_2\times(F-A_2)=0$ and
    \begin{align*}
    &\#\{\tau \in \mathcal{I}_{2}^{w_1} : \bar\tau\cap\Gamma_1\neq\emptyset, \chi=A_{1} \text{ on } \tau\} \ge \frac{C_0}{2hN}\mathcal{H}^1(\Gamma_1), \\
    &\#\{\tau \in \mathcal{I}_{2}^{w_2} : \bar\tau\cap\Gamma_2\neq\emptyset, \chi=A_{2} \text{ on } \tau\} \ge \frac{C_0}{2hN}\mathcal{H}^1(\Gamma_2).
    \end{align*}
    From the fact that $F\notin\K$ we also infer that $A_1\neq A_2$. The factor $\frac{1}{N}$ here accounts for the fact that there might be multiple possible choices of $A_1, A_2$.

    As the edges of $\Omega$ lie on the lattice lines, we can cover the whole $\Omega$ by copies of $\Omega_1^z \cup \Omega_2^z$ and its flipped version $(\Omega_1')^z \cup (\Omega_2')^z$.
    Hence, by reasoning as in \cref{eq:DiscreteToSharp_inner} of Step 1, we have
    \begin{align*}
        \int_\Omega |\nabla u - \chi|^p dx \geq C' h \Vert D_{\nu} \chi \Vert_{TV(\Omega)}.
    \end{align*}
    We now prove that, either there exists a non-degenerate interface oriented in a direction not orthogonal to $v$ or we have a non-vanishing elastic energy contribution in the bulk.

    As $w_1 \neq w_2$, we can assume, without loss of generality, that $w_2=-e_2$.
    Let
    \begin{align*}
        E:=\bigcup_{\tau\in \mathcal{T}_{h} :\ \chi\equiv A_2 \text{ on } \tau} \tau
    \end{align*}
    and let $E'$ be a connected component of $E$ whose closure intersects $\Gamma_2$ and let $\ell:=\mathcal{H}^1(\p E'\cap\Gamma_2)$.
    For simplicity, assume also
    $\p E'\cap\Gamma_2=(0,\ell)\times\{0\}$.

    Denote $\gamma=\p E'\setminus\Gamma_2$, which is a connected piecewise affine path.
    We decompose it into the part inside $\Omega$ and the part on its boundary, that is, we write $\gamma=\gamma_{\rm in}\cup\gamma_{\rm b}:=(\gamma\setminus\p\Omega)\cup(\gamma\cap\p\Omega)$.
    Consider the stripe
    \begin{align*}
    S:=\Big\{(x_1,x_2)\in\Omega : x_2>0, \frac{\ell}{4}<x_1+x_2<\frac{3\ell}{4}\Big\},
    \end{align*}
    and let $f(x_1)$ be the length of the segment in direction $\nu$ originating
    from $(x_1,0)$ and ending on $\p\Omega$ (but not necessarily in $E'$).
    Here, we recall that $v = 2^{-1/2} (1, 1)^T$ and $\nu \cdot v = 0$.
    Notice that, as $\Omega$ is a polygonal domain and $S\cap\Gamma_2$
        is well-contained in $\Gamma_2$, $f(x_1)$ is bounded from below, hence up to reducing the
    value of $\sigma$, $\sigma< f(x_1)< \diam(\Omega)$ for every
    $x_1\in(\ell/4,3\ell/4)$.
    By connectedness of $\gamma$, we notice that
    \begin{align*}
    \int_{\gamma\cap \bar S}|\nu\cdot n_{\p E'}|d\mathcal{H}^1=\int_{\gamma_{\rm in}\cap S}|\nu\cdot n_{\p E'}|d\mathcal{H}^1+\int_{\gamma_{\rm b}\cap \bar S}|\nu\cdot n_{\p E'}|d\mathcal{H}^1\ge\frac{\ell}{2 \sqrt{2}}.
    \end{align*}
    Indeed, as $E'$ is connected $\gamma$ is a continuous path from $(0,0)$ to $(\ell,0)$ hence intersects $\p S$ both in $\{(\ell/4,0) + t \nu: t \in \R\}$ and $\{(3\ell/4,0) + t \nu: t \in \R\}$.
    By the choice of $S$, that is, $\p S \cap \Omega$ has normal $v$, and the
    fact that the segment $\gamma$ of $\p E'$ consists of line segments
    with possible directions $e_1,e_2,$ and $\nu$, $\gamma\cap S$ has to
    move a distance larger than $\ell/2$ in the directions $e_1$
    or $e_2$.

    \begin{figure}
        \centering
        \begin{subfigure}[t]{0.4\textwidth}
            \centering
            \begin{tikzpicture}
                \fill[color = Blue, fill opacity = 0.3] (0,0) -- (0,3) -- (1,2) -- (5.5,2) -- (5.5,0) -- cycle;

                \draw[->] (5,1) -- (5.3,1.3) node [above] {$v$};
                \draw[->] (5,1) -- (4.7,1.3) node [above] {$\nu$};

                \draw (2,0) -- (0,2);
                \draw (4,0) -- (2,2);
                \node at (3,0) [above] {$S$};

                \draw[thick, Purple] (1,0) -- (1,1.2) -- (2,1.2) -- (1.2,2);
                \draw[thick, Purple] (5,0) -- (5,0.5) -- (4.1,0.5) -- (3.5,1.1) -- (3,1.1) -- (2.1,2);
                \draw[very thick, Red, opacity = 0.8] (1,1) -- (1,1.2) -- (2,1.2);
                \node at (1,0) [above right, Purple] {$E'$};
                \draw[|-|] (1,-0.2) -- (5,-0.2) node[below, midway] {$\ell$};

                \draw[very thick, color = Blue] (5.5,2) -- (1,2) -- (0,3) -- (0,0) -- (5.5,0) node[right] {$\p \Omega$};
            \end{tikzpicture}
        \caption{The case $\int_{\gamma_{\rm in} \cap S} |\nu \cdot n_{\p E'}| d \mathcal{H}^1 \geq \ell/(4\sqrt{2})$. The line segment relevant for the energy contribution in \cref{eq:anis-lb-h} is highlighted in orange.}
        \label{fig:cases_levelset_1}
        \end{subfigure}%
        ~
        \begin{subfigure}[t]{0.4\textwidth}
            \centering
            \begin{tikzpicture}
                \fill[color = Blue, fill opacity = 0.3] (0,0) -- (0,3) -- (1,2) -- (5.5,2) -- (5.5,0) -- cycle;

                \draw[->] (5,1) -- (5.3,1.3) node [above] {$v$};
                \draw[->] (5,1) -- (4.7,1.3) node [above] {$\nu$};

                \draw (2,0) -- (0,2);
                \draw (4,0) -- (2,2);
                \node at (3,0) [above] {$S$};
                \draw[dashed, gray] (3.7,0) -- (1.7,2);

                \draw[thick, Purple] (1,0) -- (1,0.6) -- (0.5,1.1) -- (0.5,1.3) -- (0,1.8);
                \draw[thick, Purple] (5,0) -- (4,1) -- (4,1.5) -- (3.5,1.5) -- (3,2);
                \draw[thick, Purple] (0,2.8) -- (1,1.8) -- (1.3,1.8) -- (1.1,2);
                \node at (1,0) [above right, Purple] {$E'$};
                \draw[|-|] (1,-0.2) -- (5,-0.2) node[below, midway] {$\ell$};

                \draw[very thick, color = Blue] (5.5,2) -- (1,2) -- (0,3) -- (0,0) -- (5.5,0) node[right] {$\p \Omega$};

                \draw[very thick, color = Red] (2,0) -- (2.8,0) node[midway, above] {$I$};
                \draw[very thick, color = Red] (3.1,0) -- (4,0);
            \end{tikzpicture}
            \caption{The case $\int_{\gamma_{\rm in} \cap S} |\nu \cdot n_{\p E'}| d \mathcal{H}^1 < \ell/4$. The line segment of length $f(x_1)$ in $S$ with direction $\nu$ is shown as the dashed line. The set $I$ for which we consider these line segments is highlighted in orange.}
        \label{fig:cases_levelset_2}
        \end{subfigure}
        \caption{The two cases for the sets $S$ and $E'$ in case 2(b).}
        \label{fig:cases_levelset}
    \end{figure}
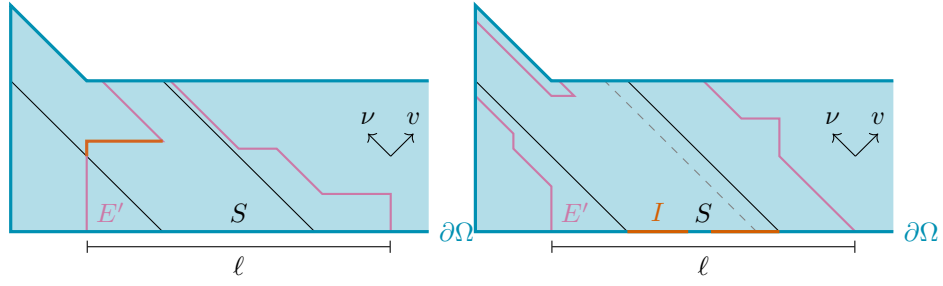

    Since we have a control only of the interfaces inside $\Omega$, if
    \begin{align*}
    \int_{\gamma_{\rm in}\cap S}|\nu\cdot n_{\p E'}|d\mathcal{H}^1\ge\frac{\ell}{4 \sqrt{2}}
    \end{align*}
    then we infer that
    \begin{equation}\label{eq:anis-lb-h}
        \|D_{\nu}\chi\|_{TV(\Omega)}\geq C \int_{\gamma_{\rm in} \cap S} |\nu \cdot n_{\p E'}| d \mathcal{H}^2 \geq C_1\ell,
    \end{equation}
    with $C_1>0$ depending on $\K$.

    Consider now the opposite case, i.e.\ that
    \begin{align*}
    \int_{\gamma_{\rm in}\cap S}|\nu\cdot n_{\p E'}|d\mathcal{H}^1\le\frac{\ell}{4 \sqrt{2}} \leq \frac{\ell}{4}.
    \end{align*}
    In this case, loosely speaking, we can prove that there is a union of stripes (contained in $S$ and oriented in its same direction) fully contained in $E'$, in which we can perform an integration argument which provides a lower bound (see below).
    To be precise, let $I:=\{x_1\in(\ell/4,3\ell/4):((x_1,0)+\R\nu)\cap\gamma_{\rm in} = \emptyset\}$ be the set of starting points on $\p S \cap \Gamma_2$ such that the ray in direction $\nu$ does not intersect $\gamma_{\rm in}$.
    Then, with a slicing argument
    \begin{align*}
        \frac{\ell}{4} &\ge \int_{\gamma_{\rm in}\cap S}|\nu \cdot n_{\p E'}|d\mathcal{H}^1 \\
        &= \int_\frac{\ell}{4}^\frac{3\ell}{4}\mathcal{H}^0(\{x_2\in(0,f(x_1)):(x_1-x_2,x_2)\in\gamma_{\rm in}, n_{\p E'}(x_1-x_2,x_2)\neq\nu\})dx_1 \ge \frac{\ell}{2}-|I|.
    \end{align*}
    Hence, after a suitable change of variables and an application of Jensen's inequality and the fundamental theorem of calculus, we get
    \begin{align*}
    \int_{E'}|\nabla u-\chi|^pdx & \ge \int_{I}\int_0^{f(x_1)}|\p_{\nu} \tilde{u}(x_1-x_2,x_2)-A_2 \nu|^p dx_2 dx_1 \\
    &\ge \sigma\int_{I}\Big|\frac{1}{f(x_1)}\int_0^{f(x_1)}(\p_{\nu} u(x_1-x_2,x_2)-A_2 \nu) dx_2\Big|^p dx_1 \\
    &= \sigma\int_{I}\Big|\frac{u(x_1-f(x_1),f(x_1))-u(x_1,0)}{\sqrt{2}\ f(x_1)}-A_2\nu\Big|^p dx_1.
    \end{align*}
    Since, by the boundary conditions, $u(x_1-f(x_1),f(x_1))=F(x_1-f(x_1),f(x_1))^T$, $u(x_1,0)=F(x_1,0)^T$, and $F-A_2=a\otimes e_2$ for some $a\in\R^2\setminus\{0\}$ we infer that
    \begin{equation}\label{eq:anis-lb-h2}
    \int_{E'}|\nabla u-\chi|^pdx \ge \sigma\frac{\ell}{4} |(a\otimes e_2) \nu|^p\ge \sigma C_1\ell,
    \end{equation}
    up to reducing the value of $C_1$ if needed.
    Gathering \cref{eq:anis-lb-h,eq:anis-lb-h2} and repeating it for all the connected components of $E$ we infer that also in this case
    \begin{equation}\label{eq:step3c}
    \int_\Omega|\nabla u-\chi|^pdx \ge C_1\min\{h,\sigma\}\frac{C_0}{2}\mathcal{H}^1(\Gamma_2).
    \end{equation}

    \emph{Step 2.2: Domains not aligned with the grid.}
    Consider now a general polygonal domain not admissible in Step 2.1, namely a polygon with an edge $\Gamma$ such that
    \begin{equation}\label{eq:misalign}
    \min\big\{|n_{\Gamma}-w| : w\in\{\pm e_1,\pm e_2, \pm v\}\big\}\ge\sigma_0,
    \end{equation}
    for some $\sigma_0 \in (0,\frac{1}{2})$ depending on the domain, see \cref{fig:misalign_boundary}.

    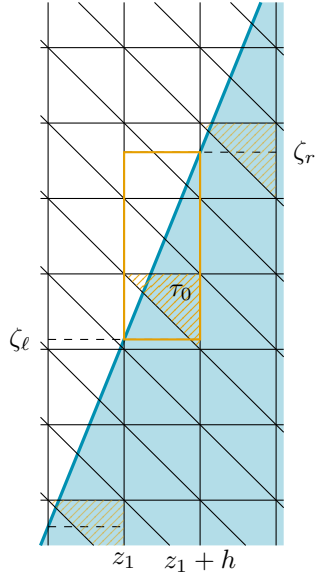
\begin{figure}
        \centering
        \begin{tikzpicture}
            \draw[very thick, Blue] (-0.1,-0.6) -- (2.8,6.6);
            \fill[Blue, opacity = 0.3] (-0.1,-0.6) -- (2.8,6.6) -- (3.1,6.6) -- (3.1,-0.6) -- cycle;

            \foreach \y in {0,1,2,...,6}{
                    \draw (-0.1,\y) -- (3.1,\y);
                }

            \foreach \x in {0,1,2,3}{
                    \draw (\x,-0.6) -- (\x,6.6);
                }
            \node at (1,-0.8) {$z_1$};
            \node at (2,-0.8) {$z_1 + h$};

            \draw (2.4,6.6) -- (3.1,5.9);
            \draw (1.4,6.6) -- (3.1,4.9);
            \draw (0.4,6.6) -- (3.1,3.9);
            \foreach \y in {3,4,5,6}{
                \draw (-0.1,\y+0.1) -- (3.1,\y-3.1);
            }
            \draw (-0.1,2.1) -- (2.6,-0.6);
            \draw (-0.1,1.1) -- (1.6,-0.6);
            \draw (-0.1,0.1) -- (0.6,-0.6);

            \draw[dashed] (0,-0.6+7.2/29) -- (1,-0.6+7.2/29);
            \draw[dashed] (0,-0.6+79.2/29) -- (2,-0.6+79.2/29);
            \draw[dashed] (1,-0.6+151.2/29) -- (3,-0.6+151.2/29);

            \draw[thick, Orange] (1,-0.6+79.2/29) rectangle (2,-0.6+151.2/29);
            \node [left] at (-0.1,-0.6+79.2/29) {$\zeta_\ell$};
            \node [right] at (3.1,-0.6+151.2/29) {$\zeta_r$};

            \fill[pattern = north east lines, pattern color = Orange, fill opacity = 0.9] (2,2) -- (2,3) -- (1,3) -- cycle;
            \fill[pattern = north east lines, pattern color = Orange, fill opacity = 0.5] (0,0) -- (1,0) -- (1,-0.6) -- (0.6,-0.6) -- cycle;
            \fill[pattern = north east lines, pattern color = Orange, fill opacity = 0.5] (2,5) -- (3,5) -- (3,4) -- cycle;

            \node[] at (1.75,2.75) {$\tau_0$};
        \end{tikzpicture}
        \caption{In the case of an edge $\Gamma$ misaligned with the triangulations, there are always ``many'' triangles with a majority of the area inside $\Omega$. The rectangle $E$ is highlighted in orange, with the part inside $\Omega$ in blue. The triangle $\tau_0 \in \mathcal{I}_\Gamma \cap \mathcal{I}_1$ is hashed in orange. For the adjacent vertical segments the corresponding triangles in $\mathcal{I}_\Gamma \cap \mathcal{I}_1$ are also hashed in orange.}
        \label{fig:misalign_boundary}
    \end{figure}
    Due to the misalignment with the grid there are no boundary triangles around $\Gamma$ in the set $\mathcal{I}_2$, namely, denoting $\mathcal{I}_\Gamma:=\{\tau\in\mathcal{I} : \bar\tau\cap\Gamma\neq\emptyset\}$, we have
    \begin{align*}
        \mathcal{I}_\Gamma=(\mathcal{I}_\Gamma\cap\mathcal{I}_1)\cup(\mathcal{I}_\Gamma\cap\mathcal{I}_{\rm small}).
    \end{align*}
    We now show that, for every $h<h_0$, $\#(\mathcal{I}_\Gamma\cap\mathcal{I}_1)\ge C_0 h^{-1}\mathcal{H}^1(\Gamma)$, which then leads to the desired lower bound by arguing as in Case 1 of Step 3.
    By \cref{eq:misalign}, consider the case $0<|(n_\Gamma)_2|\le|(n_\Gamma)_1|\leq 1-\frac{\sigma_0^2}{2}$ (the case with switched coordinates works analogously).
    Loosely speaking, $\Gamma$ has ``slope'' larger then 1 but is not vertical.
    Hence, for $h$ sufficiently small, it intersects $h^{-1}$ many vertical lattice lines, and (as proved in the following) for each line intersected there is a non-degenerate boundary triangle.

    Indeed, let $z_1\in h\Z$ such that $\Gamma\cap(\{z_1\}\times\R),\Gamma\cap(\{z_1+h\}\times\R)\neq\emptyset$ and consider the rectangle $E$ with vertices $\{(z_1,\zeta_\ell)\}:=\Gamma\cap(\{z_1\}\times\R)$ and $\{(z_1,\zeta_r)\}:=\Gamma\cap(\{z_1+h\}\times\R)$.
    Note that, as $|\zeta_\ell-\zeta_r|=h\big|\frac{(n_\Gamma)_1}{(n_\Gamma)_2}\big|$, there are at most $2\big\lfloor\big|\frac{(n_\Gamma)_1}{(n_\Gamma)_2}\big|\big\rfloor+4$ triangles intersecting $E$ which all belong to $\mathcal{I}_\Gamma$.
    Now we have
    \begin{align*}
    \Big|\frac{(n_\Gamma)_1}{(n_\Gamma)_2}\Big|\frac{h^2}{2}=|E\cap\Omega|=\sum_{\substack{\tau\in\mathcal{I}_\Gamma \\ \tau\cap E\neq\emptyset}}|\tau\cap\Omega| \le |\tau_0\cap\Omega|\Big(2\Big\lfloor\Big|\frac{(n_\Gamma)_1}{(n_\Gamma)_2}\Big|\Big\rfloor+4\Big)
    \end{align*}
    where $\tau_0\in\mathcal{I}_\Gamma$ is such that $|\tau_0\cap\Omega|\ge|\tau\cap\Omega|$ for every $\tau\in\mathcal{I}_\Gamma$ with $\tau\cap E\neq\emptyset$.
    Since $|(n_\Gamma)_2|\le|(n_\Gamma)_1|$, the previous inequality yields that $|\tau_0\cap \Omega|\geq\frac{h^2}{16}$.
    By $|(n_\Gamma)_1|\le 1-\frac{\sigma_0^2}{2}$, $\Gamma$ intersect $\cup_{z \in h \Z} \{z\} \times \R$ at least
    \begin{align*}
    \big\lfloor h^{-1}|(n_\Gamma)_2|\mathcal{H}^1(\Gamma)\big\rfloor\ge\Big\lfloor \frac{\sigma_0^2}{2h}\mathcal{H}^1(\Gamma)\Big\rfloor
    \end{align*}
    many times.
    Thus, up to reducing the value of $\sigma$ we have $\tau_0 \in \mathcal{I}_\Gamma \cap \mathcal{I}_1$, and repeating the argument above, we find $\frac{\sigma^2}{4}h^{-1}\mathcal{H}^1(\Gamma)$ many triangles in $\mathcal{I}_\Gamma\cap\mathcal{I}_1$.
    Eventually, arguing as in \cref{eq:step3} we infer
    \begin{equation}\label{eq:step4}
    \int_\Omega|\nabla u-\chi|^p dx\ge C' C_0 \sigma h \mathcal{H}^1(\Gamma),
    \end{equation}
    here we possibly reduced the value of $C_0$ (e.g.\ by multiplying it with a universal constant) if needed.

    \emph{Step 3: Conclusion.}
    Gathering \cref{eq:DiscreteToSharp_inner,eq:step3,eq:step2a,eq:step3c,eq:step4} we get
    \begin{align}\label{eq:final-discrete}
        \int_\Omega |\nabla u - \chi|^p dx & \geq \frac{1}{3} \int_\Omega |\nabla u - \chi|^p dx + \frac{1}{3} C h \Vert D_{R\nu} \chi \Vert_{TV(\Omega_h)} + \frac{1}{3} \sigma C'' h
    \end{align}
    for suitable constants $C'',C>0$ depending on $\Omega$, $R$, $p$, $v$, $F$, and $\K$.

    Furthermore, since
    \begin{align*}
        \Vert D_{R\nu} \chi \Vert_{TV(\Omega \setminus \Omega_h)} \leq C(\K,d,\Omega) h^{-1} \Per(\tau) \leq C(\K,d,\Omega),
    \end{align*}
    we get that
    \begin{align*}
    \|D_{R\nu}\chi\|_{TV(\Omega)}+1\le C(\|D_{R\nu}\chi\|_{TV(\Omega_h)}+1)
    \end{align*}
    which then gives the result together with \cref{eq:final-discrete}.
\end{proof}

With this we can deduce lower scaling bounds for a discrete model by using the corresponding lower bounds for the sharp interface model, provided we do not argue on scales smaller than $h$. We will comment on this in more detail below.

\begin{rmk}
    The proof of \cref{cor:DiscreteToSharp} in the isotropic setting follows the same idea as for the anisotropic setting, hence we omit it at this point.
\end{rmk}

We expect the same result to hold in higher dimensions, where the triangulation has to be replaced by the corresponding higher-dimensional generalization.

\subsection{Applications}

In this last section we apply the result of \cref{cor:DiscreteToSharp} to derive scaling laws for the discretized Lorent and Tartar settings.

\subsubsection{The discrete Lorent three-well problem}

Consider $\K_3 \subset \R^{2 \times 2}$ given in \cref{eq:Lorent_wells}. For the (anisotropic, continuous) sharp interface model the scaling is given by \cref{thm:L1_3wells}. By \cref{cor:DiscreteToSharp} we expect the scaling law to carry over to the discretized model as soon as the rank-one connection in $\K_3$ is misaligned with the triangulation $\mathcal{T}_h^R$.
Building on these observations, we present the proof of \cref{cor:Lorent_discrte}.

\begin{proof}[Proof of the lower bounds in \cref{cor:Lorent_discrte}]
We recall that the only rank-one connection within the wells forming the set $\K_3$ is between $A_1$ and $A_2$ and is given by the vector $e_1$.

    In \cref{itm:Lorent_discrete_i} we have, $e_1 \not \in \{\pm R e_1, \pm R e_2, \pm 2^{-1/2} R(e_1+e_2)\}$.
    Hence, since $e_1$ is the only available rank-one direction,
    \begin{align*}
        (Re_1) \times [A_j-A_k] \neq 0, \ (Re_2) \times [A_j - A_k] \neq 0, \ (R \begin{pmatrix} 1 \\ 1 \end{pmatrix}) \times [A_j-A_k] \neq 0,
    \end{align*}
    for all $A_j, A_k \in \K_3$, $j \neq k$.
    As a consequence, by \cref{cor:DiscreteToSharp} we have
    \begin{align*}
        \int_\Omega |\nabla u - \chi|^2 dx \geq C \left( \int_\Omega |\nabla u - \chi|^2 dx + h \Vert D \chi \Vert_{TV(\Omega)} \right).
    \end{align*}
    The statement then follows by an application of \cref{thm:L1_3wells}. 

    In the second case, that is in the case that the compatible direction is in one of the ``bad'' directions, we use the anisotropic version of \cref{cor:DiscreteToSharp}. Thus, there is $v \in \{e_1, e_2, 2^{-1/2}(e_1 + e_2)\}$ such that $Rv = \pm e_1$ and, hence, $Rv \times [A_1 - A_2] = 0$.
    By \cref{cor:DiscreteToSharp}, we get with $\nu \in \S^1, \nu \cdot v = 0$
    \begin{align*}
        \int_\Omega |\nabla u - \chi|^2 dx \geq C \left( \int_\Omega |\nabla u - \chi|^2 dx + h \Vert D_{R \nu} \chi \Vert_{TV(\Omega)}+h \right).
    \end{align*}
    As $\nu \cdot v = R\nu \cdot Rv = \pm R\nu \cdot e_1 = 0$ it holds, without loss of generality, that $R \nu = e_2$. As a consequence, we infer that
    \begin{align*}
        \int_\Omega |\nabla u - \chi|^2 dx \geq C \left( \int_\Omega |\nabla u - \chi|^2 dx + h \Vert D_{e_2} \chi \Vert_{TV(\Omega)}+h \right).
    \end{align*}
    The desired lower bound then follows by \cref{thm:L1_3wells}.
\end{proof}

We can use our upper bound constructions from \cref{sec:L1AnisotropicSurf} and interpolate them on the triangulation to get matching upper bounds.

\begin{proof}[Proof of upper bounds in \cref{cor:Lorent_discrte}]
    We follow the arguments of \cite{L09,C99,CM99}.
    Let us start by considering the isotropic setting.

    Taking the corresponding upper bound construction $u, \chi$ from \cref{thm:L1_3wells} with $\epsilon=h$,
    we define $u_h \in \mathcal{A}_{h,F}^{2,R}$ as the piecewise affine approximation of $u$, preserving the boundary condition.
    For this it holds
    \begin{align*}
        \int_\Omega \dist^2(\nabla u_h,\K_3) dx \leq \int_\Omega |\nabla u_h - \nabla u|^2 dx + E_{el}(u,\chi).
    \end{align*}
    As $\nabla u$ is bounded for each of our constructions (see \cref{eq:PropertiesBranching}), we get
    \begin{align*}
        \int_\Omega \dist^2(\nabla u_h,\K_3) dx \leq C(\K_3) |\{u_h \neq u\}| + E_{el}(u,\chi).
    \end{align*}
    We now seek to bound the volume $|\{u_h \neq u\}|$.
    For this we note that, firstly, $u$ is piecewise affine in our constructions (but not on the triangulation), and, secondly, that when we are at least at distance $2h$ from the jump set of $\nabla u$ and $\p \Omega$, then it holds $u_h = u$.
    By this we conclude
    \begin{align*}
        |\{u_h \neq u\}| \leq C h ( \Vert D(\nabla u) \Vert_{TV(\Omega)} +\Per(\Omega) ).
    \end{align*}
    Moreover, by construction of the continuous upper bound deformation we also have, see \cref{eq:PropertiesBranching} and the proof of \cref{thm:L1_3wells},
    \begin{align*}
        \Vert D(\nabla u) \Vert_{TV(\Omega)} \leq C ( \Vert D \chi \Vert_{TV(\Omega)} + \Per(\Omega)).
    \end{align*}

    In conclusion this yields
    \begin{align*}
        \int_\Omega \dist^2(\nabla u_h,\K_3) dx \leq C(E_{el}(u,\chi) + h \Vert D \chi\Vert_{TV(\Omega)} + h \Per(\Omega)).
    \end{align*}
    Thus, after choosing $\chi_h$ pointwise as the projection of $\nabla u_h$ onto $\K_3$ (with an arbitrary choice where the projection is not well-defined) the upper bound follows also in the discrete set-up.

    Compared to the above isotropic argument, additional care is needed for the anisotropic setting.
    If $F \in \K_3^1$ in the anisotropic case, the oscillation of the upper bound construction is finer than $h$. Thus, an unmodified variant of the above argument would yield the trivial bound
    \begin{align*}
        \int_\Omega \dist^2(\nabla u_h,\K_3) dx \leq C |\Omega|.
    \end{align*}
    To achieve the desired upper bound which displays an $h$ scaling behaviour, we consider a simple laminate of $A_1$ and $A_2$ on a scale $h$ with a suitable cut-off. Then, in the bulk of $\Omega$ we have $\nabla u_h \in \K_3$ due to the compatibility and as we can laminate exactly on the boundaries of triangles.
    Therefore, the energy concentrates in a boundary region, and we obtain
    \begin{align*}
        \int_\Omega \dist^2(\nabla u_h,\chi) dx \leq C |\{\dist(x,\p\Omega) \leq 2 h\}| \leq C h.
    \end{align*}
    Again choosing $\chi_h$ as the pointwise projection of $\nabla u_h$ onto $\K_3$ yields the desired result.

    For $F \in \K_3^2$ in the anisotropic case, we remark that the rescaled unit cells always have lengths and heights such that $\ell_j, h_j \gg \epsilon$. In this case, we can therefore directly translate the continuous upper bound construction (laminate within branching) by discretization. More precisely, we fix the laminate on the scale $\epsilon = h$ and apply analogous arguments as in the isotropic setting above which allow us to transfer the bound from the continuum to the discrete framework.
\end{proof}

\subsubsection{The discrete Tartar square}
Last but not least, we turn to the proof of the bounds for the Tartar square.

\begin{proof}[Proof of \cref{cor:Tartar_discrete}]
Recalling that there are no rank-one connections in the Tartar square, the proof follows from applying the isotropic version of \cref{cor:DiscreteToSharp} and invoking the lower scaling result from \cite{RT22}.
\end{proof}

The almost matching upper bound can be found in \cite{C99}.

\subsection*{Acknowledgements}
A.R. and C.T. gratefully acknowledge support by the Deutsche Forschungsgemeinschaft (DFG, German Research Foundation)
under Germany’s Excellence Strategy -- EXC-2047/1.
A.R. and A.T. gratefully acknowledge support by the Deutsche Forschungsgemeinschaft (DFG, German Research Foundation) through SPP 2256, project ID 441068247.
C.Z. gratefully acknowledges support by the Deutsche Forschungsgemeinschaft (DFG, German Research Foundation) -- Project-ID 258734477-SFB 1173.

\bibliographystyle{alpha}
\bibliography{citations1}

\end{document}